\documentclass[a4paper,10pt,reqno]{amsart}

\usepackage{hyperref}

\usepackage{a4wide,color,array}

\usepackage{enumitem}
\usepackage[utf8]{inputenc} 

\usepackage{amsaddr}
\usepackage{amsmath,amssymb,amsthm}
\usepackage{booktabs}
\usepackage{adjustbox}
\usepackage{multirow}
\usepackage{enumitem}

\usepackage[french,english]{babel}  
\usepackage{lmodern}             
\usepackage{mathtools}
\numberwithin{equation}{section}
\usepackage{subcaption}
\allowdisplaybreaks
\usepackage{color}
\definecolor{darkred}{RGB}{100,0,0}
\definecolor{darkgreen}{RGB}{0,100,0}
\definecolor{darkblue}{RGB}{0,0,150}
\definecolor{citecol}{RGB}{30,80,150}
\definecolor{tabcol}{RGB}{200,230,255}

\usepackage{colortbl}

\hypersetup{colorlinks=true, linkcolor=darkred, citecolor=citecol, urlcolor=darkblue}
\usepackage{url}

\numberwithin{equation}{section}
\newtheorem{theorem}{Theorem}[section]

\newtheorem{proposition}[theorem]{Proposition}
\newtheorem{lemma}[theorem]{Lemma}
\newtheorem{definition}[theorem]{Definition}

\theoremstyle{remark}

\newtheorem{remark}[theorem]{Remark}

\newcommand{\E}{\mathbb{E}}

\newcommand{\N}{\mathbb{N}}

\renewcommand{\P}{\mathbb{P}}

\newcommand{\R}{\mathbb{R}}

\newcommand{\V}{\mathbb{V}}
\newcommand{\W}{\mathbb{W}}
\newcommand{\X}{\mathbb{X}}
\newcommand{\Y}{\mathbb{Y}}

\newcommand{\BB}{\mathcal{B}}
\newcommand{\CC}{\mathcal{C}}
\newcommand{\DD}{\mathcal{D}}

\newcommand{\GG}{\mathcal{G}}

\newcommand{\LL}{\mathcal{L}}
\newcommand{\MM}{\mathcal{M}}

\newcommand{\OO}{\mathcal{O}}
\newcommand{\PP}{\mathcal{P}}
\newcommand{\QQ}{\mathcal{Q}}
\newcommand{\RR}{\mathcal{R}}
\renewcommand{\SS}{\mathcal{S}}
\newcommand{\TT}{\mathcal{T}}

\newcommand{\eps}{\varepsilon}

\newcommand{\p}[1]{\left(#1 \right)}

\renewcommand{\t}[1]{\tilde{#1}}

\newcommand{\defeq}{\vcentcolon=}
\newcommand{\eqdef}{=\vcentcolon}
\newcommand{\op}[1]{\left\| #1  \right\|_{\mathrm{op}}}
\newcommand{\dotp}[1]{\langle #1 \rangle}
\newcommand{\ones}{\mathbf{1}}
\newcommand{\dd}{\mathrm{d}}

\DeclareMathOperator*{\Leb}{Leb}

\newcommand{\vol}{\mathrm{vol}}
\DeclareMathOperator*{\id}{id}
\DeclareMathOperator*{\TV}{TV}

\DeclareMathOperator*{\diam}{diam}

\DeclareMathOperator*{\Var}{Var}
\DeclareMathOperator*{\argmin}{arg\,min}

\newcommand{\taumin}{\tau_{\min}}
\newcommand{\fmin}{f_{\min}}
\newcommand{\fmax}{f_{\max}}

\begin{document}

\title[]{Measure estimation on manifolds: an optimal transport approach}

\author{Vincent~Divol}
\address{Courant Institute of Mathematical Science, New York University}
\address{Center for Data Science, New York University}
\curraddr{}
\email{firstname.lastname@nyu.edu}
\thanks{}

\begin{abstract}
Assume that we observe i.i.d.~points lying close to some unknown $d$-dimensional $\CC^k$ submanifold $M$ in a possibly high-dimensional space. We study the problem of reconstructing the probability distribution generating the sample. After remarking that this problem is degenerate for a large class of standard losses ($L_p$, Hellinger, total variation, etc.), we focus on the Wasserstein loss, for which we build an estimator, based on kernel density estimation, whose rate of convergence depends on $d$ and the regularity $s\leq k-1$ of the underlying density, but not on the ambient dimension. In particular, we show that the estimator is minimax and matches previous rates in the literature in the case where the manifold $M$ is a $d$-dimensional cube.  The related problem of the estimation of the volume measure of $M$ for the Wasserstein loss is also considered, for which a minimax estimator is exhibited.
\end{abstract}

\maketitle

\section{Introduction}

Density estimation is one of the most fundamental tasks in non-parametric statistics. If efficient methods (from both a theoretical and a practical point of view) exist when the ambient space is of low dimension, minimax rates of estimation become increasingly slow as the dimension increases. To overcome this so-called \emph{curse of dimensionality}, some structural assumptions on the underlying probability are to be made in moderate to high dimensions, which may take different forms, including e.g.~the existence of a parametric component \cite{liu2007sparse}, the single-index model \cite{liu2013robust}, sparsity assumptions \cite{tibshirani1996regression},  or constraints on the shape of the support. We focus in this work on the latter, namely on the case where the probability distribution $\mu$ generating the observations is assumed to be concentrated around a submanifold $M$ of $\R^D$, of dimension $d$ smaller than $D$. This assumption, known as the manifold assumption, has been fruitfully studied, with an emphasis put on reconstructing different geometric quantities related to the manifold, such as $M$ itself \cite{genovese2012minimax,  aamari2018stability,aamari2019nonasymptotic,divol2020minimax}, its homology groups \cite{niyogi2008finding, balakrishnan2012minimax}, its dimension \cite{hein2005intrinsic,little2009multiscale, kim2019minimax} or its reach \cite{aamari2019estimating, berenfeld2020estimating}. The topic of density estimation in the manifold setting has itself been studied for over thirty years, with the emphasis initially being put on reconstructing the density in the case where the manifold $M$ is given---think for instance of datasets lying on the space of orthogonal matrices---notable works including \cite{hendriks1990nonparametric,hendriks1993strong, pelletier2005kernel, cleanthous2020kernel}.  Less attention has been dedicated to the more general setting where the manifold $M$ is unknown and acts as a nuisance parameter. Kernel density estimators on manifolds are designed in \cite{berry2017density, wu2020strong}, where rates are exhibited, respectively in the case where the manifold has a boundary and in the case where the density is H\"older continuous. In \cite{berenfeld2019density}, kernel density estimators are shown to be minimax, and an adaptive procedure is designed, based on Lepski's method, to estimate the unknown density in a point $x\in \R^D$ which is known to belong to the unknown (and possibly nonsmooth) manifold $M$.  

To go beyond the pointwise estimation of $\mu$, even the choice of a relevant loss is nontrivial. Indeed, most standard losses between probability measures (e.g.~the $L_p$ distance, the Hellinger distance or the Kullback-Leibler divergence) are degenerate when comparing mutually singular measures, which will typically be the case for measures on two distinct manifolds, even if they are very close to each other with respect to the Hausdorff distance. This implies that the estimation problem is degenerate from a minimax perspective when choosing such losses (see Theorem \ref{thm:choice_of_loss}). On the contrary, the Wasserstein distances $W_p$, $1\leq p \leq \infty$ are particularly adapted to this problem, as they are by design robust to small metric perturbations of the support of a measure.

Apart from this first motivation, the use of Wasserstein distances, and more generally of the theory of optimal transport, has shown to be an efficient tool in widely different recent problems of machine learning, with fast implementations and sound theoretical results (see e.g.~\cite{peyre2019computational} for a survey). From a statistical perspective, most of the attention has been dedicated to studying rates of convergence between a probability distribution $\mu$ and its empirical counterpart $\mu_n$ \cite{dudley1969speed,dereich2013constructive,fournier2015rate,singh2018minimax,weed2019sharp,lei2020convergence}. Unsurprisingly, if more regularity is assumed on $\mu$, then it is possible to build estimators with smaller risks than the empirical measure $\mu_n$. Assume for instance that $\mu$ is a probability distribution on the cube $[-1,1]^D$, with density $f$ of regularity $s$ (measured through the Besov scale $B^s_{p,q}$). In this setting, it has been shown in \cite{weed2019estimation} that, given $n$ i.i.d. points of law $\mu$, the minimax rate (up to logarithmic factors) for the estimation of $\mu$ with respect to the Wasserstein distance $W_p$ is of order 
\begin{equation}\label{eq:weed}
\begin{cases}
 n^{-\frac{s+1}{2s+D}} &\text{ if } D\geq 3\\
n^{-\frac 1 2}\log n &\text{ if } D= 2 \\
n^{-\frac 1 2} &\text{ if } D=1,
\end{cases}
\end{equation}
 and that this rate is attained by a modified linear wavelet density estimator.
Our main contribution consists in extending the results of \cite{weed2019estimation} by allowing the support of the probability to be any $d$-dimensional compact $\CC^k$ submanifold $M\subset \R^D$ for $k\geq 2$. More precisely, assume that some probability $\mu$ on $M$  has a lower and upper bounded density $f$ which belongs to the Besov space $B^s_{p,q}(M)$ for some $0< s \leq  k-1$, $1\leq p<\infty$, $1\leq q\leq \infty$ (see Section \ref{sec:model} for details). We first show (Theorem \ref{thm:estimator_M_known}) that some weighted kernel density estimator that we integrate against the volume measure $\vol_M$ on $M$ attains, for the $W_p$ distance,  the rate of estimation
\begin{equation}\label{eq:the_rate}
\begin{cases}
n^{-\frac{s+1}{2s+d}} &\text{ if } d\geq 3\\
n^{-\frac 1 2}\p{\log n}^{\frac 1 2} &\text{ if } d= 2 \\
n^{-\frac 1 2} &\text{ if } d=1.
\end{cases}
\end{equation}
In the case where the manifold $M$ is unknown, we do not have access to the volume measure $\vol_M$, so that the latter estimator is not computable. We therefore propose to estimate the volume measure $\vol_M$ in a preliminary step. Such an estimator $\widehat\vol_M$ is defined by using local polynomial estimation techniques from \cite{aamari2019nonasymptotic}. We show that this estimator is a minimax estimator of the volume measure up to logarithmic factors (Theorem \ref{thm:estimator_volume}), with a risk of order $\p{\log n/n}^{k/d}$. We then show (Theorem \ref{thm:M_unknown}) that a weighted kernel density estimator integrated against $\widehat \vol_M$ attains the rate \eqref{eq:the_rate}. Those rates are significantly faster than the rates of  \eqref{eq:weed} if $d\ll  D$ and are shown to be minimax up to logarithmic factors. 

Being able to estimate accurately the volume measure $\vol_M$ has also other useful implications, e.g. we provide an algorithm to sample points uniformly on a (possibly unknown) manifold, and we leverage results from \cite{trillos2020error} to provide precise estimates of the eigenvalues of the Laplace-Beltrami operator on $M$ (see Section \ref{sec:num}).
\medskip

In Section \ref{sec:model}, we define our statistical model and give some preliminary results on Wasserstein distances.  In Section \ref{sec:def_estim}, we define kernel density estimators on a manifold $M$, and state our main results.  
Proofs of the main theorems are then given in Section \ref{sec:proofs}. Section \ref{sec:num} discusses the implementation of our estimators, in particular proving that the local polynomial estimators of Aamari \& Levrard \cite{aamari2019nonasymptotic} can be efficiently computed. Additional proofs are found in the Appendix.

\section{Preliminaries}\label{sec:model}

\subsection{Regularity of manifolds}\label{sec:est_manifold}

For any $d>0$, we write $\dotp{\cdot,\cdot}$ for the dot product and $\vert v\vert $ for the norm of a vector $v\in \R^d$. The open ball centered at $x\in \R^d$ of radius $h>0$ is denoted by $\BB(x,h)$. For $\Omega\subset \R^d$ a set and $x\in \R^d$, we let $d(x,\Omega)\defeq \inf\{\vert x-y\vert ,\ y\in \Omega\}$ be the distance from $x$ to $\Omega$ and we write $\BB_\Omega(x,h)$ for $\BB(x,h)\cap \Omega$. Also,  we let $\Omega^h \defeq \{x\in \R^d,\ d(x,\Omega)<h\}$ be the $h$-tubular neighborhood of $\Omega$. Given a tensor  $A:(\R^{d_1})^i\to \R^{d_2}$ of order $i\geq 0$, the operator norm $\op{A}$ is defined as $\op{A} \defeq \max\{ \vert A[v_1,\dots,v_i] \vert,\ \vert v_1\vert ,\dots,\vert v_i\vert \leq 1\}$. Also, we let $A^\top :\R^{d_2}\to\R^{d_1}$ denote the adjoint of the operator $A:\R^{d_1}\to\R^{d_2}$. If $f:\Omega\to \R^{d_2}$ is a $\CC^k$ function defined on an open set $\Omega$ of $\R^{d_1}$, we let $\|f\|_{\CC^k(\Omega)} \defeq \max_{0\leq i\leq l} \sup\{ \op{d^i  f(x)},\ x\in\Omega\}$, where $d^i f(x)$ is the $i$th differential of $f$ at $x$.

 Let $D>0$ and let $\MM_{d}$ be the set of all smooth $d$-dimensional connected submanifolds in $\R^D$ without boundary, endowed with the metric induced by the standard metric on $\R^D$. We denote by $d_g$ the geodesic distance on $M$. The tangent space at a point $x\in M$ is denoted by $T_x M$.  It is identified with a $d$-dimensional subspace of $\R^D$, and the orthogonal projection on $T_x M$ is denoted by $\pi_x$. We also let $\t\pi_x:\R^D\to T_x M$ be defined by $\t\pi_x(y)=\pi_x(y-x)$. We denote by $T_x M^\bot$ the normal space at $x\in M$.
 The key quantity used to describe the regularity of a manifold $M$ is its reach $\tau(M)$. It is defined as the distance between $M$ and its medial axis, that is the set of points $x\in\R^D$ for which there are at least two points of $M$ which attain the distance from $x$ to $M$. In particular, the projection $\pi_M$ on the manifold $M$ is defined on $M^{\tau(M)}$. Originally introduced in \cite{federer1959curvature}, the reach $\tau(M)$ measures both the local regularity of $M$ (namely its curvature) and its global regularity, see e.g.~\cite{aamari2019estimating, berenfeld2020estimating} or \cite[Section 6.6]{delfour2001shapes} for precise results on the relationships between the reach of a manifold and its geometry. We then measure the regularity of $M$ through the regularity of local parametrizations of $M$ (see \cite{aamari2019nonasymptotic}).

\begin{definition}\label{def:manifold_holder}
Let $M\in \MM_{d}$, and $\taumin,L>0$, $k\geq 2$. Let $r_0=(\taumin\wedge L)/4$. We say that $M$ is in $\MM_{d,\taumin,L}^k$ if $M$ is closed, of reach larger than $\taumin$ and if, for all $x\in M$, the projection $\t\pi_x:M\to T_x M$ is a local diffeomorphism in $x$, with inverse $\Psi_x$ defined on $\BB_{T_x M}(0,r_0)$, satisfying $\|\Psi_x\|_{\CC^k(\BB_{T_x M}(0,r_0))}\leq  L$.
\end{definition}
\begin{remark}
\begin{enumerate}[label=(\roman*)]
\item Remark that we only consider manifolds $M$ that are smooth in the above definition, with a controlled $\CC^k$ norm. As the set of smooth submanifolds is dense in the set of $\CC^k$ submanifolds (for an appropriate topology), this is not a strong assumption. Dealing with smooth submanifolds is more convenient for us, as we do not have to deal with tricky existence issues for defining different functional spaces on manifolds.
\item For the sake of convenience, we use a definition slightly different from the definition of \cite{aamari2019nonasymptotic}, where authors assume the existence of local parametrizations $\t\Psi_x$ having controlled $\CC^k$ norms, with $\t \Psi_x$ not necessarily equal to the inverse $\Psi_x$ of the orthogonal projection. However, our definition is not restrictive. Indeed, on can write $\Psi_x = \t \Psi_x \circ (\t \pi_x \circ \t\Psi_x)^{-1}$, where the $\CC^k$ norm of $(\t \pi_x \circ \t\Psi_x)^{-1}$ is controlled by the inverse function theorem.  Therefore, the $\CC^k$ norm of $\Psi_x$ can always be controlled by the $\CC^k$ norms of other parametrizations $\t\Psi_x$. Both definitions can also be proven to be equivalent to assuming that the function $d^2(\cdot,M)$ has a controlled $\CC^k$ norm on $M^{\tau(M)}$, see e.g.~\cite{poly1984function}.
\item The value of the scale parameter $r_0$ is used for convenience. Other small scales could be used, or the radius $r_0$ could also be added as another parameter of the model, without any substantial gain in doing so.
\end{enumerate}
\end{remark}

  If $M_1\in \MM_{d_1}$, $M_2\in \MM_{d_2}$, $x\in M_1$ and $f:M_1\to M_2$ is a $\CC^1$ function, then we let $df(x): T_x M_1 \to T_{f(x)} M_2$ be the differential of $f$ at $x$. For $0\leq l \leq k$, if $f$ is $\CC^l$, we let $\op{d^if(x)} \defeq \max_{0\leq i\leq l} \op{d^i  (f\circ\Psi_x)(0)}$ and 
   $\| f\|_{\CC^l(M_1)}\defeq  \sup_{x\in M_1} \op{d^if(x)}.$  If $d_1\leq d_2$, then we define the Jacobian of $f$ at $x\in M_1$  as $Jf(x) =\sqrt{\det(df(x)^\top  df(x))}$. We let $\CC^l(M)$ be the space of all $\CC^l$ functions $f:M\to \R$ (with possibly $l=\infty$) and for $f\in\CC^1(M)$, we let $\nabla f$ denote the gradient of $f$. We also denote by $\nabla \cdot$ the divergence operator on $M$. 
 
 Let $\vol_M$ be the volume measure associated with the Riemannian metric on $M$. We will denote the integration with respect to $\dd\vol_M(x)$ by $\dd x$ when the context is clear.
 For $1\leq p \leq \infty$, we let $L_p(M)$ be the set of measurable functions $f:M\to \R$ with finite $p$-norm $\|f\|_{L_p(M)} \defeq \p{\int f \dd \vol_M}^{1/p}$ (and usual modification if $p=\infty$). We say that a locally integrable function $f$ is weakly differentiable if there exists a measurable section $\nabla f$ of the tangent bundle $T M$ (uniquely defined almost everywhere) such that for all smooth vector fields $w$ on $M$ with compact support, we have 
\[ \int f (\nabla\cdot w)\ \dd \vol_M=- \int (\nabla f)\cdot w \ \dd \vol_M.\]
  Furthermore, we will denote by $p^* \in [1,\infty]$ the number satisfying $\frac 1 p + \frac{1}{p^*}=1$.

\subsection{Besov spaces on manifolds}

Let $M\in \MM_{d,\taumin,L}^k$ for some $k\geq 2$, $\taumin,L>0$. As stated in the introduction, minimax rates for the estimation of a given probability will depend crucially on the regularity of its density $f$, which is assumed to belong to some Besov space $B^s_{p,q}(M)$. We first introduce Sobolev spaces $H^l_p(M)$ on $M$ for $l\leq  k$ an integer, and Besov spaces on $M$ are then defined by real interpolation.

\begin{definition}[Sobolev space on a manifold]\label{def:sobolev}
Let $0\leq  l \leq k$ , $1\leq p< \infty$ and let $f\in \CC^\infty(M)$ function.  We let
\begin{equation}\label{eq:def_sobolev}
\begin{split}
\|f\|_{H^l_p(M)} &\defeq \max_{0 \leq  i \leq l} \p{\int\op{d^i f(x)}^p \dd \vol_M(x)}^{1/p}.
\end{split}
\end{equation}
The space $H^l_p(M)$ is the completion of $\CC^\infty(M)$ for the norm $\|\cdot\|_{H^l_p(M)}$.
\end{definition}

\begin{remark}[On the case $p=\infty$]\label{rem:pinf}
The previous definition cannot be extended to the case $p=\infty$. Indeed, the completion of $\CC^\infty(M)$ for the norm $\|\cdot\|_{H^l_\infty(M)}$ is equal to $\CC^l(M)$, whereas for instance $H^0_\infty(M)$ should be equal to $L_\infty(M)$. For $l=1$, the space $H^1_p(M)$ can equivalently be defined as the space of weakly differentiable functions $f$ with $\|f\|_{H^1_p(M)} <\infty$, while this definition can be easily extended to the case $p=\infty$. In particular, if $f\in H^1_\infty(M)$, then one can verify that $f\circ\Psi_x \in H^1_\infty(\BB_{T_x M}(0,r_0))$ for any $x\in M$. It follows from standard results on Sobolev spaces on domains that $f\circ \Psi_x$ is Lipschitz continuous (see e.g.~\cite[Proposition 9.3]{brezis2010functional}). Hence, $f$ is also locally Lipschitz continuous. By Rademacher theorem, $f$ is therefore almost everywhere differentiable, and its differential coincides with the weak differential. As a consequence, a function $f\in H^1_\infty(M)$ is Lipschitz continuous, with Lipschitz constant for the distance $d_g$ equal to $\|f\|_{H^1_\infty(M)}$.
\end{remark}
For $1\leq p <\infty$, we introduce the negative homogeneous Sobolev norm $\|\cdot\|_{\dot{H}^{-1}_p(M)}$, defined, for $f\in L_p(M)$ with $\int f\dd \vol_M=0$, by
\begin{equation}
\|f\|_{\dot{H}^{-1}_p(M)}\defeq \sup\left\{\int fg \dd\vol_M,\ \|\nabla g\|_{L_{p^*}(M)}\leq 1\right\},
\end{equation}
where the supremum is taken over all functions $g\in H^1_{p^*}(M)$. For $f\in L_p(M)$, the negative Sobolev norm is defined by
\begin{equation}
\|f\|_{H^{-1}_p(M)}\defeq \sup\left\{\int fg \dd\vol_M,\ \|g\|_{H^1_{p^*}(M)}\leq 1\right\},
\end{equation}
and the corresponding Banach space is denoted by $H_p^{-1}(M)$.

\begin{proposition}\label{prop:congestionned}
Let $1\leq p<\infty$ and $f\in H^{-1}_p(M)$ with $\int f \dd \vol_M=0$. 
\begin{enumerate}[label=(\roman*)]
\item We have $C_{d,\taumin}\vert  \vol_M\vert ^{\frac{d-1}{p}-d}\|f\|_{\dot H^{-1}_p(M)}\leq \|f\|_{H^{-1}_p(M)} \leq \|f\|_{\dot H^{-1}_p(M)}$ for some positive constant $C_{d,\taumin}$ depending on $d$ and $\taumin$.\label{it:cong1}
\item We have $\|f\|_{\dot H^{-1}_p(M)} = \inf\{\|w\|_{L_p(M)},\ \nabla\cdot w= f\},$
where the infimum is taken over all measurable vector fields $w$ on $M$ with finite $p$-norm, and where $\nabla\cdot w=f$ means that $\int f g \dd \vol_M = -\int w\cdot \nabla g\dd \vol_M$ for all $g\in \CC^\infty(M)$.\label{it:cong2}
\end{enumerate}
\end{proposition}

 Following \cite{triebel1992theory}, Besov spaces on a manifold $M$ are defined as real interpolation of Sobolev spaces. We refer to \cite{lunardi2018interpolation} for definition of real interpolations of Banach spaces.

\begin{definition}[Besov space on a manifold]
Let $1\leq p <\infty$ and $0<  s <  k$. 
The Besov space $B^s_{p,q}(M)$ is defined as $B^s_{p,q}(M) \defeq (L_p(M),H^{k}_p(M))_{s/k,q},$ the real interpolation space between $L_p(M)$ and $H^{k}_p(M)$ of parameters $s/k$ and $q$.
\end{definition}
Basic results from interpolation theory then imply that $\|\cdot\|_{B^s_{p,q}(M)}\leq  \|\cdot\|_{B^{s'}_{p,q}(M)}$ if $0 < s\leq s'< k$. 

\subsection{Wasserstein distances and negative Sobolev distances}

Let $\PP$ be the set of finite Borel measures $\mu$ on $\R^D$, with $\vert  \mu \vert $ the total mass of $\mu$. Let $\PP_1$ be the set of measures in $\PP$ with $\vert  \mu \vert =1$. For $1\leq p\leq \infty$,  let $\PP^p$ be the set of measures $\mu \in \PP$ satifying $\p{\int \vert x\vert ^p\dd\mu(x)}^{1/p}<\infty$ (with usual modification for $p=\infty$) and let $\PP^p_1=\PP^p\cap \PP_1$. The pushforward of a measure $\mu$ by a measurable application $\phi:\R^D\to\R^D$ is defined by
\begin{equation}
\phi_\# \mu(A) \defeq \mu(\phi^{-1}(A))
\end{equation}
for any Borel set $A\subset \R^D$. For $\rho:\R^D\to [0,\infty)$ a measurable function, we denote by $\rho\cdot \mu$ the measure having density $\rho$ with respect to $\mu$.

\begin{definition}[Wasserstein distance]
Let $1\leq p\leq \infty$ and let $\mu,\nu\in \PP^p$ with the same total mass. Let $\Pi(\mu,\nu)$ be the set of transport plans between $\mu$ and $\nu$, i.e.~measures on $\R^D\times \R^D$ with first marginal $\pi^1$ (resp.~second marginal $\pi^2$) equal to $\mu$ (resp.~$\nu$). The cost $C_p(\pi)$ of $\pi\in \Pi(\mu,\nu)$ is defined as $\int \vert x-y\vert ^p\dd \pi(x,y)$. The $p$-Wasserstein distance between $\mu$ and $\nu$ is defined as
\begin{equation}
W_p(\mu,\nu) \defeq \inf_{\pi\in \Pi(\mu,\nu)} C_p(\pi)^{1/p},
\end{equation}
with usual modification for $p=\infty$.
\end{definition}

A crucial point in the study conducted in the following is the relation between Wasserstein distances and negative Sobolev norms. 
\begin{proposition}[Wasserstein distances and negative Sobolev norms]\label{prop:wass_neg} Let $1\leq p<\infty$. Let $M\in \MM_{d}$ be a manifold with reach $\tau(M)\geq \taumin$, and let $\mu,\nu\in \PP_1^p$ be two probability measures supported on $M$, absolutely continuous with respect to $\vol_M$, with densities $f$ and $g$. Assume that $f, g \geq \fmin \cdot \vol_M$ for some $\fmin>0$. Then, we have
\begin{equation}\label{eq:wass_neg}
\begin{split}
W_p(\mu,\nu) &\leq  p^{-1/p}\fmin^{1/p-1}\|f-g\|_{\dot H^{-1}_p(M)}\\
&\leq p^{-1/p}C_{d,\taumin,\fmin}\|f-g\|_{ H^{-1}_p(M)},
\end{split}
\end{equation}
for some constant $C_{d,\taumin,\fmin}$ depending on $d$, $\taumin$ and $\fmin$.
\end{proposition}
In particular, if $p=1$, then the first inequality in \eqref{eq:wass_neg} is actually an equality by the Kantorovitch duality formula \cite[Particular Case 5.16]{villani2008optimal}.
This inequality appears in \cite{peyre2018comparison} for $p=2$ and in \cite[Section 5.5.1]{santambrogio2015optimal} for measures having density with respect to the Lebesgue measure. We carefully adapt their proofs in Appendix \ref{sec:proof_preli}.

\subsection{Statistical models}\label{sec:stat}

  We consider the two following models, where points are sampled on a manifold, with possibly tubular noise. We fix in the following some parameters $\taumin,L_s,L_k>0$, $1\leq q\leq \infty$  and $0<\fmin<\fmax<\infty$. We also write $\MM_{d}^k$ instead of $\MM_{d,\taumin,L_k}^k$.

\begin{definition}[Noise-free model]
Let $d\leq  D$ be integers, $k\geq 2$, $0\leq s < k$ and $1\leq p<\infty$. Let $M\in \MM_{d}^k$. For $s=0$, the set $\QQ^0(M)$ is the set of probability distributions $\mu$ on $\R^D$ absolutely continuous with respect to the volume measure $\vol_M$, with a density $f$ satisfying $\fmin\leq f  \leq \fmax$ almost everywhere. For $s>0$, the set $\QQ^s(M)$ is the set of distributions $\mu \in \QQ^0(M)$, with density $f\in B^s_{p,q}(M)$ satisfying $\|f\|_{B^s_{p,q}(M)} \leq L_s$. The model $\QQ^{s,k}_{d}$ is equal to the union of the sets $\QQ^s(M)$ for $M\in \MM_d^k$. 
\end{definition}

\begin{remark}\label{remark:fmin}
If $\mu\in \QQ^s(M)$, then, as $\mu\geq \fmin \vol_M$, one has $\vol_M(M)\leq \fmin^{-1}$. One can then use standard packing arguments to show that this implies that $\diam(M) \leq C_d/(\fmin \taumin^{d-1})$ for some constant $C_d$ depending only on $d$. In particular, the manifold $M$ is automatically compact.
\end{remark}

Given a set of observations $X_1,\dots,X_n$ sampled according to $\mu\in \QQ^{s,k}_d$, the goal of statistical inference is to reconstruct some quantity $\vartheta(\mu)$ related to $\mu$. If $\LL$ is a loss function defined on the set of outputs of $\vartheta$, we define the minimax risk for this problem as
 \begin{equation}
 \RR_n(\vartheta,\QQ^{s,k}_d,\LL) \defeq \inf_{\hat \vartheta} \sup_{\mu\in \QQ} \E\LL(\hat\vartheta,\vartheta),
 \end{equation}
where $\hat\vartheta = \hat\vartheta(X_1,\dots,X_n)$ and $X_1,\dots,X_n$ is an i.i.d.~sample with law $\mu$.  We will focus here on reconstructing (i) the measure $\vartheta(\mu)=\mu$, and (ii) the uniform measure $\vartheta(\mu)=U_M \defeq \vol_M/\vert \vol_M\vert $, where $M$ is the support of $\mu$.  We first show that task (i) is impossible if the loss function $\LL$ is larger than the total variation distance $\TV$, which is defined by $\TV(\mu,\nu) \defeq \sup_{A}\vert  \mu (A)-\nu(A)\vert $ for $\mu,\nu\in \PP_1$, where the supremum is taken over all measurable sets $A\subset \R^D$. 

\begin{theorem}\label{thm:choice_of_loss}
Let $d\leq  D$ be integers, $k\geq 2$, $0\leq s< k$, $1\leq p<\infty$. Let $\LL: \PP\times \PP \to [0,\infty]$ be a measurable map with respect to the Borel $\sigma$-algebra associated with the total variation distance on $\PP\times \PP$. Assume that $\LL(\mu,\nu) \geq g(\TV(\mu,\nu))$ for a convex nondecreasing function $g:\R \to [0,\infty]$ with $g(0)=0$. Then, for any $\taumin>0$, if $\fmin$ is small enough and $L_k, L_s, \fmax$ are large enough, we have
\begin{equation}\label{eq:minimax_rate_bad}
\RR_n(\mu,\QQ^{s,k}_d,\LL) \geq g(c_d),
\end{equation}
for some constant $c_d>0$.
\end{theorem}
Examples of such losses include the total variation distance, the Hellinger distance (with $g(x)=x$), the Kullback-Leibler divergence (with $g(x)=x^2/2$), and the $L_p$ distance with respect to some dominating measure (with $g(x)=x^p$). 
We give a proof of Theorem \ref{thm:choice_of_loss}, based on Assouad's lemma, in Appendix \ref{sec:lowerbounds}. A simple example of loss $\LL$ which is not degenerate for mutually singular measures is given by the $W_p$ distance. As stated in the introduction, we will therefore choose this loss, and study $\RR_n(\mu,\QQ^{s,k}_d,W_p)$, the minimax rate of estimation for $\mu$ with respect to $W_p$. 

We will also study this problem in the presence of tubular noise.

\begin{definition}[Tubular noise model]
Let $d\leq  D$ be integers, $k\geq 2$, $0\leq s < k$, $1\leq p<\infty$ and $\gamma \geq 0$.  The set $\QQ^{s,k}_{d}(\gamma)$ is the set of laws $\nu$ of random variables $Y+Z$ where $Y\sim \mu\in \QQ^{s,k}_{d}$ and $Z\in\BB(0, \gamma)$ is such that $Z\in T_YM^\bot$.
\end{definition}
Note that the tubular noise model is not identifiable, in the sense that there are several admissible couples $(Y,Z)$ (with possibly different supports) such that $X=Y+Z$ follows the same distribution $\nu$. For each $\nu\in\QQ^{s,k}_{d}(\gamma)$, we will make an arbitrary choice among the admissible couples, while this will not have any impact on the following study (all the admissible couples have their first marginal at most $2\gamma$ apart for the Wasserstein distance).

\begin{remark}
For ease of notation, we will write in the following $a\lesssim b$ to indicate that there exists a constant $C$ depending on the parameters $p,k,\taumin,L_s,L_k,\fmin,\fmax$, \textbf{but not on $D$}, such that $a\leq Cb$, and write $a\asymp b$ to indicate that $a\lesssim b$ and $b\lesssim a$. Also, we will write $c_\alpha$ to indicate that a constant $c$ depends on some parameter $\alpha$.

Note that in particular, the risk of the estimators proposed in the next section will not depend on the ambient dimension $D$, but only on intrinsic parameters such as the regularity of $\nu$ and its effective dimension $d$.
\end{remark}

\section{Kernel density estimation on an unknown manifold}\label{sec:def_estim}
Before building an estimator in the model $\QQ^{s,k}_{d}(\gamma)$, let us consider the easier problem of the estimation of $\mu$ in the case where $\gamma=0$ (noise free model) and the support $M$ is known. Let $\mu\in \QQ^s(M)$ and $Y_1,\dots,Y_n$ be a $n$-sample of law $\mu$. Let $\mu_n=\frac{1}{n}\sum_{i=1}^n\delta_{Y_i}$ be the empirical measure of the sample. Identify $\R^d$ with $\R^d\times\{0\}^{D-d}$ and consider a kernel $K:\R^D\to \R$ satisfying the following conditions:
\begin{itemize}
\item \textbf{Condition $A$:} The kernel $K$ is a smooth radial function with support $\BB(0,1)$ such that $\int_{\R^d} K = 1$.
\item \textbf{Condition $B(m)$:}  The kernel $K$ is of order $m\geq 0$ in the following sense. Let $\vert  \alpha \vert \defeq \sum_{j=1}^d \alpha_j$ be the length of a multiindex $\alpha=(\alpha_1,\dots,\alpha_d)$. Then, for all multiindexes $\alpha^0$, $\alpha^1$ with $0\leq\vert  \alpha^0 \vert < m$, $0\leq \vert  \alpha^1 \vert < m+\vert  \alpha_0 \vert $, and with $\vert  \alpha^1 \vert >0$ if $\alpha^0=0$, we have
\begin{equation}
\int_{\R^d} \partial^{\alpha^0}K(v)v^{\alpha^1}\dd v= 0,
\end{equation}
where $v^\alpha=\prod_{j=1}^d v_j^{\alpha_j}$ and $\partial^\alpha K$ is the partial derivative of $K$ in the direction $\alpha$.
\item \textbf{Condition $C(\beta)$:} The negative part $K_- \defeq -\min(0,K)$ of $K$ satisfies $\int_{\R^d} K_- \leq\beta$.
\end{itemize}
We show in Appendix \ref{sec:kernel} that for every integer $m\geq 0$ and real number $\beta>0$, there exists a kernel $K$ satisfying conditions $A$, $B(m)$ and $C(\beta)$. Define the convolution of $K$ with a measure $\rho\in\PP$ as
\begin{equation}
K* \rho(x) \defeq\int K(x-y)\dd\rho(y),\quad x\in \R^D,
\end{equation}
and, for $h>0$, let $K_h\defeq h^{-d}K(\cdot/h)$. Let $\rho_h \defeq K_h*\vol_M$ and let $\mu_{n,h}$ be the measure with density $f_{n,h}\defeq K_h*(\mu_n/\rho_h)$ with respect to $\vol_M$. Dividing by $\rho_h$ ensures that $\mu_{n,h}$ is a measure of mass $1$. Remark that the computation of $\mu_{n,h}$ requires to have access to $M$, that is $\mu_{n,h}$ is an estimator on $\QQ^s(M)$ but not on $\QQ^{s,k}_{d}$. By linearity, the expectation of $\mu_{n,h}$ is given by $\mu_h$, the measure having for density $f_h\defeq K_h*(\mu/\rho_h)$ on $M$.

\begin{theorem}\label{thm:estimator_M_known}
Let $d\leq  D$ be integers, $0< s\leq k-1$ with $k\geq 2$  and $1\leq   p<\infty$. Let $M\in \MM_d^k$ and $\mu \in \QQ^s(M)$ with density $f$. Let $Y_1,\dots,Y_n$ be a $n$-sample of law $\mu$. There exists a constant $\beta$ depending on the parameters of the model such that, if $K$ is a kernel satisfying conditions $A$, $B(k)$ and $C(\beta)$,  then the measure $\mu_{n,h}$ satisfies the following:
\begin{enumerate}[label=(\roman*)]
\item If $(\log n/n)^{1/d}\lesssim h\lesssim 1$, then, with probability larger than $1-cn^{-k/d}$, the density $f_{n,h}$ of $\mu_{n,h}$ is  larger than $\fmin/2$ and smaller than $2\fmax$ everywhere on $M$. \label{it:pointwise_control}%%Furthermore, $\vert  \mu _{n,h}\vert =1$. 
\item If $n^{-1/d}\lesssim h\lesssim 1$, then we have
\begin{align}
\E\|f-f_{n,h}\|_{H^{-1}_p(M)} &\leq \|f- f_h\|_{H^{-1}_p(M)} + \E\|f_{n,h}-f_h\|_{H^{-1}_p(M)}\label{eq:decomposition_bias} \\ 
& \lesssim h^{s+1}+ \frac{h^{1-d/2}I_d(h)}{\sqrt{n}}, \label{eq:decomposition_variance}
\end{align}
where $I_d(h) = 1$ if $d\geq 3$, $(-\log(h))^{1/2}$ if $d=2$ and $h^{-1/2}$ if $d=1$. \label{it:sobo_control}
\item Let $h\asymp n^{-1/(2s+d)}$ if $d\geq  3$, $h \asymp (\log n/n)^{1/d}$ if $d\leq 2$. Define $\mu^0_{n,h} =\mu_{n,h}$ if $\mu_{n,h}$ is a nonnegative measure and $\mu_{n,h}^0=\delta_{X_1}$ otherwise. Then, 
\begin{equation}
\E W_p( \mu^0_{n,h},\mu ) \lesssim
\begin{cases}
n^{-\frac{s+1}{2s+d}} &\text{ if } d\geq 3,\\
n^{-\frac 1 2}\p{\log n}^{\frac 1 2} &\text{ if } d= 2,\\
n^{-\frac 1 2} &\text{ if } d= 1.
\end{cases}
\end{equation}\label{it:wass_control}
\item \label{it:noiseless_minimax} Furthermore, for any $0\leq s<k$ and $\taumin>0$, if $\fmin$ is small enough and if $\fmax$ and $L_s$ are large enough, then there exists a manifold $M\in \MM^k_d$ such that
\begin{equation}
\RR_n(\mu,W_p,\QQ^s(M)) \gtrsim \begin{cases}
n^{-\frac{s+1}{2s+d}} &\text{ if } d\geq 3,\\
n^{-\frac 1 2} &\text{ if } d\leq 2.
\end{cases}
\end{equation}
\end{enumerate}
\end{theorem}

\begin{remark}
The condition $C(\beta)$ on the kernel is only used to ensure that the measure $\mu_{n,h}$ has a lower and upper bounded density on $M$. An alternative possibility to ensure this property is to assume that the density of $\mu$ is H\"older continuous of exponent $\delta$ for some $\delta>0$. Techniques  from \cite{berenfeld2019density} then imply that $\|\mu_{n,h}-\mu\|_{L^\infty(M)} \lesssim h^\delta + n^{-1/2}h^{-d/2} \ll 1$ with high probability, ensuring in particular that the density is lowerbounded. If  $sp>d$, then every element of $B^s_{p,q}(M)$ is H\"older continuous \cite[Theorem 7.4.2]{triebel1992theory}, and condition $C(\beta)$ is no longer required. However, Theorem \ref{thm:estimator_M_known} also holds for non-continuous densities.
\end{remark}

\begin{remark}
Let $K$ be a nonnegative kernel satisfying conditions $A$, $B(0)$ and $C(\beta)$. It is straightforward to check that $W_p(\mu_n,\mu_{n,h})\lesssim h$. Therefore, Theorem \ref{thm:estimator_M_known}\ref{it:sobo_control} and Proposition \ref{prop:wass_neg} imply in particular that $W_p(\mu_n,\mu)\lesssim h + \frac{h^{1-d/2}I_d(h)}{\sqrt{n}}$. By choosing $h$ of the order $n^{-1/d}$, we obtain that 
\begin{equation}
W_p(\mu_n,\mu) \lesssim \begin{cases}
n^{-\frac 1 d} &\text{ if } d\geq 3\\
n^{-\frac 1 2}(\log n)^{\frac 1 2} &\text{ if } d=2\\
n^{-\frac 1 2} &\text{ if } d=1.
\end{cases}
\end{equation}
Such a result was already shown for $p=\infty$ \cite{trillos2020error} with additional logarithmic factors, with a proof very different than ours. See also \cite{divol2021short} for a short proof of this result when $M$ is the flat torus.
\end{remark}

\begin{remark}
There is a logarithmic gap between the minimax lower bound and the upper risk of the proposed estimator for $d=2$ in Theorem \ref{thm:estimator_M_known}. When $\mu$ is the uniform measure on the square $[0,1]^2$, it is known that the empirical measure $\mu_n$ attains exactly the rate $n^{-\frac 1 2}(\log n)^{\frac 1 2}$ \cite[Section 6.4]{talagrand2014upper}, suggesting that the factor $(\log n)^{1/2}$ is not a proof artifact. It is however not clear how one can transform Talagrand's tree construction used to lower bound $\E W_1(\mu_n,\mu)$ into a more general minimax lower bound on $\RR_n(\mu,W_p,\QQ^s(M))$, so that the (non-)existence of estimators attaining a rate of $n^{-1/2}$ for $d=2$ is still an open problem.
\end{remark}

In \eqref{eq:decomposition_variance}, a classical bias-variance trade-off appears. Namely, the bias of the estimator is of order $h^{s+1}$, whereas its fluctuations are of order $h^{1-d/2}/\sqrt{n}$ (at least for $d\geq 3$). This decomposition can be compared to the classical bias-variance decomposition for a kernel density estimator of bandwidth $h$, say for the pointwise estimation of a function of class $\CC^s$ on the cube $[0,1]^d$. It is then well-known (see e.g.~\cite[Chapter 1]{tsybakov2008introduction}) that the bias of the estimator is of order $h^s$ whereas its variance is of order $h^{-d/2}/\sqrt{n}$. The supplementary factor $h$ appearing both in the bias and fluctuation terms can be explained by the fact that we are using a norm $H^{-1}_p(M)$ instead of a pointwise norm to quantify the risk of the estimator: in some sense, we are estimating the antiderivative of the density rather than the density itself. This is particularly striking when $d=1$ and $p=1$, the Wasserstein distance between two measures being then given by the $L_1$ distance between the cumulative distribution functions of the two measures \cite[Proposition 2.17]{santambrogio2015optimal}.

Before giving a proof of Theorem \ref{thm:estimator_M_known}, let us explain how to extend it to the case where the manifold $M$ is unknown and in the presence of tubular noise. The measure $\mu_{n,h}$ is the measure having density $K_h*(\mu_n/\rho_h)$ with respect to $\vol_M$. Of course, if $M$ is unknown, then so is $\vol_M$, and we therefore propose the following estimation procedure of $\vol_M$, using local polynomial estimation techniques from \cite{aamari2019nonasymptotic}. Let $X_1,\dots,X_n$ be a $n$-sample in the model with tubular noise $\QQ^{s,k}_{d}(\gamma)$, with $X_i=Y_i+Z_i$, $Y_i$ of law $\mu\in \QQ^s(M)$ and $Z_i\in T_{Y_i}M^\bot$ with $\vert Z_i\vert \leq \gamma$. Let $\nu_n^{(i)}$ be the empirical measure $\frac{1}{n-1}\sum_{j\neq i}\delta_{X_j-X_i}$. For two positive parameters $\ell$, $\eps$, the local polynomial estimator $(\hat{\pi}_i,\hat{V}_{2,i},\dots,\hat{V}_{m-1,i})$ of order $m$ at $X_j$ is defined as an element of 
\begin{equation}\label{eq:def_poly}
\argmin_{\pi,\sup_{2\leq j\leq m-1} \op{V_j}\leq \ell}  \mu_n^{(i)}\p{ \left\vert  x-\pi(x)- \sum_{j=2}^{m-1} V_j[\pi(x)^{\otimes j}]\right \vert ^2 \ones\{x\in \BB(0,\eps)\} },
\end{equation}

where the argmin is taken over all orthogonal projectors $\pi$ of rank $d$ and symmetric tensors $V_j:(\R^D)^j\to \R^D$ of order $j$. 
Let $\hat T_i$ be the image of $\hat\pi_i$ and $\hat{\Psi}_i : v\in \R^D \mapsto X_i+v + \sum_{j=2}^{m-1} \hat V_{j,i}[v^{\otimes j}]$. Let $\angle(T_1,T_2)$ denote the angle between two $d$-dimensional subspaces, defined by $\op{\pi_{T_1}-\pi_{T_2}}$, where $\pi_{T_i}$ is the orthogonal projection on $T_i$ for $i=1,2$. We summarize the results of \cite{aamari2019nonasymptotic} in the following proposition (see Appendix \ref{sec:geom_prop} for details).

\begin{proposition}\label{prop:estim_manifold}
With probability at least $1- cn^{-k/d}$, if $m\leq k$, $\p{\log n/ n}^{1/d} \lesssim \eps \lesssim 1$, $\gamma \lesssim \eps$ and $ 1\lesssim \ell \lesssim \eps^{- 1}$, then, 
\begin{equation}\label{eq:control_angle}
\max_{1\leq i\leq n}\angle(T_{Y_i}M,\hat{T}_i) \lesssim \eps^{m-1} +  \gamma \eps^{-1}
\end{equation}
and, for all $1\leq i\leq n$, if $v\in \hat T_i$ with $\vert v\vert \leq 3\eps$, we have
\begin{align}
&\vert  \hat{\Psi}_i(v)-\Psi_{Y_i}\circ \pi_{Y_i}(v) \vert  \lesssim \eps^m + \gamma \label{eq:control_error_manifold}\\
&\op{d\hat{\Psi}_i(v) - d(\Psi_{Y_i}\circ\pi_{Y_i})(v)} \lesssim \eps^{m-1} + \gamma \eps^{-1}. \label{eq:control_error_manifold_derivative}
\end{align}
\end{proposition}
Hence, if $\gamma$ is of order at most $\eps^k$, then it is possible to approximate the tangent space at $Y_i$ with precision $\eps^{k-1}$ and the local parametrization with precision $\eps^k$. In particular, authors in \cite{aamari2019nonasymptotic} show that, with high probability, $ \bigcup_{i=1}^n \BB_{\hat\Psi_i (\hat T_i)}(X_i,\eps)$ is at Hausdorff distance of order at most $\eps^k + \gamma$ from $M$. We now define an estimator $\widehat \vol_M$ of $\vol_M$ by using an appropriate partition of unity $(\chi_j)_j$, which is built thanks to the next lemma. For $A,B\subset \R^D$, introduce the asymmetric Hausdorff distance $d_H(A\vert B) \defeq \sup_{x\in A} d(x,B)$ and the Hausdorff distance $d_H(A,B)\defeq d_H(A\vert B) \vee d_H(B\vert A)$. We say that a set $S$ is $\delta$-sparse if $\vert x-y\vert \geq \delta$ for all distinct points $x,y \in S$.

\begin{lemma}[Construction of partitions of unity]\label{lem:existence_pou}
Let $\delta \lesssim 1$. Let $S\subset M^\delta$ be a set which is $\frac 7 3 \delta$-sparse, with $d_H(M^\delta\vert S)\leq 4\delta$. Let $\theta:\R^D\to [0,1]$ be a smooth radial function supported on $\BB(0,1)$, which is equal to $1$ on $\BB(0,1/2)$. Define, for $y\in M^\delta$ and $x\in S$,
\begin{equation}
\chi_x(y) = \frac{\theta \p{\frac{y-x}{8\delta}}}{\sum_{x'\in S} \theta \p{\frac{y-x'}{8\delta}}}.
\end{equation}
Then, the sequence of functions $\chi_x:M^\delta\to [0,1]$ for $x\in S$, satisfies (i) $\sum_{x\in S} \chi_x \equiv 1$, with at most $c_d$ non zero terms in the sum at any given point of $M^{\delta}$, (ii) $\|\chi_x\|_{\CC^l(M^\delta)}\leq C_{l,d}\delta^{-l}$ for any $l\geq 0$ and, (iii) $\chi_x$ is supported on $\BB_{ M^\delta}(x,8\delta)$.
\end{lemma}
A proof of Lemma \ref{lem:existence_pou} is given in Appendix \ref{sec:geom_prop}. Lemma \ref{lem:existence_pou} requires the set $S$ to be $7\delta/3$-sparse. Actually, given a set $S_0\subset M^\delta$ with $d_H(M^\delta\vert S_0)\leq 5\delta/3$, there always exist a subset $S\subset S_0$ that satisfies the assumptions of Lemma \ref{lem:existence_pou}. Such a subset $S$ can be computed using the farthest point sampling algorithm (see e.g.~\cite[Section 3.3]{aamari2018stability}): initialize $S$ with an arbitrary point of $S_0$, and, at each step, add the farthest point $s\in S_0$ from $S$ (that is the point that maximizes $d(s,S_0)$). We stop the algorithm when every point of $s\in S_0$ satisfies $d(s,S)\leq 7\delta/3$. By construction, the set $S$ is $7\delta/3$-sparse and satisfies $d_H(S,S_0) \leq 7\delta/3$. We then have $d_H(M^\delta\vert S) \leq 5\delta/3 + 7\delta/3=4\delta$, so that indeed one can construct a partition of unity using $S$.

 The next proposition describes how we may define a minimax  estimator $\widehat \vol_M$ of the volume measure on $M$ (up to logarithmic factors) using such a partition of unity.  

\begin{theorem}[Minimax estimation of the volume measure on $M$]\label{thm:estimator_volume}
Let $d\leq  D$ be integers and $k\geq 2$. Let $\nu\in \QQ^{0,k}_{d}(\gamma)$ and let $X_1,\dots,X_n$ be a $n$-sample of law $\nu$. Let $\p{\log n/ n}^{1/d} \lesssim \eps \lesssim 1$,  $\gamma \lesssim \eps$, $1\lesssim \ell \lesssim \eps^{-1}$.
\begin{enumerate}[label=(\roman*)]
\item Let $\{X_{i_1},\dots,X_{i_J}\}$ be the output of the farthest point sampling algorithm with parameter $7\eps/24$ and input $\{X_1,\dots,X_n\}$. With probability larger than $1-cn^{-k/d}$, there exists a sequence of smooth nonnegative functions $\chi_j: M^{\eps/8}\to [0,1]$ for  $1\leq j \leq J$,  such that $\chi_j$ is supported on $\BB_{M^{\eps/8}}(X_{i_j},\eps)$,  $\|\chi_j\|_{\CC^1(M^{\eps/8})} \lesssim \eps^{-1}$ and $\sum_{j=1}^J \chi_j(z) =1$ for $z\in M^{\eps/8}$, with at most $c_d$ non-zero terms in the sum. \label{it:pou}
\item \label{it:control_estimator_volume} Let $\hat\Psi_i$ be the local polynomial estimator of order $m\leq k$ with parameter $\eps$ and $\ell$, and $\hat T_i$ the associated tangent space. Let $\widehat\vol_M$ be the measure defined by, for all continuous bounded functions $ \phi:\R^D\to \R$,
\begin{equation}\label{eq:estimator_volume}
\int \phi(x)\dd \widehat\vol_M(x) = \sum_{j=1}^J \int_{\hat \Psi_{i_j}(\hat T_{i_j})}\phi(x) \chi_j(x)\dd x,
\end{equation} 
where the integration is taken against the $d$-dimensional Hausdorff measure on $\hat \Psi_{i_j}(\hat T_{i_j})$. Then, for $1\leq r \leq \infty$, with probability larger than $1- cn^{-k/d}$, we have
\begin{equation}\label{eq:bound_wass_volume}
\begin{split}
&W_r\p{ \frac{\widehat\vol_M}{\vert \widehat\vol_M\vert }, \frac{\vol_M}{\vert \vol_M\vert }} \lesssim (\gamma + \eps^m)(1+\gamma \eps^{-2}).
\end{split}
\end{equation} 
\item \label{it:minimax_volume} In particular, if $m=k$, $\eps \asymp (\log n/n)^{1/d}$ and $\gamma\lesssim \eps^2$, we obtain that 
\begin{equation}\label{eq:risk_volume_estimator}
\begin{split}
&\E W_r\p{ \frac{\widehat\vol_M}{\vert \widehat\vol_M\vert }, \frac{\vol_M}{\vert \vol_M\vert }} \lesssim \gamma + \p{\frac{\log n}{n}}^{\frac k d}.
\end{split}
\end{equation}
Also, for any $\taumin>0$ and $0\leq s <k$, if $\fmin$ is small enough, and if $\fmax, L_k, L_s$ are large enough, then
\begin{equation}\label{eq:minimax_risk_volume}
\RR_n\p{\frac{\vol_M}{\vert \vol_M\vert },\QQ^{s,k}_d,W_r} \gtrsim  \p{\frac{1}{n}}^{\frac k d}.
\end{equation} 
\end{enumerate}
\end{theorem}

Let $\hat\rho_h \defeq K_h*\widehat\vol_M$. We define $\hat\nu_{n,h}$ as the measure having density $K_h*(\nu_n/\hat\rho_h)$ with respect to the measure $\widehat \vol_M$, where $\nu_n=\frac{1}{n} \sum_{i=1}^n \delta_{X_i}$ is the empirical measure of the sample $(X_1,\dots,X_n)$.

\begin{theorem}\label{thm:M_unknown}
Let $d\leq  D$ be integers, $0< s\leq k-1$ with $k\geq 2$  and $1\leq  p<\infty$. Let $\nu\in \QQ^{s,k}_{d}(\gamma)$ and let $X_1,\dots,X_n$ be a $n$-sample of law $\nu$.  Decompose $X_i$ as $Y_i+Z_i$, where $Y_i$ has law $\mu\in\QQ^{s,k}_d$ and $Z_i$ is an orthogonal noise of size $\gamma$.  There exists a constant $\beta$ depending on the parameters of the model such that the following holds. Assume that $K$ is a kernel satisfying conditions $A$, $B(k)$ and $C(\beta)$, that $(\log n/n)^{1/d}\lesssim \eps \lesssim h \lesssim 1$, $\gamma\lesssim \eps$,  $1\lesssim \ell \lesssim \eps^{-1}$ and consider the estimator $\widehat \vol_M$  defined in \eqref{eq:estimator_volume} with parameters $m$, $\eps$ and $\ell$. We have the following:
\begin{enumerate}[label=(\roman*)]
\item The measure $\hat \nu_{n,h}$ is a nonnegative measure with probability larger than $1-cn^{-k/d}$. \label{it:nonnegative}
\item \label{it:proxim_estimation_volume} Define $\hat\nu^0_{n,h} =\hat\nu_{n,h}$ if $\hat\nu_{n,h}$ is a nonnegative measure and $\hat\nu_{n,h}^0=\delta_{X_1}$ otherwise. Then, with probability larger than $1-cn^{-k/d}$,
\begin{equation} 
W_p(\hat\nu^0_{n,h},\mu_{n,h}^0) \lesssim (\gamma+\eps^m)(1+\gamma \eps^{-2}).
\end{equation} 
\item \label{it:rate_final_estimator} In particular, let $m=\lceil s+1\rceil$, $\eps \asymp (\ln n/n)^{1/d}$, $\gamma\lesssim \eps^2$, $\ell \asymp\eps^{-1}$ and $h\asymp n^{-1/(2s+d)}$ if $d\geq  3$, $h \asymp (\log n/n)^{1/d}$ if $d\leq 2$. Then, \begin{equation}\label{eq:consequence_of_the_previous_eq}
\E W_p(\hat\nu_{n,h}^0,\mu) \lesssim \gamma+ \begin{cases}
n^{-\frac{s+1}{2s+d}} &\text{ if } d\geq 3,\\
n^{-\frac 1 2}\p{\log n}^{\frac 1 2}&\text{ if } d= 2, \\
n^{-\frac 1 2} &\text{ if } d=1.
\end{cases}
\end{equation} 
\end{enumerate}
\end{theorem}

We thus manage to obtain the same rate of convergence as in Theorem \ref{thm:estimator_M_known}, although not having access to the manifold $M$  (up to a noise factor $\gamma$).

\section{Proofs of the main theorems}\label{sec:proofs}

The proof of Theorem \ref{thm:estimator_M_known} relies on the classical bias variance decomposition displayed in \eqref{eq:decomposition_bias}, together with the linearization inequality given in Proposition \ref{prop:wass_neg}. We first bound the bias of the estimator, which can be expressed in term of a convolution operator on $M$.

\subsection{Bias of the kernel density estimator}\label{sec:study_estimator}
The bias of the estimator is given by the distance $\|\cdot\|_{H^{-1}_p(M)}$ between $f$ the density of $\mu$ and $f_h\defeq K_h*(\mu/\rho_h)$. For $\phi \in L_1(M)$, we write $\t\phi$ for $\phi/\rho_h$. Introduce the operator $A_h: B^s_{p,q}(M) \to H^{-1}_p(M)$ defined for $\phi\in L_1(M)$ by 
\begin{equation}
A_h\phi \defeq K_h*\p{\frac{\phi}{\rho_h}}-\phi =\int_M K_h(\cdot-y)\p{\t\phi(y)-\t\phi}\dd \vol_M(y).
\end{equation}
Then, 
\begin{equation}\label{eq:how_to_bound_bias}
\begin{split}
\|f_h-f\|_{H^{-1}_p(M)} = \|A_h f\|_{H^{-1}_p(M)} &\leq  \|A_h\|_{B^s_{p,q}(M), H^{-1}_p(M)} \|f\|_{B^s_{p,q}(M)}\\
&\leq \|A_h\|_{B^s_{p,q}(M), H^{-1}_p(M)} L_s. 
\end{split}
\end{equation} 
Therefore, it suffices to control the operator norm of $A_h$ to bound the bias.
\begin{proposition}\label{prop:bound_bias}
Let $0< s\leq k-1$ and $1\leq p <\infty$. Assume that the kernel $K$ is of order $k$. Then, if $h\lesssim 1$,
\begin{equation}
\|A_h\|_{B^s_{p,q}(M), H^{-1}_{p}(M)} \lesssim h^{s+1}.
\end{equation}
\end{proposition}

The proof of Proposition \ref{prop:bound_bias} consists in considering the Taylor expansion of a function $\phi\in B^s_{p,q}(M)$. In this Taylor expansion, all polynomial terms of low order disappear when integrated against $K$, as the kernel $K$ is of sufficiently large order. Namely, we have the following property, whose proof is given in Appendix \ref{sec:proof_bias}.

\begin{lemma}\label{lem:kernel_of_order_k}
Assume that the kernel $K$ is of order $k$ and let $B:(\R^D)^j\to \R$ be a tensor of order $1\leq j< k$.  Then, 
 for all $x\in  M$,
\begin{align}
&\left\vert \int_M K_h(x-y)B[(x-y)^{\otimes j}] \dd y\right \vert  \lesssim \op{B}h^{k} \label{eq:kernel_1} \\
&\left\vert \rho_h(x) - 1\right \vert  \lesssim h^{k-1}\ \text{ and }\ \|\rho_h\|_{\CC^j(M)} \lesssim h^{k-1-j}  \label{eq:kernel_0}
\end{align}
\end{lemma}

Let us now give a sketch of proof of Proposition \ref{prop:bound_bias} in the case $0< s\leq 1$. The $H^{-1}_{p}(M)$-norm of $A_h\phi$ is by definition equal to
\[
\|A_h \phi\|_{H^{-1}_{p}(M)} = \sup \left\{ \int (A_h\phi) g\dd\vol_M,\ \|g\|_{H^{1}_{p^*}(M)}\leq 1 \right\}.
\]
Let $g\in H^{1}_{p^*}(M)$ with $ \|g\|_{H^{1}_{p^*}(M)}\leq 1$. We use the following symmetrization trick:
\begin{align}
&\int (A_h\phi)(x)g(x)\dd x =  \iint K_h(x-y) (\t\phi(y)- \t\phi(x))g(x) \dd y\dd x \nonumber\\
&\quad =\iint K_h(y-x)(\t\phi(x)- \t\phi(y))g(y)\dd y \dd x \text{ (by swapping the indexes $x$ and $y$)}\nonumber \\
&\quad = \frac{1}{2} \iint K_h(x-y)(\t\phi(y)- \t\phi(x))(g(x)-g(y))\dd y\dd x \label{eq:the_trick}
\end{align}
where, at the last line, we averaged the two previous lines and used that $K$ is an even function. Informally, as $K_h(x-y)=0$ if $\vert x-y\vert \geq h$, and as $\rho_h$ is roughly constant, we expect $\vert \t\phi(y)- \t\phi(x)\vert $ to be of order $h^s$ and $\vert g(x)-g(y)\vert $ to be of order $h$, leading to a bound of $\int (A_h\phi)(x)g(x)\dd x$ of order $h^{s+1}$.  Such a trick can be generalized to a function $\phi$ of regularity $s>1$ by writing a Taylor expansion of $\phi$ and using that all the polynomial terms $(x-y)^\alpha$ vanish when integrated against $K_h(x-y)$ according to Lemma \ref{lem:kernel_of_order_k}.  More precisely, we prove in Appendix \ref{sec:proof_bias} the following higher order analogue of the symmetrization trick.

\begin{lemma}[Symmetrization trick]\label{lem:bias2}
There exists $h_0\lesssim 1$ such that the following holds. Let  $l$ be an even number between $0$ and $k-1$ and let $K^{(l)}(x) = \int_0^1 K_\lambda(x) \frac{(1-\lambda)^{l-1}\lambda^{-l}}{(l-1)!}\dd \lambda$ for $x\in\R^D$. Fix $x_0\in M$ and let $\phi \in \CC^\infty(M)$ be a function supported in $\BB_M(x_0,h_0)$. Define $\t \phi_l \defeq d^l (\t \phi \circ\Psi_{x_0}) \circ \t\pi_{x_0}$. Let $g\in L_{p^*}(M)$ with $\|g\|_{L_{p^*}(M)}\leq  1$.

  Then, for $h\lesssim 1$, $\int (A_h\phi)(x) g(x)\dd x$ is equal to
\begin{align}
&\frac{1}{2} \iint_{ \BB_M(x_0,h_0)^2} K^{(l)}_h(x-y)(\t \phi_l(y)-\t \phi_l(x)) \left[\pi_{x_0}(x-y) \right]^{\otimes l}\p{g(x)-g(y)}\dd y\dd x +R,
\end{align}
where $R$ is a remainder term satisfying $\vert R\vert \lesssim \|\t \phi\|_{H^{l}_p(M)} h^{l+1}$. Furthermore, if $l\leq k-2$, we have $\vert R\vert \lesssim \|\t \phi\|_{H^{l+1}_p(M)} h^{l+2}$.
\end{lemma}

There are two restrictions in the above lemma: (i) the function $\phi$  has to be supported on a small ball and (ii) the regularity $s$ of $\phi$ is an integer $l$. We will bypass the first restriction by using a partition of unity, whereas we will use interpolation inequalities to go from integer regularity $l$ to a regularity $s$ being any real number between $0$ and $k-1$. Let us first see how one can use Lemma \ref{lem:bias2} to bound the norm of $\int (A_h\phi)(x) g(x)\dd x$. We use the following technical lemma.

\begin{lemma}\label{lem:equiv_norm_sobolev}
Let $\eta\in \CC^\infty(M)$ and let $0\leq l\leq k-2$. Assume that either $l=0$ or that $\eta$ is supported on $\BB_M(x_0,h_0)$. Let $\eta_l= d^l(\eta\circ\Psi_{x_0}) \circ \t\pi_{x_0}$. Then,  for any $h\lesssim 1$,
\begin{equation}
\begin{split}
&\p{h^{-d} \iint_{\BB_M(x_0,h_0)^2} \ones\{\vert x-y\vert \leq  h\} \frac{\op{ \eta_l(x)- \eta_l(y)}^p}{\vert x-y\vert ^p}\dd x\dd y}^{1/p}\\
&\qquad \qquad \lesssim  \p{\int_{ \BB_M(x_0,h_0)}\op{\eta_{l+1}(x)}^p\dd x}^{1/p} \lesssim   \|\eta\|_{H^{l+1}_p(M)}.
\end{split}
\end{equation}
\end{lemma}

Let $\phi \in \CC^\infty(M)$ be a function supported in $\BB_M(x_0,h_0)$ and $g\in H^{1}_{p^*}(M)$ with $ \|g\|_{H^{1}_{p^*}(M)}\leq 1$.

\paragraph{Case $1$: the regularity $s$ is an even number} Let $l=s$. 
Assume first that $p>1$ and that $g$ is smooth. We may use that $\vert K_{h\lambda}(x) \vert=0$ if $\vert x \vert>h\lambda$ to obtain
\begin{align}
&\iint_{  \BB_M(x_0,h_0)^2}\vert  K_{\lambda h}(x-y)\vert \op{\t \phi_l(y)-\t \phi_l(x)}\vert g(x)-g(y)\vert \vert x-y\vert ^{l}\dd x\dd y \nonumber\\
&\leq \|K\|_{\CC^0(\R^D)} (\lambda h)^{l+1-d} \iint_{  \BB_M(x_0,h_0)^2} \ones\{\vert x-y\vert \leq \lambda h\}\op{\t \phi_l(y)-\t \phi_l(x)} \frac{\vert g(x)-g(y)\vert }{\vert x-y\vert }\dd x\dd y\nonumber \\
&\lesssim (\lambda h)^{l+1}\hspace{-.1cm} \p{ (\lambda h)^{-d} \iint_{  \BB_M(x_0,h_0)^2} \ones\{\vert x-y\vert \leq \lambda h\}\op{\t \phi_l(y)-\t \phi_l(x)}^{p}\dd x\dd y}^{\hspace{-.1cm} 1/p}\nonumber \\
&\qquad\qquad \times \p{(\lambda h)^{-d} \iint_{  \BB_M(x_0,h_0)^2} \ones\{\vert x-y\vert \leq \lambda h\}  \frac{\vert g(x)-g(y)\vert ^{p^*}}{\vert x-y\vert ^{p^*}}\dd x\dd y}^{\hspace{-.1cm}1/p^*} \nonumber\\
&\lesssim (\lambda h)^{l+1} \p{(\lambda h)^{-d}\int_{x\in \BB_M(x_0,h_0)}\op{\t \phi^l(x)}^p \vol_M(\BB_M(x,\lambda h))\dd x}^{1/p}\hspace{-.2cm} \|g\|_{H^{1}_{p^*}(M)} \nonumber\\
&\lesssim\|\t\phi\|_{H^l_p(M)} (\lambda h)^{l+1} \lesssim \|\phi\|_{H^l_p(M)} (\lambda h)^{l+1},\label{eq:proofbias2}
\end{align}
where at the last line, we used Lemma \ref{lem:prop_proj}\ref{it:bound_volume} in the Appendix to control the volume of $\BB_M(x,\lambda h)$ and, at the second to last line, we used Lemma \ref{lem:equiv_norm_sobolev}. Furthermore, it follows from Leibniz formula for the derivative of a product and Lemma \ref{lem:kernel_of_order_k} that $\|\t\phi\|_{H^l_p(M)}\lesssim \|\phi\|_{H^l_p(M)}$.

As $\CC^\infty(M)$ is dense in $H^1_{p^*}(M)$, inequality \eqref{eq:proofbias2} actually holds for every $g\in H^1_{p^*}(M)$. If $p=1$, then every function $g\in H^1_{p^*}(M)$ with $\|g\|_{H^{1}_{p^*}(M)}\leq 1$ is Lipschitz continuous for the distance $d_g$ (see Remark \ref{rem:pinf}). Using that $d_g(x,y) \leq 2\vert x-y\vert $ if $\vert x-y\vert \leq \taumin/4$ (see \cite[Proposition 30]{aamari2018stability}), a similar computation than in the case $p<\infty$ shows that inequality \eqref{eq:proofbias2} also holds if $p=\infty$. 

 By integrating inequality \eqref{eq:proofbias2} against $\lambda\in (0,1)$ and by using Lemma \ref{lem:bias2}, we obtain the inequality $\|A_h \phi\|_{ H^{-1}_{p}(M)} \lesssim h^{s+1}\|\phi\|_{H^s_{p}(M)}$. 
 \medskip

\paragraph{Case $2$: the regularity $s$ is an odd number} 
 Similarly, we treat the case where $s\leq k-1$ is odd. Let $l=s-1$. 
Once again, assume first that $p>1$ and that $g$ is smooth. Then,
\begin{align}
&\iint_{  \BB_M(x_0,h_0)^2}\vert  K_{\lambda h}(x-y)\vert \op{\t \phi_l(y)-\t \phi_l(x)}\vert g(x)-g(y)\vert \vert x-y\vert ^{l}\dd x\dd y \nonumber\\
&\leq\iint_{ \BB_M(x_0,h_0)^2} \vert K_{\lambda h}(x-y)\vert \frac{\op{\t \phi_l(y)-\t \phi_l(x)}}{\vert x-y\vert }\frac{\vert g(x)-g(y)\vert }{\vert x-y\vert }\vert x-y\vert ^{l+2}\dd x\dd y \nonumber\\
&\leq \|K\|_{\CC^0(\R^D)}(\lambda h)^{l+2-d} \iint_{\BB_M(x_0,h_0)^2} \ones\{\vert x-y\vert \leq \lambda h\}\frac{\op{\t \phi_l(y)-\t \phi_l(x)}}{\vert x-y\vert }\frac{\vert g(x)-g(y)\vert }{\vert x-y\vert }\dd x\dd y \nonumber\\
&
\begin{aligned}
{}\lesssim (\lambda h)^{l+2}\hspace{-.1cm} \p{ (\lambda h)^{-d} \iint_{  \BB_M(x_0,h_0)^2} \ones\{\vert x-y\vert \leq \lambda h\}\op{\t \phi_l(y)-\t \phi_l(x)}^{p}\dd x\dd y}^{\hspace{-.1cm} 1/p}\nonumber \\
\times \p{(\lambda h)^{-d} \iint_{  \BB_M(x_0,h_0)^2} \ones\{\vert x-y\vert \leq \lambda h\}  \frac{\vert g(x)-g(y)\vert ^{p^*}}{\vert x-y\vert ^{p^*}}\dd x\dd y}^{\hspace{-.1cm}1/p^*} \nonumber
\end{aligned}\\*
&\lesssim  (\lambda h)^{l+2} \|\phi\|_{H^s_p(M)},\label{eq:proofbias1}
\end{align}
where at last line we used Lemma \ref{lem:equiv_norm_sobolev} and the inequality $\|\t\phi\|_{H^l_p(M)}\lesssim \|\phi\|_{H^l_p(M)}$. As in the previous case, the same inequality holds for $g\in H^1_{p^*}(M)$ non necessarily smooth and if $p=1$. By using Lemma \ref{lem:bias2} and by integrating \eqref{eq:proofbias1} against $\lambda\in (0,1)$, we obtain that  $\|A_h \phi\|_{ H^{-1}_{p}(M)} \lesssim h^{s+1}\| \phi\|_{H^{s}_{p}(M)}$. 
\medskip

So far, we have proven that
\begin{equation}
\|A_h \phi\|_{ H^{-1}_{p}(M)} \lesssim h^{s+1}\|\phi\|_{H^s_{p}(M)}
\end{equation} 
for all integers $0\leq s \leq k-1$ and $\phi$ a smooth function supported on $\BB_M(x_0,h_0)$.  
To obtain the result when $\phi$ is not supported on some ball $\BB_M(x_0,h_0)$, we use an appropriate partition of unity. Indeed, for $\delta =h_0/8$, standard packing arguments show the existence of a set $S_0$ of cardinality $N\leq c_d \vert \vol_M\vert \delta^{-d}$ with $d_H(M^\delta\vert S_0)\leq 5\delta/3$. By the remark following Lemma \ref{lem:existence_pou}, the output $S$ of the farthest point sampling algorithm with parameter $7\delta/3$ satisfies the assumption of Lemma \ref{lem:existence_pou}, and is of cardinality smaller than $N\lesssim 1$. We consider such a covering $(\BB_M(x,h_0))_{x\in S}$, with associated partition of unity $(\chi_x)_{x\in S}$ given by Lemma \ref{lem:existence_pou}. Then, $\|A_h \phi\|_{ H^{-1}_{p}(M)}$ is bounded by
\begin{align*}
  &\sum_{x\in S} \|A_h (\chi_x \phi)\|_{ H^{-1}_{p}(M)} \lesssim h^{s+1} \sum_{x\in S} \|\chi_x \phi\|_{H^s_{p}(M)} \lesssim  h^{s+1}\sum_{x\in S} \|\chi_x\|_{\CC^s(M)}\|\phi\|_{H^s_{p}(M)} \lesssim h^{s+1}\|\phi\|_{H^s_{p}(M)} ,
\end{align*}
where the second to last inequality follows from Leibniz rule for the derivative of a product. Also, the last inequality follows from the fact that $(\chi_x)_{\vert M}= \chi_x \circ i_M$, where $i_M : M\to M^\delta$ is the inclusion, which is a $\CC^k$ function with controlled $\CC^k$-norm. Hence,  $\|\chi_x\|_{\CC^s(M)}\lesssim \|\chi_x\|_{\CC^s(M^{\delta})}\lesssim 1$ by the chain rule.
 
 As $\CC^\infty(M)$ is dense in $H^s_p(M)$, this gives the desired bound on the operator norm of $A_h:H^s_p(M)\to H^{-1}_p(M)$ for $0\leq s\leq k-1$ an integer.  To obtain the conclusion for Besov spaces $B^s_{p,q}(M)$, we use an interpolation inequality \cite[Theorem 1.1.6]{lunardi2018interpolation}. By the reiteration theorem \cite[Theorem 1.3.5]{lunardi2018interpolation}, for $0<s< k-1$, $B^s_{p,q}(M) = (L_p(M),H^{k-1}_p(M))_{s/(k-1),q}$, with an equivalent norm. Hence, we have, for $0<s<k-1$, letting $\theta=s/(k-1)$,
\begin{align*}
 \|A_h\|_{B^s_{p,q}(M), H^{-1}_{p}(M)} &\lesssim \|A_h\|_{L_p(M), H^{-1}_{p}(M)}^{1-\theta}\|A_h\|_{H^{k-1}_p(M), H^{-1}_{p}(M)}^{\theta} \\
 &\lesssim  h^{1-\frac{s}{k-1}}h^{k\frac{s}{k-1}} \lesssim  h^{s+1},
\end{align*}
so that Proposition \ref{prop:bound_bias} is proven for $s<k-1$. It remains to prove the inequality in the case $s=k-1$. By Fatou's lemma and the definition of interpolation spaces \cite[Definition 1.1.2]{lunardi2018interpolation}, we have, for some constant $C$ not depending on $s$,
\begin{align*}
\| A_h f\|_{B^{k-1}_{p,q}(M)} &\leq \liminf_{\substack{s\to k-1\\s<k-1}} \| A_h f\|_{B^{s}_{p,q}(M)} \leq \liminf_{\substack{s\to k-1\\s<k-1}} \p{ C h^{s+1} \|f\|_{B^{s}_{p,q}(M)}}\leq   C h^{k} \|f\|_{B^{k-1}_{p,q}(M)},
\end{align*}
where we used that $\|f\|_{B^{s}_{p,q}(M)}\leq \|f\|_{B^{k-1}_{p,q}(M)}$. This concludes the proof of Proposition \ref{prop:bound_bias}.

\subsection{Fluctuations of the kernel density estimator}\label{sec:fluctuation}
The second step in bounding the Sobolev risk $\E\|f_{n,h}-f\|_{H^{-1}_p(M)}$ is to control the fluctuation term $\E \|f_{n,h}-f_h\|_{H^{-1}_p(M)}$. If we were considering a classical $L_p$-norm instead of a negative Sobolev norm, then we could simply express $f_{n,h}-f_h$ as a sum of the i.i.d.~terms to obtain the right order. The key idea to bound the fluctuation term is to show that  $\|f_{n,h}-f\|_{H^{-1}_p(M)}$ is smaller than the $L_p$-norm of a similar sum of i.i.d.~functions that can be expressed in term of the Green's function on $M$. We obtain the following control.

\begin{proposition}\label{prop:fluc}
Let $\mu \in \QQ^s(M)$ with $Y_1,\dots,Y_n$ a $n$-sample of law $\mu$. Assume that $h\lesssim 1$ and that $nh^d\gtrsim 1$. Then, 
\begin{equation}\label{eq:fluc0}
\E \|f_{n,h}-f_h\|_{H^{-1}_p(M)} \lesssim n^{-1/2}h^{1-d/2}I_d(h),
\end{equation}
where $I_d(h)$ is defined in Theorem \ref{thm:estimator_M_known}.
\end{proposition}

Let $\Delta$ be the Laplace-Beltrami operator on $M$ and  $G:\DD_M\to \R$ be a Green's function, defined on  $\DD_M\defeq \{(x,y)\in M\times M,\ x\neq y\}$ (see \cite[Chapter 4]{aubin1982nonlinear}). By definition, if $\phi\in \CC^\infty(M)$, then the function $G\phi:x\in M \mapsto  \int G(x,y)\phi(y)\dd y$ is a smooth function satisfying $\Delta G\phi=\phi$, with $ \nabla G \phi(x) = \int \nabla_x G(x,y) \phi(y) \dd y$ for $x\in M$. 
 Hence, if $w=\nabla G\phi$, then $\nabla\cdot w=\phi$, so that, Proposition \ref{prop:congestionned} yields
\[ \|\phi\|_{H^{-1}_p(M)} \leq \|\phi\|_{\dot H^{-1}_p(M)} \leq \|\nabla G \phi\|_{L_p(M)}.\]
By linearity, we have 
\begin{equation}\label{eq:fluc1}
\begin{split}
&\|f_{n,h}-f_h\|_{H^{-1}_p(M)}  \leq  \left\| \frac{1}{n} \sum_{i=1}^n \nabla G\p{K_h*\p{\frac{\delta_{Y_i}}{\rho_h(Y_i)}}} -  \E \left[ \nabla  G\p{K_h*\p{\frac{\delta_{Y_i}}{\rho_h(Y_i)}}}\right]\right\|_{L_p (M)}.
\end{split}
\end{equation}
The expectation of the $L_p$-norm of the sum of i.i.d.~centered functions is controlled using Rosenthal inequality.
\begin{lemma}\label{lem:rosenthal}
Let $U_1,\dots,U_n$ be i.i.d. functions on $L_p(M)$. Then, $\E \left\|\frac{1}{n}\sum_{i=1}^n (U_i-\E U_i)\right\|_{L_p(M)}^p$ is smaller than
\begin{equation}\label{eq:rosenthal}
 \begin{cases}
n^{-p/2} \int \p{\E \left[\vert U_1(z)\vert ^2\right]}^{p/2}\dd z &\text{ if $p\leq  2$},\\
C_pn^{-p/2} \int \p{\E \vert U_1(z)\vert ^2}^{p/2}\dd z + C_p n^{1-p}\int_M \E \left[\vert U_1(z)\vert ^p\right]\dd z  & \text{ if $p>2$}.
\end{cases}
\end{equation}
\end{lemma}

\begin{proof}
If $p\leq 2$, one has by Jensen's inequality
\[ \E\left\vert \sum_{i=1}^n (U_i(z)-\E U_i(z)) \right \vert ^p \leq \p{\E\left\vert \sum_{i=1}^n (U_i(z)-\E U_i(z)) \right \vert ^2}^{p/2} \leq n^{p/2}\p{\E \vert U_1(z)\vert ^2}^{p/2}\]
 and \eqref{eq:rosenthal} follows by integrating this inequality against $z\in M$. For $p>2$, we use Rosenthal inequality \cite[Theorem 3]{rosenthal1970subspaces} for a fixed $z\in M$, and then integrate the inequality against $z\in M$.
\end{proof}

It remains to bound $\E\left[\left\vert \nabla  G\p{K_h*\p{\frac{\delta_{Y}}{\rho_h(Y)}}}(z)\right \vert ^p\right]$  where $Y\sim \mu$, $z\in  M$ and $p\geq 2$.

\begin{lemma}\label{lem:fluc2}
Let $p\geq 2$. Then, for all $z\in M$ and $h\lesssim 1$,
\begin{equation}
\E\left[\left\vert \nabla  G\p{K_h*\p{\frac{\delta_{Y}}{\rho_h(Y)}}}(z)\right \vert ^p\right] \lesssim \begin{cases}
1 & \text{ if } d=1\\
-\log h & \text{ if } p=d=2\\
h^{p+d-dp} & \text{ else}.
\end{cases}
\end{equation}
\end{lemma}
A proof of Lemma \ref{lem:fluc2} is found in Appendix \ref{sec:proof_fluctuation}. From \eqref{eq:fluc1}, Lemma \ref{lem:rosenthal} and Lemma \ref{lem:fluc2}, we obtain, in the case $p\geq 2$ and $d\geq 3$
\begin{align*}
 \E\|f_{n,h}-f_h&\|_{H^{-1}_p(M)} \leq \p{\E\|f_{n,h}-f_h\|_{H^{-1}_p(M)}^p}^{1/p} \\
 &\leq C_p n^{-1/2} \p{\int \p{\E \left\vert \nabla  G\p{K_h*\p{\frac{\delta_{Y}}{\rho_h(Y)}}}(z)\right \vert ^2}^{p/2}\dd z}^{1/p} \\
 &\qquad + C_p n^{1/p -1}\p{\int \E \left[\left\vert \nabla  G\p{K_h*\p{\frac{\delta_{Y}}{\rho_h(Y)}}}(z)\right \vert ^p \right]\dd z}^{1/p} \\
 &\lesssim n^{-1/2} \vert \vol_M\vert ^{1/p} h^{1-d/2} +  n^{1/p -1}\vert \vol_M\vert ^{1/p} h^{1+d/p-d}. 
\end{align*}
 Recalling that $\vert \vol_M\vert \leq \fmin^{-1}\lesssim 1$ and that $nh^d\gtrsim 1$, one can check that this quantity is smaller up to a constant than $n^{-1/2}h^{1-d/2}$, proving Proposition \ref{prop:fluc} in the case $p\geq 2$ and $d\geq  3$.  A similar computation shows that  Proposition \ref{prop:fluc} also holds if $p\leq 2$ or $d\leq 2$.

\subsection{Proof of Theorem \ref{thm:estimator_M_known}}
We are now ready to conclude the proof of Theorem \ref{thm:estimator_M_known}. 
The proof of  \ref{it:pointwise_control} follows from a standard control of the $\infty$-norm between $f$ and $f_{n,h}$ and is found in Appendix \ref{sec:pointwise_proof}. 
Point \ref{it:sobo_control} is a combination of the results from the two previous sections, and we can obtain \ref{it:wass_control} using the linearization inequality (Proposition \ref{prop:wass_neg}). 
More precisely, let $E$ be the event described in \ref{it:pointwise_control}. If $E$ is realized, then $\mu_{n,h}^0$ is equal to $\mu_{n,h}$, and it satisfies $\mu_{n,h}^0 \geq \frac{\fmin}{2}\vol_M$.
 Thus, Proposition \ref{prop:wass_neg} yields $W_p(\mu_{n,h}^0,\mu) \lesssim  \left\|\mu_{n,h}-\mu\right\|_{H^{-1}_p(M)}$.
If $E$  is not realized, we bound $W_p(\mu_{n,h}^0,\mu)$ by $\diam(M)$, which is itself bounded by a constant depending only on the parameters of the model (see \cite[Lemma 2.2]{aamari2018stability}). Hence,
\begin{align*}
\E W_p(\mu_{n,h}^0,\mu) &\leq \E\left[ W_p(\mu_{n,h}^0,\mu) \ones\{E\}\right] + \diam(M)\P(E^c)\\
& \lesssim \E\|\mu_{n,h}-\mu\|_{H^{-1}_p(M)}  + n^{-k/d},
\end{align*}
and we conclude thanks to \eqref{eq:decomposition_variance}. 
Eventually, a proof of the minimax lower bound \ref{it:noiseless_minimax}, based on Assouad's lemma, is given in Appendix \ref{sec:lowerbounds}.

\subsection{Proofs of Theorem \ref{thm:estimator_volume} and Theorem \ref{thm:M_unknown}}\label{sec:estim_volume}

There are three different statement to prove in Theorem \ref{thm:estimator_volume}. The first one is a direct application of the construction of partitions of unity proposed in Lemma \ref{lem:existence_pou}.
\medskip

\paragraph*{}\textit{Proof of Theorem \ref{thm:estimator_volume}\ref{it:pou}.} 

Assume that $\gamma\leq \eps/24$. Let $\X=\{X_1,\dots,X_n\}$ and $\Y=\{Y_1,\dots,Y_n\}$. By the remark following Lemma \ref{lem:existence_pou}, the existence of a partition of unity satisfying the requirements of Theorem \ref{thm:estimator_volume}\ref{it:pou} is ensured as long as $d_H(M^{\eps/8}\vert  \X)\leq 5\eps/24$. We have $d_H(M^{\eps/8}\vert \X)\leq d_H(M^{\eps/8}\vert \Y)+\eps/24 \leq d_H(M\vert \Y)+4\eps/24$. Hence, the partition of unity exists if $d_H(M\vert \Y)\leq \eps/24$. This is satisfied with probability larger than $1-cn^{-k/d}$ if $\eps \gtrsim (\log n/n)^{1/d}$ by \cite[Lemma III.23]{aamari2017vitesses}.
\medskip

\paragraph*{}\textit{Proof of Theorem \ref{thm:estimator_volume}\ref{it:control_estimator_volume}.}
For ease of notation, we will assume that the output $\{X_{i_1},\dots,X_{i_J}\}$ of the farthest point sampling algorithm is equal to $\{X_1,\dots,X_J\}$. The measure $\widehat \vol_M$ can be written as $\sum_{j=1}^J \xi_j$, where $\xi_j$ is the measure having density $\chi_j$ with respect to the $d$-dimensional Hausdorff measure on $\hat\Psi_{j}(\hat T_j)$. 

Remark that the restriction of each function $\chi_j$ on $M$ is supported on a small neighborhood covered by the chart $\Psi_{Y_{j}}$. We may therefore write for any continuous bounded function $\phi$,
\[ \int \phi(x)\dd \vol_M(x) = \sum_{j=1}^J \int_{\Psi_{Y_j}(T_{Y_j})}\phi(x) \chi_j(x)\dd x.\]
As $\hat \Psi_{j}$ is at distance $\eps^k + \gamma$ from $\Psi_{Y_j}$ and $\chi_j$ is Lipschitz, one can  hope that the measure $\widehat \vol_M$ defined in \eqref{eq:estimator_volume} is also at distance $\eps^k + \gamma$ from $\vol_M$. The Wasserstein distance between $\xi_j$ and the measure $\chi_j\cdot\vol_M$ is bounded in two steps. First, we transport $\xi_j$ on $M$ by using the parametrizations $\hat \Psi_j$ and $\Psi_{Y_j}$. When doing so, we obtain a measure on $M$, whose density is a modification of $\chi_j$, that is distorted by the Jacobian of the transport map. We then crudely bound the Wasserstein distance between this new measure and $\chi_j\cdot\vol_M$ by their $L_1$-distance. We first need a technical result that controls how the density $\chi_j$ is impacted  by the transport map.

\begin{lemma}\label{lem:pointwise_comparison}
If $(\log n/n)^{1/d}\lesssim \eps \lesssim 1$ and $\gamma\lesssim \eps$, with probability larger than $1-cn^{-k/d}$, for all $j=1,\dots,J$:
\begin{enumerate}[label=(\alph*)]
\item The map $\Psi_{Y_j}\circ\pi_{Y_j}:\BB_{\hat T_j}(0,3\eps)\to M$ is a diffeomorphism on its image, which contains $\BB_M(Y_j,2\eps)$. Let $S_j:\BB_M(Y_j,2\eps)\to \BB_{\hat T_j}(0,3\eps)$ be the inverse of $\Psi_{Y_j}\circ\pi_{Y_j}$. Then, $\hat \Psi_j\circ S_j:\BB_M(Y_j,2\eps)\to \hat\Psi_j(\hat T_j)$ is also a diffeomorphism on its image, which contains $\BB_{\hat\Psi_j(\hat T_j)}(X_j,\eps)$. Furthermore, for all $z\in \BB_M(Y_j,2\eps)$, we have $\vert \hat\Psi_j \circ S_j(z)-X_j\vert \geq \frac{7}{8}\vert z-Y_j\vert $. \label{it:comp_diffeo}
\item The measure $(\hat \Psi_j\circ S_j)^{-1}_{\#}\xi_j$ has a density $\t\chi_j$ on $M$ equal to 
\begin{equation}
\t\chi_j(z)= \chi_j(\hat \Psi_j\circ S_j(z))J(\hat \Psi_j\circ S_j)(z) \text{ for $z\in M$},
\end{equation}
where the function is extended by $0$ for $z\in M \backslash \BB_M(Y_j,2\eps)$.\label{it:comp_diffeo2}
\item \label{it:comp_diffeo3} For $z\in \BB_M(Y_j,2\eps)$, we have%, if $\gamma \lesssim \eps^2$,
\begin{align}
&\vert \hat \Psi_j\circ S_j(z)-z\vert \lesssim \eps^{m} + \gamma, \label{eq:first_eq}\\
&\vert \t\chi_j(z)-\chi_j(z)\vert \lesssim (\eps^m+\gamma)(1+\gamma \eps^{-2}).\label{eq:third_eq}
\end{align}
\end{enumerate}
\end{lemma}

A proof of Lemma \ref{lem:pointwise_comparison} is found in Appendix \ref{sec:last_thm}. Let $\hat M_\eps = \bigcup_{j=1}^J \BB_{\hat\Psi_j(\hat T_j)}(X_j,\eps)$ be the support of $\widehat\vol_M$. We are now ready to state a stability result between the approximated measure $\widehat \vol_M$ and $\vol_M$. Note that we state a lemma that is slightly more general (considering measures having a density $\phi$ with respect to the volume measures), so that we can also control the distance between the kernel density estimators of Theorem \ref{thm:M_unknown}.

\begin{figure}
\centering
\includegraphics[width=\textwidth]{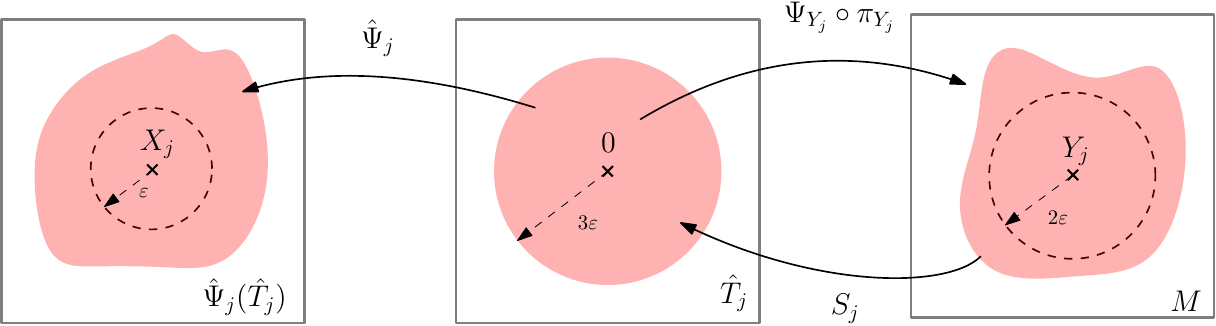}
\caption{Illustration of Lemma \ref{lem:pointwise_comparison}\ref{it:comp_diffeo}}
\end{figure}

\begin{lemma}\label{lem:the_smart_lemma}
Let $(\log n/n)^{1/d}\lesssim \eps \lesssim 1$ and $\gamma\lesssim\eps$. Fix $1\leq r \leq\infty$ and let  $\phi:M \to \R$, $\t\phi:\hat M_\eps\to \R$  be functions satisfying $\phi_{\min}\leq \phi,\t\phi \leq \phi_{\max}$ for some positive constants $\phi_{\min},\phi_{\max}>0$. Assume further that for all $j=1,\dots,J$ and for all $z\in M$ we have, $\vert \t\phi(\hat \Psi_j\circ S_j(z))-\phi(z)\vert \leq T \lesssim 1$. Then, with probability larger than $1-cn^{-k/d}$, we have
\begin{equation}\label{eq:the_smart_eq}
W_r\p{\frac{\t\phi\cdot \widehat\vol_M}{\vert \t\phi\cdot\widehat\vol_M\vert },\frac{\phi\cdot \vol_M}{\vert \phi\cdot\vol_M\vert }} \lesssim C_0 (T+(\eps^m+\gamma)(1+\gamma \eps^{-2})),
\end{equation}
where $C_0$ depends on $\phi_{\min}$ and $\phi_{\max}$.
\end{lemma}
In particular, inequality \eqref{eq:bound_wass_volume} is a consequence of Lemma \ref{lem:the_smart_lemma} with $\phi\equiv\t\phi \equiv 1$.

\begin{proof}
 Assume first that $r<\infty$. %If $(\eps^{m} + \gamma) \gtrsim 1$, there is nothing to prove, so we may assume that $(\eps^{m} + \gamma)\leq 1/(2Cc_d)$, where $c_d$ is the constant of Lemma \ref{lem:existence_pou} and $C\defeq \sup_{z\in M}\vert \t\chi_j(z)-\chi_j(z)\vert /(\eps^m+\gamma)$. 
 We have the bound

\begin{align}
W_r\p{ \frac{\t\phi\cdot\widehat\vol_M}{\vert \t\phi\cdot\widehat\vol_M\vert }, \frac{\phi\cdot\vol_M}{\vert \phi\cdot\vol_M\vert }} &= \frac{1}{\vert \t\phi\cdot\widehat\vol_M\vert ^{1/r}}W_r\p{ \t\phi\cdot\widehat\vol_M,\phi\cdot\vol_M \frac{\vert \t\phi\cdot\widehat\vol_M\vert }{\vert \phi\cdot\vol_M\vert }} \nonumber \\*
&\hspace{-4cm}\leq \frac{1}{\vert \t\phi\cdot\widehat\vol_M\vert ^{1/r}}\Bigg( W_r\p{\sum_{j=1}^J \t\phi\cdot\xi_j, \sum_{j=1}^J(\hat \Psi_j\circ S_j)^{-1}_{\#}(\t\phi\cdot\xi_j)}\nonumber  \\*
&\hspace{-4cm}+ W_r\p{\sum_{j=1}^J (\hat \Psi_j\circ S_j)^{-1}_{\#}(\t\phi\cdot\xi_j) ,\phi\cdot\vol_M\frac{\vert \t\phi\cdot\widehat\vol_M\vert }{\vert \phi\cdot\vol_M\vert }} \Bigg)\label{eq:wass_volume}
\end{align}

 We use Proposition \ref{prop:wass_neg} to bound the second term in \eqref{eq:wass_volume}. By a change of variables, the density of $(\hat \Psi_j\circ S_j)^{-1}_{\#}(\t\phi\cdot\xi_j)$ is given by $\t\phi_j:z\mapsto\t\phi(\hat \Psi_j\circ S_j(z))\t\chi_j(z)$. With probability larger than $1-cn^{-k/d}$, we have for $z\in M$, should $(\eps^m+\gamma)(1+\gamma\eps^{-2})$ be small enough,
\begin{align*}
\sum_{j=1}^J \t\chi_j(z)\geq \sum_{j=1}^J\chi_j(z) - Cc_d(\eps^{m} + \gamma)(1+\gamma\eps^{-2}) \geq 1 - \frac 1 2= \frac 1 2,
\end{align*}
where $c_d$ is the constant of Lemma \ref{lem:existence_pou}. 
Therefore, the density of $\sum_{j=1}^J (\hat \Psi_j\circ S_j)^{-1}_{\#}(\t\phi\cdot\xi_j)$ is larger than $\phi_{\min}/2$. Remark also that $\t\chi_j(z)\leq 2$ for any $z\in M$. Hence, we have according to Lemma \ref{lem:pointwise_comparison}, $\vert \t\phi_j(z)-\phi(z)\chi_j(z)\vert \leq T+2\phi_{\max}\vert \chi_j(z)-\t\chi_j(z)\vert \lesssim T+\phi_{\max}(\eps^{m} + \gamma)(1+\gamma\eps^{-2})$ for some constant $C_0$. This gives the bound,
\begin{align}
\vert \vert \t\phi\cdot\widehat\vol_M\vert -\vert \phi\cdot\vol_M\vert \vert  &\leq \left\|\sum_{j=1}^J \t\phi_j-\phi \right\|_{L_1(M)} \leq \left\|\sum_{j=1}^J \t\phi_j-\phi \right\|_{L_r(M)}\vert \vol_M\vert ^{1-1/r} \nonumber\\
& \leq \left\|\sum_{j=1}^J \t\phi_j-\phi \right\|_{L_\infty(M)}\vert \vol_M\vert  \nonumber\\
&\leq C_0\vert \vol_M\vert ( T+\phi_{\max}(\eps^{m} + \gamma)(1+\gamma\eps^{-2})).\label{eq:error_volume}
\end{align}
Therefore, $\phi\frac{\vert \t\phi\cdot\widehat\vol_M\vert }{\vert \phi\cdot\vol_M\vert }$ is larger than 
\begin{align*}
&\phi_{\min}\p{1-C_0\vert \vol_M\vert \frac{T+\phi_{\max}(\eps^{m} + \gamma)(1+\gamma\eps^{-2})}{\phi_{\min}\vert \vol_M\vert }} \\
&\qquad \geq \phi_{\min}- C_0(T+\phi_{\max}(\eps^{m} + \gamma)(1+\gamma\eps^{-2})) \geq \frac{\phi_{\min}}{2}
\end{align*} 
if $T, \eps$ and $\gamma$ are small enough. Hence, by Proposition \ref{prop:wass_neg} and using \eqref{eq:error_volume},
\begin{align*}
&W_r\p{\sum_{j=1}^J (\hat \Psi_j\circ S_j)^{-1}_{\#}(\t\phi\cdot\xi_j) ,\phi\cdot\vol_M\frac{\vert \t\phi\cdot\widehat\vol_M\vert }{\vert \phi\cdot\vol_M\vert }}  \\
&\leq r^{-1/r}\p{\frac{2}{\phi_{\min}}}^{1-1/r} \left\| \sum_{j=1}^J \t\phi_j - \phi\frac{\vert \t\phi\cdot\widehat\vol_M\vert }{\vert \phi\cdot\vol_M\vert }\right\|_{H^{-1}_r(M)}\\
&\leq \p{\frac{2}{\phi_{\min}} \vee 1}\left\| \sum_{j=1}^n \t\phi_j - \phi\frac{\vert \t\phi\cdot\widehat\vol_M\vert }{\vert \phi\cdot\vol_M\vert }\right\|_{L_r(M)} \\
&\leq \p{\frac{2}{\phi_{\min}} \vee 1}\p{\left\| \sum_{j=1}^J \t\phi_j - \phi \right\|_{L_r(M)} + \frac{\vert \vert \phi\cdot\vol_M\vert -\vert \t\phi\cdot\widehat\vol_M\vert \vert }{\vert \phi\cdot\vol_M\vert }\|\phi\|_{L_r(M)}}\\ 
&\leq  \p{\frac{2}{\phi_{\min}} \vee 1}C_0( T+\phi_{\max}(\eps^{m} + \gamma)(1+\gamma\eps^{-2})) \p{\vert \vol_M\vert ^{1/r}+ \frac{\vert \vol_M\vert }{\phi_{\min}\vert \vol_M\vert }\vert \vol_M\vert ^{1/r}\phi_{\max}}\\
&\leq C_{\phi_{\min},\phi_{\max}}\p{ T+ (\eps^{m} + \gamma)(1+\gamma\eps^{-2})},
\end{align*}
where we used that $\vert \vol_M\vert \leq \fmin^{-1}\lesssim 1$, and the constant $C_{\phi_{\min},\phi_{\max}}$ in the upper bound depending on $\phi_{\min}$ and $\phi_{\max}$, but not on $r$.

To bound the first term in \eqref{eq:wass_volume}, consider the transport plan 
$\sum_{j=1}^J (\id,(\hat \Psi_j\circ S_j)^{-1})_{\#}(\t\phi\cdot\xi_j),$
 which has, according to Lemma \ref{lem:pointwise_comparison}, a cost bounded  by 
 \[\sum_{j=1}^J \int \vert y-(\hat \Psi_j\circ S_j)^{-1}(y)\vert ^r\dd (\t\phi\cdot\xi_j)(y)\lesssim \phi_{\max}\p{ \eps^m + \gamma }^r\vert \widehat\vol_M\vert .\]
As $\vert \widehat\vol_M\vert  \lesssim \vert \vol_M\vert  + T+\phi_{\max}(\eps^{m} + \gamma)(1+\gamma\eps^{-2})\lesssim 1$, we obtain the desired bound. By letting $r\to \infty$, and remarking that the different constants involved are independent of $r$, we observe that the same bound holds for $r=\infty$.
\end{proof}
%Note also that inequality \eqref{eq:error_volume} with $\phi\equiv \t\phi \equiv 1$ is exactly the second inequality in \eqref{eq:bound_wass_volume}.
\begin{remark}
Inequality \eqref{eq:error_volume} with $\phi\equiv\phi'\equiv 1$ gives a bound on the distance between the total mass of $\widehat\vol_M$ and the volume $\vert \vol_M\vert $ of $M$: choosing $k=m$, it is of order $(\eps^k+\gamma)(1+\gamma\eps^{-2})$ with probability larger than $1-cn^{-k/d}$.
\end{remark}

\paragraph*{}\textit{Proof of Theorem \ref{thm:estimator_volume}\ref{it:minimax_volume}.}

Inequality \eqref{eq:risk_volume_estimator} is a consequence of Theorem \ref{thm:estimator_volume}\ref{it:control_estimator_volume}, whereas the lower bound on the minimax risk \eqref{eq:minimax_risk_volume} is proven in Appendix \ref{sec:lowerbounds}.
\medskip

\paragraph*{}\textit{Proof of Theorem \ref{thm:M_unknown}.}

Note first that $\hat\nu_{n,h}$ is indeed a measure of mass $1$. We show in Lemma \ref{lem:bound_on_T} that 
\[ T\defeq \max_{j=1\dots J}\sup_{z\in \BB(Y_j,\eps)}\left\vert  K_h*\p{\frac{\nu_n}{\hat\rho_h}}(\hat\Psi_j\circ S_j(z)) -K_h*\p{\frac{\mu_n}{\rho_h}}(z) \right \vert \]
 satisfies $T\lesssim (\eps^m + \gamma)(1+\gamma\eps^{-2})$ with probability larger than $1-cn^{-k/d}$, as long as $nh^d\gtrsim 1$.
 As $\fmin/2\leq K_h*\mu_n\leq 2\fmax$ on $M$ by Theorem \ref{thm:estimator_M_known}\ref{it:pointwise_control}, and as every $y\in \hat M_\eps$ is in the image of $\hat\Psi_j\circ S_j$ for some $j=1\dots J$, we have $\fmin/3\leq K_h*\nu_n\leq 3\fmax$ on $\hat M_\eps$ should $\eps^k+\gamma$ be small enough. This proves Theorem \ref{thm:estimator_M_known}\ref{it:nonnegative} and, together with Lemma \ref{lem:the_smart_lemma}, this also proves Theorem \ref{thm:M_unknown}\ref{it:control_estimator_volume}. Theorem \ref{thm:M_unknown}\ref{it:rate_final_estimator} is a consequence of Theorem \ref{thm:M_unknown}\ref{it:control_estimator_volume}.

\section{Numerical considerations}\label{sec:num}

\subsection{Computation of the local polynomial estimator}\label{sec:conv_opt}
A crucial step in the implementation of the estimators $\widehat \vol_M$ and $\hat \nu_{n,h}$ is the minimization procedure described in \eqref{eq:def_poly}. Assume for the sake of simplicity that no noise is present in the dataset. Further assume without loss of generality that $0\in M$, and let $X_1, \dots, X_N$ be $N$ points sampled according to some distribution $\tilde \mu$ supported on $\BB_M(0,\eps)$, with density lower bounded by $\fmin$ (in the setting of Section \ref{sec:def_estim}, $\tilde \mu$ is the conditional probability of $X\sim\mu$ given that $X\in\BB_M(0,\eps)$, where $\mu\in \QQ^s(M)$). Our goal is to find a minimizer of the functional 
\begin{equation}\label{eq:naive_functional}
\frac{1}{2N}\sum_{i=1}^N \left\vert X_i-\pi(X_i)-\sum_{j=2}^{m-1} V_j[\pi(X_i)^{\otimes j}]\right \vert ^2,
\end{equation}
where $\pi$ is an orthogonal projector of rank $d$ and each $V_j$ is a $j$-tensor with operator norm smaller than $\ell$.  We propose here a fast procedure with theoretical guarantees to solve this problem, answering a question raised in \cite{aamari2019nonasymptotic}. There are two issues that make this optimization problem not trivial. First, the objective functional is defined on a manifold, and second, it is not globally geodesically convex. We will actually show that the functional is $\lambda$-strongly geodesically convex and $\beta$-smooth on a small neighborhood of size $\eps$ around its minimizer, for both $\lambda$ and $\beta$ of order $\eps^2$. Furthermore, we show that it is possible to find a point in this neighborhood, by letting $V_j=0$, and $\pi$ being given by a PCA on the dataset $X_1,\dots,X_N$. Any standard optimization algorithm on Riemannian manifolds will then converge with such an initialization. As an example, we show that a classical gradient descent converges linearly, although it is expected that more refined algorithms such as a Riemannian SVRG will also converge, with possibly a better behavior in practice \cite{sato2019riemannian}. Before going further, we give some background on the Grassmannian manifold that will allow us to rewrite the objective functional in a more practical way. We refer to \cite{edelman1998geometry} for a more detailed introduction on the geometry of the Grassmannian. We also provide a short review of convex optimization on Riemannian manifolds in Appendix \ref{sec:det_num}.
\medskip

 Let $\OO(D)$ be the manifold of orthogonal matrices on $\R^D$, that we endow with the Riemannian structure induced by the inclusion in $\R^{D\times D}$. For $Q\in \OO(D)$, the tangent space of $\OO(D)$ at $Q$ is given by $T_Q \OO(D) = \{QA:\ A\text{ skew-symmetric}\}$, whereas the exponential map is given by $\exp_{Q}(QA) = Q\exp(A)$.
\medskip

 Let $\GG(d,D)$ be the Grassmannian manifold. It can be defined as the quotient $\OO(D)/(\OO(d)\times \OO(D-d))$: we identify a $d$-dimensional subspace $U$ of $\R^D$ with the set of matrices 
 \begin{equation}
 \left\{Q\begin{pmatrix}
 Q_1 & 0 \\ 0 & Q_2
 \end{pmatrix}:\ Q_1\in \OO(d), Q_2\in \OO(D-d)\right\},
 \end{equation}
 where $Q$ is any orthogonal matrix such that the vector space $\Pi(Q)$ spanned by its first $d$-columns is equal to $U$. An orthogonal projector of rank $d$ is then identified with $\pi_U$, the projection on the subspace $U$. The manifold $\GG(D,d)$ is endowed with the Riemannian structure given by the quotient map. Fix an element $U_*\in\GG(d,D)$, with an arbitrary representant $Q_*$. Then, the tangent space $T_{U_*}\GG(d,D)$ at $U_*$ is given by the set of matrices
 \begin{equation}
  \Delta = Q_*\begin{pmatrix}
0 & -B^\top  \\
B & 0
\end{pmatrix},
\end{equation} 
where $B$ is any $(D-d)\times d$ matrix. The exponential map on the Grassmannian is given by the quotient projection of the exponential on the orthogonal group, that is $\exp_{U_*}(\Delta)=\Pi\circ \exp_{Q_*}(\Delta)$. The exponential map being surjective, any element of the Grassmannian can be written as $\Pi\circ \exp_{Q_*}(\Delta)$ for some $\Delta\in T_{U_*}\GG(d,D)$. Furthermore, the quotient map $\Pi: \exp_{Q_*}( T_{U_*}\GG(d,D)) \to \GG(d,D)$ is a local isometry. As it will be convenient for us to work with deformations of a fixed orthogonal basis $Q_*$ of $\R^D$, we will consider the $d(D-d)$-dimensional submanifold $\OO_*(d,D) \defeq \exp_{Q_*}( T_{U_0}\GG(d,D))$ of $\OO(D)$ instead of $\GG(d,D)$. A small neighborhood of $U_*$ in $\GG(d,D)$ is isometric to a neighborhood of $Q_*$ in $\OO_*(d,D)$, so that both points of view are equivalent. Given $Q\in \OO_*(d,D)$, we write $Q= (Q_{[d]}\ Q_{[d,D]})$, where $Q_{[d]}=(e_1,\dots,e_d)$ is a $D\times d$ matrix, and $Q_{[d,D]}=(e_{d+1},\dots,e_D)$ is a $D\times (D-d)$ matrix.
\medskip

  Given $E$ and $F$ two vector spaces, we let $\SS^j(E,F)$ be the set of symmetric $j$-tensors from $E$ to $F$. Then, the functional in \eqref{eq:naive_functional} is defined on $\GG(d,D)\times \prod_{j=2}^{m-1} \SS^j(\R^D,\R^D)$, that is a manifold of possibly very large dimension (of order $D^{m}$). However, given a subspace $U\in \GG(d,D)$, one can always decrease the loss function by replacing a tensor $V_j$ by the tensor $\pi_U^\bot\circ V_j\circ \pi_U$. That is, we may assume that $V_j\in \SS^j(U,U^\bot)$. For $Q=(e_1,\dots,e_D)\in \OO_*(d,D)$ and $V_j = (V_{j,k})_{d+1\leq k \leq D} \in \SS_j(\R^d,\R^{D-d})$, we let $\iota_j(Q,V)$ be the $j$-tensor in $\SS_j(\Pi(Q),\Pi(Q)^\bot)$ defined by
\[\iota_j(Q,V)[x_1,\dots,x_j] = \sum_{k=d+1}^D e_k V_{j,k}[Q_{[d]}^\top  (x_1 \cdots x_j)]\] 
for $x_1,\dots,x_j\in \Pi(Q)^j$. Introduce the functional
\begin{equation}
G_{m,x}(Q,V) \defeq \frac{1}{2} \left\vert x-\sum_{k=1}^d \dotp{x,e_k}e_k -\sum_{j=2}^{m-1} \iota_j(Q,V)[x^{\otimes j}] \right\vert^2.
\end{equation}
The function $G_m \defeq \frac{1}{N} \sum_{i=1}^N G_{m,X_i}$ is defined on the manifold  
\[\MM \defeq \OO_*(d,D)\times \prod_{j=2}^{m-1}\SS_j(\R^d,\R^{D-d}).\] Note that $\MM$ is a manifold of dimension $(D-d)\sum_{j=1}^{m-1} \binom{d+j-1}{j}$, that is of order $Dd^{m-1}\ll D^m$. We endow $\MM$ with the Riemannian metric $g_{\MM}$ given by 
\begin{equation}\label{eq:def_metric}
(g_{\MM})_{(Q,V)}((\Delta_1,W_1),(\Delta_2,W_2)) = \dotp{\Delta_1,\Delta_2} + \sum_{j= 2}^{m-1} \eps^{2(j-1)}\dotp{W_{1,j},W_{2,j}},
\end{equation}
for $(Q,V)\in\MM$ and $(\Delta_1,W_1),(\Delta_2,W_2)\in T_{(Q,V)}\MM$. We also denote by $d_{\MM}$ the geodesic distance on $\MM$. Consider the submanifold $\MM_0$ of $\MM$, where we impose that each $j$-tensor $V_j$ has operator norm smaller than $\ell$. Then, solving \eqref{eq:naive_functional} is equivalent to minimizing $G_m$ on $\MM_0$.

Let us choose $(Q_*,V_*)$ that minimizes $G_m$ on $\MM_0$. Consider the (geodesically convex) neighborhood $\Omega$ of $(Q_*,0)$ in $\MM_0$ given by matrices $Q\in  \OO_*(d,D)$ that are $r$ close from $Q_*$ for the geodesic distance.
 
 \begin{proposition}\label{prop:strong_convexity}
 Let $\eps\lesssim 1$, $\delta >0$, and let $r\leq \delta\eps$. If $1\lesssim \ell \lesssim \eps^{-1}$ is small enough with respect to $\delta$, then, with probability $1-e^{-cN}$, the functional $G_m$ is geodesically $\lambda$-strongly convex and $\beta$-smooth on $\Omega$ with $\eps^{2} \lesssim \lambda \leq \beta \lesssim \eps^2$, where the constant $c$ depends on $d$, $m$, $\taumin$ and $\fmin$.
 \end{proposition}
This implies in particular that a gradient descent with step $\alpha$ of order $\eps^{-2}$ will converge at linear rate towards the minimizer of $G_m$, when initialized in $\Omega$ (see Appendix \ref{sec:det_num} for details). A number of steps of order $\vert \log \eps\vert  \simeq \log n$ is needed to attain a point at distance $\eps^{m-1}$ from the minimizer. Such a point $(Q,V)$ will then satisfy the same inequalities than the minimizer $(\hat \pi,\hat V)$ from Proposition \ref{prop:estim_manifold}. Proposition \ref{prop:estim_manifold} with $m=2$ implies that a local PCA yields a subspace that is $\delta\eps$-close from $\Pi(Q^*)$ for some $\delta>0$. Therefore, choosing $\ell$ small enough, Proposition \ref{prop:strong_convexity} implies that one can initialize the gradient descent at the local PCA, with linear convergence of the gradient descent.

Let us end this section with a word about the computational complexity of gradient descent. Evaluating the gradient of $G_m$ requires $O(NDd^{m-1})$ operations. Also, one need to compute the exponential map on $\MM$ at each step of the gradient descent. This boils down to computing a SVD, which can be made in $O(Dd^2)$ time. In total, each step of the gradient descent takes $O(NDd^{m-1})$ time. The dependence in $N$ is not an issue as $N$ is of order $\log n$ in our setting (that is the expected number of points of a $n$-sample in a ball of radius $\eps \simeq (\log n/n)^{1/d}$). Still, one could use a stochastic gradient descent algorithm to remove the $N$ factor if needed.

\subsection{Sampling from the estimators}

The issue of evaluating integrals and sampling from distributions supported  defined on manifolds has been addressed in several works using Monte Carlo methods: e.g. authors in \cite{diaconis2013sampling} propose rejection sampling and Gibbs sampling methods in the case where the manifold can be covered by a single known chart (except on a set of null measure), while authors in \cite{zappa2018monte} design a MCMC in the case where the manifolds is defined through equality constraints. 
We here address the problem of sampling from two estimators: the estimated uniform measure $\hat U_M\defeq \widehat\vol_M/\vert \widehat\vol_M\vert $ and the measure $\hat \mu_{n,h}$. Our goal is not to propose state-of-the-art procedures but to show that basic sampling algorithms already have good theoretical behaviors. 
 
Let 
\begin{equation}\label{eq:qj}
q_j \defeq \hat U_M(\hat\Psi_{i_j}(\hat T_{j})) = \int_{\hat T_j} \chi_j(\hat\Psi_j(u))J\hat\Psi_j(u)\dd u
\end{equation}
and $\hat U_M^{(j)} = q_j^{-1}\cdot (\hat U_M)_{\vert \hat \Psi_{j}(\hat T_{j})}$, the measure $\hat U_M$ conditioned on being in $\hat \Psi_{j}(\hat T_{j})$. Each $q_j$ can be estimated by a Monte Carlo method, for instance by sampling uniform points on $\BB_{\hat T_{j}}(0,\eps)$. As the integrand in \eqref{eq:qj} is bounded by $2$ (for $\eps$ small enough), the variance of the estimator of the integral can be bounded, and $\delta^{-2}$ samples are necessary to obtain a precision $\delta$. The volume $\vert \widehat \vol_M\vert $ can then be approximated by the sum of the $q_j$s. 
A $N$-sample of law $\hat U_M$ is built in two steps. First, by sampling $(N_1,\dots,N_J)$ that follows a multinomial distribution of parameters $n$ and $(q_1,\dots,q_J)$. Second, by sampling $N_j$ points from $\hat U_M^{(j)}$ for every $j$. 
  We propose the following method to produce a sample with approximate distribution $\hat U_M^{(j)}$. Let $\t U_M^{(j)}$ be the pushforward measure  of the uniform distribution on $\BB_{\hat T_{j}}(0,\eps)$ by $\hat\Psi_{j}$.
It is immediate to simulate from $\t U_M^{(j)}$, while $\hat U_M^{(j)}\ll \t U_M^{(j)}$, with density $\frac{d\hat U_M^{(j)}}{d\t U_M^{(j)}}$ being equal to, for $u\in \hat T_{j}$,
\begin{equation}\label{eq:density_tilde_U}
\frac{d\hat U_M^{(j)}}{d\t U_M^{(j)}}(\hat\Psi_{j}(u)) = \frac{\omega_d\eps^d}{q_j}\chi_j(\hat\Psi_{j}(u))J\hat\Psi_{j}(u).
\end{equation} 
Rejection sampling with proposal distribution $\t U_M^{(j)}$ then allows one to create a sample from law $\hat U_M^{(j)}$. One can check that the density  in \eqref{eq:density_tilde_U} is upper bounded by some constant depending only on $d$ (for $\eps$ small enough), so that the acceptance ratio of the procedure is bounded away from zero.

Another quantity of interest is the number $N$ of samples that are needed to obtain a good approximation of $\hat U_M$ based on the empirical distribution $(\hat U_M)_N$  of a $N$-sample from law $\hat U_M$. It is known \cite{trillos2020error} that if a measure is supported on a $d$-dimensional manifold, then the $W_\infty$-distance between a measure and a $N$-sample is of order $(\log N/N)^{1/d}$ (for $d\geq 3$). However, $\hat U_M$ is not supported on a manifold, but on a union of overlapping polynomial patches. We however show that the expected rate of convergence still holds in this case. A proof is provided in Appendix \ref{sec:last_thm}.

 \begin{proposition}\label{prop:hat_u_m_approx}
 Let $\hat U_M$ be the estimator of the uniform measure built on $n$ points sampled from $\mu \in \QQ^{0,k}_d$.  Let $W_1,\dots,W_N$ be a $N$-sample of law $\hat U_M$, with associated empirical measure $(\hat U_M)_N$. Then, 
 \begin{equation}\label{eq:hat_u_m_approx}
 \E[W_\infty((\hat U_M)_N,U_M)] \lesssim \p{\frac{\log n}{n}}^{k/d} + \begin{cases}
  \frac{(\log N)^{3/4}}{N^{1/2}} & \text{ if } d=2,\\
 \p{\frac{\log N}{N}}^{1/d}&  \text{ if } d\geq 3.
 \end{cases}
 \end{equation}
 \end{proposition}

\begin{remark}[Spectral estimation of the Laplace-Beltrami operator on $M$]
Garc\'ia Trillos \& al. \cite{trillos2020error} study the problem of estimating the spectral properties of the Laplace-Beltrami operator $\Delta$ on $M$. They show that the eigenvalues of a properly tuned graph Laplacian built on top of a uniform sample of points $W_1,\dots,W_n$ on $M$ will converge at rate $(\log n/n)^{1/(2d)}$ towards the eigenvalues of $\Delta$. We may build upon their results using the estimated volume measure $\widehat\vol_M$. Let $W_1,\dots,W_N$ be a $N$-sample of law $\hat U_M$. Then, Theorem 4 in \cite{trillos2020error} together with Proposition \ref{prop:hat_u_m_approx} yield that the eigenvalues of an appropriate graph Laplacian built on top of the sample $W_1,\dots,W_N$ will approximate the eigenvalues of $\Delta$ at rate (for $d\geq 3$)
\[ \sqrt{\p{\frac{\log n}{n}}^{k/d} +\p{\frac{\log N}{N}}^{1/d} }.\] 
In particular, the faster rate of convergence $(\log n/n)^{k/(2d)}$ is attained by such an estimator, at the price of of being able to sample $N=n^k$ points according to $\hat U_M$. This tradeoff between statistical accuracy and computational efficiency is a common phenomenon, discussed for instance in \cite[Section 6]{weed2019estimation} in the setting of Wasserstein density estimation. Another advantage of our procedure is that, unlike results from \cite{trillos2020error}, it allows one to recover the eigenvalues of the Laplace-Beltrami operator $\Delta$ even in the case where one has access to a sample $X_1,\dots,X_n$ of non-uniform points on $M$. %Further details on this procedure are provided in Appendix ...
\end{remark}

We now turn to the problem of sampling from $\hat \nu_{n,h}$. Computing the density of $K_h*(\nu_n/\hat\rho_h)$ of $\hat \nu_{n,h}$ requires first the computation of 
\begin{equation}\label{eq:hatrhoX}
\hat\rho_h(X_i) = \sum_{j=1}^J \int_{\hat T_j} K_h(\hat \Psi_j(u)-X_i)\chi_j(\hat\Psi_j(u))J\hat\Psi_j(u)\dd u
\end{equation}
 for every point $X_i$ of the sample. Note that there are only a small number of non-zero terms in this sum. Still, the previous basic Monte Carlo method using uniform samples will not perform well in this case, as the integrand in \eqref{eq:hatrhoX} has $\infty$-norm of order $h^{-d}\gg 1$. A slight modification of the estimator $\hat \nu_{n,h}$ is however possible: consider the measure $\tilde \nu_{n,h}$ with density proportional to $K_h*\nu_n$ with respect to $\widehat \vol_M$. Then, Lemma \ref{lem:the_smart_lemma} yields that $\tilde \nu_{n,h}$ is at distance $\eps^{k}$ from the measure with density $cK_h*\nu_n$ with respect to $\vol_M$, where $c =\frac{1}{n}\sum_{i=1}^n \int K_h(x-X_i)\dd  \vol_M(x)$ is a normalizing constant. Furthermore, Lemma \ref{lem:kernel_of_order_k} implies that both $c$ and $\hat\rho_h(X_i)$ deviates from $1$ with an error at most of order $h^{k-1}$. This implies, from the linearization inequality \eqref{eq:wass_neg} that the distance $W_p(\tilde\nu_{n,h},\hat \nu_{n,h})$ is also of order $h^{k-1}$. In particular, the risk of the estimator $\t\nu_{n,h}$ is the same as long as $s\leq k-2$ (with a deterioration of the rate for $k-2\leq s \leq k-1$), while sampling from $\tilde\nu_{n,h}$ does not require computing any integrals. Indeed, the measure $\tilde\nu_{n,h}$ has a density $\alpha$ with respect to the $d$-dimensional Hausdorff measure on $\bigcup_{j=1}^J \hat\Psi_j(\hat T_j)$ equal to, at the point $x$,
\begin{equation}
\alpha(x)  \propto \sum_{j=1}^J \ones\{x\in \hat\Psi_j(\hat T_j)\} \frac{1}{n} \sum_{i=1}^n K_h(x-X_i) \chi_j(x).
\end{equation}
A Metropolis-Hasting scheme with proposal distribution $\sum_{j=1}^J q_j\t U_M^{(j)}$ is then implementable. Furthermore, adapting the proof of Theorem \ref{thm:estimator_M_known}\ref{it:nonnegative}, one can check that $K_h*\nu_n(x)$ is upper bounded by (say) $3\fmax$ and lower bounded by $\fmin/3$ for $x$ in the support of $\widehat \vol_M$. This implies that the acceptance ratio of the Metropolis-Hasting scheme is lower bounded by a positive constant depending only on $\fmin$, $\fmax$ and $d$ (for $\eps$ small enough). The  behavior in practice of those different estimators remains to be investigated and is left for future work.

\subsection*{Acknowledgements} I am grateful to   E.~Aamari, C.~Berenfeld, F.~Chazal, C.~Levrard and P.~Massart for helpful discussions and valuable comments on different mathematical aspects of this work.

\appendix

\section{Geometric properties of \texorpdfstring{$\CC^k$}{Ck} manifolds with positive reach and their estimators}\label{sec:geom_prop}

Let $M \in \MM_{d,\taumin,L}^k$ for some $k\geq 2$ and $\taumin,L>0$. %For $x\in M$, we let $T_xM^\bot$ be the normal space of $M$ at $x$. 
Recall that the angle between two $d$-dimensional subspaces $T_1$ and $T_2$ is given by $\angle(T_1,T_2) \defeq\op{\pi_{T_1}-\pi_{T_2}} = \|\pi_{T_1}^\bot\circ \pi_{T_2}\|_{\mathrm{op}}$, where $\pi_{T_1}$ (resp. $\pi_{T_2}$) is the orthogonal projection on $T_1$ (resp. $T_2$) and $\pi_{T_1}^\bot \defeq \id -\pi_{T_1}$.

\begin{lemma}\label{lem:prop_proj}
Let $x,y\in M$. The following properties hold:
\begin{enumerate}[label=(\roman*)]
\item One has $\vert \pi_y^\bot(x-y)\vert \leq \frac{\vert x-y\vert ^2}{2\taumin}$ and $\angle(T_xM,T_yM)\leq  2\frac{\vert y-x\vert }{\taumin}$.  \label{it:angle}
\item If $\pi_M(z) =x$  for some $z\in M^{\taumin}$, then $z-x \in T_x M^\bot$. \label{it:normal}
\item If $h\leq  \taumin/4$, then $c_dh^d\leq \vol_M(\BB_M(x,h))\leq C_dh^d$.\label{it:bound_volume}
\item If  $h\leq r_0$, then $\BB_M(x,h)\subset \Psi_x(\BB_{T_x M}(0,h))\subset \BB_M(x,8h/7)$. Also, if $u\in\BB_{T_x M}(0,r_0)$, then $\vert u\vert \leq \vert \Psi_x(u)-x\vert \leq 8\vert u\vert /7$.\label{it:param_contains_ball}
\item There exists a map $N_x:\BB_{T_x M}(0,r_0)\to T_xM^\bot$ satisfying $d N_x(0)= 0$, and such that, for $u\in \BB_{T_x M}(0,r_0)$, we have $\Psi_x(u) = x + u + N_x(u)$ with $\vert N_x(u)\vert \leq L\vert u\vert ^2$. \label{it:existence_N}
\item There exist tensors $B_x^1,\dots,B_x^{k-1}$ of operator norm controlled by a constant depending on $L$, $d$, $k$ and $\taumin$, such that, if $u\in T_xM$ satisfies $\vert u\vert \leq C_{k,d,L}$, then $J\Psi_x(u)= 1+ \sum_{i=2}^{k-1} B_{x}^i[u^{\otimes i}] + R_x(u)$, with $\vert R_x(u)\vert \leq C'_{k,d,L}\vert u\vert ^k$. \label{it:taylor_jacob}
\end{enumerate}
\end{lemma}

\begin{proof}
See Theorem 4.18 in \cite{federer1959curvature} and Lemma 6 in \cite{giesen2003shape} for \ref{it:angle}, Theorem 4.8 in \cite{federer1959curvature} for \ref{it:normal}, and Proposition 8.7 in \cite{aamari2018stability} for \ref{it:bound_volume}.
See Lemma A.2 in \cite{aamari2019nonasymptotic} for the second inclusion of balls in \ref{it:param_contains_ball}, which also implies the second inequality in \ref{it:param_contains_ball}. The first inclusion as well as the first inequality in \ref{it:param_contains_ball} follow from the fact that $\Psi_{x}$ is the local inverse of $\t\pi_x$, which is $1$-Lipschitz. 

By a Taylor expansion of $\Psi_x$ at $u=0$, we have $\Psi_x(u)=x+u + N_x(u)$, with $N_x(u) = \int_0^1 d^2\Psi_x(tu)[u^{\otimes 2}]\dd t$. Hence, $\vert N_x(u)\vert \leq L\vert u\vert ^2$. Furthermore, as $\t\pi_x\circ \Psi_x(u)=u$, we have $\pi_x(N_x(u))=0$, i.e.~$N_x$ takes its values in $T_x M^\bot$. This proves \ref{it:existence_N}.

Eventually, we prove \ref{it:taylor_jacob}. Let us write $d\Psi_x(u) = \id_{T_x M} + d N_x(u)$ and $d\Psi_x(u)^\top  d\Psi_x(u) = \id_{T_x M} + (d N_x(u))^\top  dN_x(u)$. We obtain 
\[J\Psi_x(u) = \sqrt{\det(d\Psi_x(u)^\top  d\Psi_x(u))}= \sqrt{\det(\mathrm{id}_{T_x M} + (d N_x(u))^\top  dN_x(u))}.\] One has $dN_x(u) = dN_x(0) + \sum_{j=2}^{k-1} \frac{d^j N_x(0)}{(j-1)!}[u^{\otimes (j-1)}] + R_x(u)$, with $\vert R_x(u)\vert \leq C_{k,L} \vert u\vert ^{k-1}$ and $dN_x(0)=0$. Hence, $ (d N_x(u))^\top  dN_x(u)$ is written as $\sum_{j= 2}^{k-1} B_j[u^{\otimes j}] + R'_x(u)$ with $\vert R'_x(u)\vert \leq C'_{k,l} \vert u\vert ^k$, for some $j$-tensors $B_j$ whose operator norms are bounded in terms of $L$. The operator norm of this operator is smaller than, say, $1/2$ for $\vert u\vert $ sufficiently small, and we conclude the proof by writing a Taylor expansion at $0$ of the function $F \mapsto \sqrt{\det(\id + F)}$.
\end{proof}

We now prove Lemma \ref{lem:existence_pou}, on the construction of smooth partitions of unity based on  some set $S$ which is sufficiently sparse and dense over a tubular neighborhood of $M$.
\begin{proof}[Proof of Lemma \ref{lem:existence_pou}]
Consider the functions $\theta$ and $(\chi_x)_{x\in S}$ as in the statement of the lemma, and, for $y\in M^\delta$, let $Z(y)=\sum_{x'\in S}\theta\p{\frac{y-x'}{8\delta}}$. As $d_H(M^{\delta}\vert S) \leq 4\delta$, we have $Z(y)\geq 1$ and the quantity $\chi_x(y)$ is well defined. The function $\chi_x$ is smooth and we have $\sum_{x\in S} \chi_x \equiv 1$ on $M^{\delta}$. One has $d^l \chi_x(y)$ which is written as a sum of terms of the form $d^{l-j}\theta\p{\frac{y-x}{8\delta}} d^j (Z^{-1})(y)$, and $d^j (Z^{-1})(y)$ is equal to a sum of terms of the form $Z^{j'-j-2}(y) d^{j'} Z(y)$ for $1\leq j'\leq j$. Also, $\op{d^j\theta \p{\frac{y-x'}{8\delta}} }\leq C_{j}\delta^{-j}$ and $\op{d^j Z(y)}\leq C_{j}\delta^{-j}\sum_{x\in S} \ones\{\vert x-y\vert \leq 8\delta\}$. Hence, as $Z\geq 1$, we have for any $l\geq 0$
\[ \op{d^l \chi_x(y)} \leq C'_{l} \delta^{-l}\sum_{x\in S} \ones\{\vert x-y\vert \leq 8\delta\}.\]
It remains to bound this sum. If $x\in \BB(y,8\delta)$, then $\pi_M(x)\in \BB(\pi_M(y),10\delta)$. Also, for $x\neq x'\in S$, we have $\vert \pi_M(x)-\pi_M(x')\vert \geq \vert x-x'\vert -2\delta \geq 2\delta$. In particular, the balls $\BB_M(\pi_M(x),\delta)$ for $x\in S$ are pairwise disjoint, and are all included in $\BB_M(\pi_M(y),11\delta)$ . Therefore, if $11\delta\leq \tau(M)/4$, using Lemma \ref{lem:prop_proj}\ref{it:bound_volume} twice, we obtain that $\vol_M(\BB_M(\pi_M(x),\delta))\geq c_d \delta^d$,  and that
\begin{align*}
 \sum_{x\in S} \ones\{\vert x-y\vert \leq 8\delta\} &\leq \sum_{x\in S} \ones\{\vert x-y\vert \leq 8\delta\}\frac{\vol_M(\BB_M(\pi_M(x),\delta))}{c_d \delta^d}\\
 & \leq  \frac{\vol_M(\BB_M(\pi_M(y),11\delta))}{c_d\delta^d} \leq c'_d.
\end{align*}
This concludes the proof.
\end{proof}

We end this section by detailing the properties of the local polynomial estimators $\hat \Psi_i$ and $\hat T_i$ defined in \cite{aamari2019nonasymptotic}. In particular, the next lemma implies Proposition \ref{prop:estim_manifold}. Recall that $X_i=Y_i+Z_i$ with $Y_i\in M$ and $\vert Z_i\vert \leq \gamma$. Aamari and Levrard introduce tensors $V_{j,i}^* $ which are defined as $d^j\Psi_{X_i}(0)/j!$, where $d^j\Psi_{X_i}(0)$ is the $j$th differential of $\Psi_{X_i}$ at $0$ (see the proof of Lemma 2 in \cite{aamari2019nonasymptotic} for details). In particular, we have $V_{1,i}^*  =\pi_{Y_i}$. Furthermore, as $\t\pi_{Y_j}\circ \Psi_{Y_j}=\id$, we have  $\pi_{Y_j}\circ V_{j,i}^*  = 0$ for $j\geq 2$.
\begin{lemma}\label{lem:prop_poly_estim}
With probability larger than $1-cn^{-k/d}$, for any $1\leq i \leq n$,
\begin{enumerate}[label=(\roman*)]
\item We have $\angle(T_{Y_i}M,\hat T_i) \lesssim \eps^{m-1} + \gamma \eps^{-1}$ \label{it:prop_poly_estim_1}.
\item For $v\in \hat T_i$, we have $\hat \Psi_i(v) = X_i + v + \hat N_i(v)$, where $\hat N_i : \hat T_i\to \hat T_i^\bot$ is defined by $\hat N_i(v) = \sum_{j=2}^{m-1} \hat V_{j,i}[v^{\otimes j}]$.\label{it:prop_poly_estim_3}
\item For any $2\leq j <m$, $\op{\hat V_{j,i}\circ \hat \pi_i-V_{j,i}^*  \circ \pi_{Y_i}} \lesssim \eps^{m-j} + \gamma \eps^{-j}.$\label{it:prop_poly_estim_2}
\item \label{it:prop_proly_estim_4} For $v\in \BB_{\hat T_i}(0,3\eps)$, we have 
\begin{align}
&\vert \hat{\Psi}_i(v)-\Psi_{Y_i}(\pi_{Y_i}(v))\vert  \lesssim \eps^m + \gamma, \label{eq:proofcontrol_error_manifold}\\
&\vert \hat N_i(v)-N_{Y_i}(\pi_{Y_i}(v))\vert \lesssim \eps^m + \gamma, \label{eq:control_N}\\
&\op{d\hat{\Psi}_i(v) - d(\Psi_{Y_i}\circ\pi_{Y_i})(v)} \lesssim \eps^{m-1} + \gamma \eps^{-1} \label{eq:proofcontrol_error_manifold_derivative}\\
&\op{d\hat{N}_i(v) - d(N_{Y_i}\circ\pi_{Y_i})(v)} \lesssim \eps^{m-1} + \gamma \eps^{-1}. \label{eq:proofcontrol_error_N_derivative}
\end{align}
\end{enumerate}
\end{lemma}
\begin{proof}
Lemma \ref{lem:prop_poly_estim}\ref{it:prop_poly_estim_1} is stated in Theorem 2 in \cite{aamari2019nonasymptotic}. Remark that for $x\in \BB(X_i,\eps)$, with $\t x=x-X_i$,
\begin{align*}
&\left\vert  \t x-\pi(\t x) - \sum_{j=2}^{m-1} V_j[\pi(\t x)^{\otimes j}]\right \vert ^2 = \left\vert  \t x-\pi(\t x) - \sum_{j=2}^{m-1} \pi^\bot\circ V_j[\pi(\t x)^{\otimes j}]\right \vert ^2 + \left\vert \sum_{j=2}^{m-1} \pi\circ V_j[\pi(\t x)^{\otimes j}]\right \vert ^2
\end{align*}  
so that we may always assume that the tensors $\hat V_{j,i}$ minimizing the criterion \eqref{eq:def_poly} satisfy $\hat \pi_i \circ \hat V_{j,i}= 0$ for $j\geq 2$. This proves Lemma \ref{lem:prop_poly_estim}\ref{it:prop_poly_estim_3}.

We prove Lemma \ref{lem:prop_poly_estim}\ref{it:prop_poly_estim_2} by induction on $2\leq j < m$. The result for $j=2$ is stated in \cite[Theorem 2]{aamari2019nonasymptotic}. It is shown in \cite{aamari2019nonasymptotic} (see Equation (3)) that there exist tensors $V_{j,i}'$ for $1\leq j <m$ satisfying with probability larger than $1-cn^{-k/d}$,
\begin{equation}\label{eq:control_V'ji}
\op{V_{j,i}'\circ\pi_{Y_i}} \lesssim \eps^{m-j}+ \gamma \eps^{-j}.
\end{equation}
The tensors $V_{j,i}'$ are defined by the relation, for $y\in M$ close enough to $Y_i$,
\begin{equation}\label{eq:def_V'ji}
\begin{aligned}
y- Y_i-\hat \pi_i(y- Y_i) - \sum_{j=2}^{m-1} \hat V_{j,i}[\hat\pi_i(y- Y_i)^{\otimes j}] =\hspace{-.1cm} \sum_{j=1}^{m-1} V'_{j,i}[\pi_{Y_i}(y- Y_i)^{\otimes j}] + R(y- Y_i),
\end{aligned}
\end{equation}
with $\vert R(y-Y_i)\vert  \lesssim \eps^m$, see the proof of Lemma 3 in \cite{aamari2019nonasymptotic}. We also may write
\begin{equation}\label{eq:to_plug_in_V'ji}
y- Y_i = \pi_{Y_i}(y- Y_i) + \sum_{j=2}^{m-1}V_{j,i}^* [\pi_{Y_i}(y- Y_i)^{\otimes j}] + R'(y- Y_i),
\end{equation}
with $\vert R'(y-Y_i)\vert  \lesssim \eps^m$. By plugging \eqref{eq:to_plug_in_V'ji} in the left hand side of \eqref{eq:def_V'ji} and by noting that $\pi_{Y_i}\circ V_{j,i}^* =0$ for $j\geq 2$, we see that $V'_{j,i}\circ \pi_{Y_i}$ is written as the sum of $(\pi_{Y_i}-\hat\pi_i)\circ V_{j,i}^*  + (V_{j,i}^*  \circ \pi_{Y_i} - \hat V_{j,i}\circ \hat \pi_i)$ and of a sum of terms proportional to 
\begin{equation}
\hat V_{j',i}[\hat\pi_i\circ V_{a_1,i}^* \circ \pi_{Y_i},\dots,\hat\pi_i\circ V_{a_{j'},i}^* \circ \pi_{Y_i}],
\end{equation}
where $2\leq j' <j$ and $a_1+\cdots +a_{j'} = j$, $1\leq a_1,\dots,a_{j'} <j$. There exists in particular an index in the sum which is larger than $2$. Assume without loss of generality that $a_1,\dots,a_l >1$ and $a_{l+1},\dots,a_{j'}=1$, so that $\hat \pi_i \circ \hat V_{a_{u},i}=0$ for $1\leq u \leq l$. Then,
\begin{align*}
 &\op{\hat V_{j',i}[\hat\pi_i\circ V_{a_1,i}^* \circ \pi_{Y_i},\dots,\hat\pi_i\circ V_{a_l,i}^* \circ \pi_{Y_i},\dots,\hat\pi_i\circ V_{a_{j'},i}^* \circ \pi_{Y_i}]}  \\
 & =\Big\|\hat V_{j',i}[\hat\pi_i\circ (V_{a_1,i}^* -\hat V_{a_1,i})\circ \pi_{Y_i},\dots,\hat\pi_i\circ (V_{a_l,i}^* -\hat V_{a_l,i})\circ \pi_{Y_i},\\
 &\qquad\qquad\qquad  \hat\pi_i\circ V_{a_{j'},i}^* \circ \pi_{Y_i},\dots,\hat\pi_i\circ V_{a_{l+1},i}^* \circ \pi_{Y_i}]\Big\|_{\mathrm{op}} \\
 &\lesssim \ell\prod_{u=1}^l\op{V_{a_u,i}^* \circ \pi_{Y_i}-\hat V_{a_u,i}\circ \pi_{Y_i}} \\
 &\lesssim  \ell \prod_{u=1}^l\p{\op{V_{a_u,i}^* \circ \pi_{Y_i}-\hat V_{a_u,i}\circ \hat\pi_{i}} + \ell\op{\pi_{Y_i} - \hat \pi_i}} \\
 & \lesssim \eps^{-1}\prod_{u=1}^l\p{\eps^{m-a_u}+ \gamma\eps^{-a_u}+\eps^{m-2}+ \gamma\eps^{-2}}\\
 &\lesssim \eps^{-1}(\eps^{lm-(j-l)}+\gamma^l\eps^{-(j-l)})\lesssim \eps^{m-j}+\gamma\eps^{-j},
\end{align*}
where at the last line we use  the induction hypothesis as well as  Lemma \ref{lem:prop_poly_estim}\ref{it:prop_poly_estim_1}, the fact that $\sum_{u=1}^la_u=j-l$ and that $\ell \lesssim \eps^{-1}$. As $\op{(\pi_{Y_i}-\hat\pi_i)\circ V_{j,i}^* } \lesssim \eps^{m-1}+ \gamma\eps^{-1}$, we obtain that \[\op{(V_{j,i}^*  \circ \pi_{Y_i} - \hat V_{j,i}\circ \hat \pi_i) - V'_{j,i}\circ\pi_{Y_i} }\lesssim \eps^{m-j}+\gamma \eps^{-j}.\] Hence, using \eqref{eq:control_V'ji},
\begin{align*}
\op{V_{j,i}^*  \circ \pi_{Y_i} - \hat V_{j,i}\circ \hat \pi_i} &\leq \op{(V_{j,i}^*  \circ \pi_{Y_i} - \hat V_{j,i}\circ \hat \pi_i) - V'_{j,i}\circ\pi_{Y_i} }+ \op{V'_{j,i}\circ\pi_{Y_i} }\\
&\lesssim \eps^{m-j}+ \gamma\eps^{-j}.
\end{align*}

We now may prove \eqref{eq:proofcontrol_error_manifold}. Indeed, for $v\in \BB_{\hat T_i}(0,3\eps)$, $\hat \Psi_i(v) = X_i + v + \sum_{j=2}^{m-1} \hat V_{j,i}[v^{\otimes j}]$, whereas by a Taylor expansion, $\Psi_{Y_i}\circ \pi_{Y_i}(v)=  Y_i + \pi_{Y_i}(v) + \sum_{j= 2}^{m-1} V_{j,i}[\pi_{Y_i}(v)^{\otimes j}] + R(v)$, with $\vert R(v)\vert \lesssim \eps^m$. By Lemma \ref{lem:prop_poly_estim}\ref{it:prop_poly_estim_2}, the difference between the two quantities is bounded with high probability by a sum of terms of order $(\eps^{m-j}+ \gamma\eps^{-j})\vert v\vert ^j\lesssim \eps^m+ \gamma$.  Inequality \eqref{eq:control_N} is directly implied by \eqref{eq:proofcontrol_error_manifold} and Lemma \ref{lem:prop_poly_estim}\ref{it:prop_poly_estim_1}. Inequality \eqref{eq:proofcontrol_error_manifold_derivative} is proven as \eqref{eq:proofcontrol_error_manifold}, by noting that, for $h\in\hat T_i$,
\begin{align*}
\begin{cases}&d(\Psi_{Y_j}\circ \pi_{Y_j})(v)[h] = \pi_{Y_j}(h) + \sum_{j=2}^{m-1} jV_{j,i}^* [\pi_{Y_j}(v),\pi_{Y_j}(h)^{\otimes(j-1)}] + R'(v)h\\
&d\hat\Psi_j(v)[h] = h + \sum_{j=2}^{m-1} j\hat V_{j,i}[v,h^{\otimes(j-1)}],
\end{cases}
\end{align*}
with $\op{R'(v)}\lesssim \eps^{m-1}$. Equation \eqref{eq:proofcontrol_error_N_derivative} is shown in a similar way.
\end{proof}

\section{Properties of negative Sobolev norms}\label{sec:proof_preli}

\begin{proof}[Proof of Proposition \ref{prop:congestionned}]
The second inequality in \ref{it:cong1} is trivial. The assertion \ref{it:cong2} is stated in \cite[Theorem 2.1]{brasco2010congested} for an open set $\Omega \subset\R^d$, and their proof can be straightforwardly  adapted to the manifold setting. It remains to prove the first inequality in \ref{it:cong1}. Note that for any $g$ with $ \|\nabla g\|_{L_{p^* }(M)}\leq 1$, one has $\int fg \dd \vol_M=\int f(g-\int g\dd \vol_M)\dd \vol_M$ as $\int f\dd \vol_M=0$. Also, by Poincar\'e inequality (see \cite[Theorem 0.6]{besson2018poincare}),
\[
\left\|g-\int_M g \right\|_{L_{p^* }(M)} \leq C^{\frac 1 p}R^{\frac{d}{p^* }+\frac{1}{p}}\|\nabla g\|_{L_{p^* }(M)} \leq C^{\frac 1 p}R^{\frac{d}{p^* }+\frac{1}{p}},\]
where $R= \max\{d_g(x,y),\ x,y\in M\}$ and $C$ depends on $d$ and on a lower bound $\kappa$ on the Ricci curvature of $M$. Therefore, $\left\|g-\int_M g \right\|_{H^1_{p^* }(M)}\leq C^{\frac 1 p}R^{\frac{d}{p^* }+\frac{1}{p}}$. The quantity $\kappa$ can be further lower bounded by a constant depending on $\taumin$ and $d$. Indeed, a bound on the second fundamental form of $M$ entails a bound on the Ricci curvature according to Gauss equation (see e.g.~\cite[Chapter 6]{do1992riemannian}), and the second fundamental form is controlled by the reach of $M$, see \cite[Proposition 6.1]{niyogi2008finding}. As $C^{\frac 1 p} \leq C\vee 1$, to conclude, it suffices to bound the geodesic diameter of $M$. This is done in the following lemma.
\end{proof}
\begin{lemma}\label{lem:geodesic_diameter}
 The geodesic diameter of $M$ satisfies  $\sup_{x,y\in M} d_g(x,y) \leq c_d \vert \vol_M\vert \taumin^{1-d}$.
\end{lemma}
\begin{proof}
\begin{figure}
\centering
\includegraphics[width=0.5\textwidth]{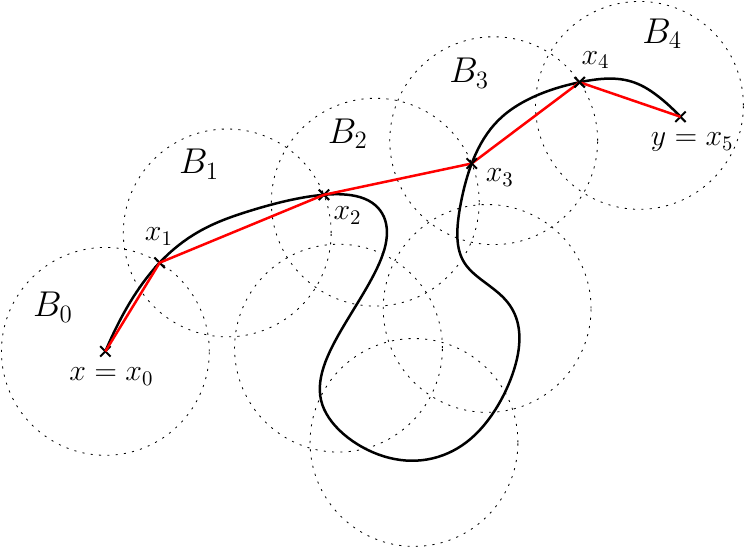}
\caption{Illustration of the construction in the proof of Lemma \ref{lem:geodesic_diameter}}
\end{figure}

Consider a covering of $M$ by $N$ open balls of radius $r_1=\taumin/4$ (for the Euclidean distance) and let $x,y\in M$. Such a covering exists with $N\leq c_d\vert \vol_M\vert  r_1^{-d}$ by standard packing arguments. Let  $\gamma:[0,\ell]\to M$ be a unit speed curve between $x$ and $y$. Let $B_0$ be the ball of the covering such that $x\in B_0$. If $y\in B_0$, then $\vert x-y\vert \leq 2r_1$, and by \cite[Proposition 6.3]{niyogi2008finding}, we have $d_g(x,y)\leq 4r_1$. Otherwise, let $t_0 = \inf\{t\in [0,\ell],\ \forall t'\geq t,\ \gamma(t')\not\in B_0\}$. Then $x_1\defeq \gamma(t_0)$ belong to the boundary of $B_0$, and is also in some other ball $B_1$. By the previous argument, we have $d_g(x,x_1)\leq 4r_1$. If $y\in B_1$, then $d_g(x_1,y)\leq 4r_1$ and $d_g(x,y)\leq 8r_1$. Otherwise, we define $t_1 = \inf\{t\in [t_0,\ell],\ \forall t'\geq t,\ \gamma(t')\not\in B_1\}$ and we iterate the same argument. At the end, we obtain a sequence $x=x_0,x_1,\dots,x_I$ of points in $M$ with associated balls $B_i$ which contain $x_i$, such that $y\in B_I$ and $d_g(x_i,x_{i+1})\leq 4r_1$. Furthermore, all the balls $B_i$ are pairwise distinct. As $d_g(x_I,y)\leq 4r_1$, we have $\ell\leq (I+1)4r_1 \leq  (N+1)4r_1\leq 8Nr_1$. By letting $\gamma$ be a geodesic, we obtain in particular $\ell=d_g(x,y) \leq 8Nr_1 \leq 8c_d\vert \vol_M\vert r_1^{1-d}$. 
\end{proof}
\begin{proof}[Proof of Proposition \ref{prop:wass_neg}]
Given a measurable map $\rho:[0,1]\to \PP^p$, $E_t$ a vectorial measure absolutely continuous with respect to $\rho_t$ (see \cite[Box 4.2]{santambrogio2015optimal}) and $v(x,t)$ a time-depending vector field, defined as the density of $E_t$ with respect to $\rho_t$, we define the Benamou-Brenier functional
\begin{equation}
\BB_p(\rho,E) \defeq \int \vert v(x,t)\vert ^p \dd \rho_t(x) \dd t.
\end{equation}
The Benamou-Brenier formula \cite{benamou2000computational, brenier2003extended}
 asserts that for $\mu,\nu\in \PP_1^p$ supported on some ball of radius $R$,
\begin{equation}\label{eq:benamou_brenier}
W_{p}^p(\mu,\nu ) = \min\left\{\BB_p(\rho,E),\ \partial_t \rho_t + \nabla \cdot E_t = 0, \rho_0=\mu, \rho_1= \nu\right\},
\end{equation}
where $\rho_t$ is supported on the ball of radius $R$, and the continuity equation $\partial_t \rho + \nabla \cdot E = \mu-\nu$ has to be understood in the distributional sense, i.e. 
\begin{equation}
\begin{split}
&\int_{[0,1]\times \R^D}\partial_t \phi(t,x)\dd \rho(t,x) + \int_{[0,1]\times \R^D} \nabla \phi(t,x) \cdot \dd E(t,x) =0,
\end{split}
\end{equation}
for all $\phi \in \CC^1((0,1)\times \BB(0,R))$ with compact support.

Assume that $\mu$ has a density $f_0$ and $\nu$ has a density $f_1$ on $M$. As $\tau(M)>0$, the existence of a probability measure of mass $1$, supported on $M$, with density larger than $\fmin$ implies that $M$ is compact, see Remark \ref{remark:fmin}. It is in particular included in a ball $\BB(0,R)$ for some $R$ large enough. Let $w$ be a vector field on $M$ with $\nabla \cdot w = \mu-\nu$ in a distributional sense, i.e. $\int \nabla g\cdot w=- \int g(\mu-\nu)$ for all $g\in \CC^1(M)$. Let  $\rho_t = (1-t)\mu + t\nu$ and define $E$ the vector measure having density $w$ with respect to $\Leb_1 \times \vol_M$, where $\Leb_1$ is the Lebesgue measure on $[0,1]$. Then $(\rho,E)$ satisfies the continuity equation and $E = v\cdot \rho$ where $v(t,x)=\frac{w(x)}{(1-t)f_0(x)+tf_1(x)}$ for $t\in[0,1]$, $x\in M$. Hence,
\begin{align*}
W_p^p(\mu,\nu) &\leq \int_0^1 \int \frac{1}{p}\vert v\vert ^p \dd \rho \\
&= \frac{1}{p}\int_0^1\int \frac{\vert w(x)\vert ^p}{\vert (1-t)f_0(x)+tf_1(x)\vert ^p}( (1-t)f_0(x)+tf_1(x)) \dd x\dd t \\
&\leq  \frac{1}{p}\int\vert w(x)\vert ^p\dd x \frac{1}{\fmin^{p-1}}.
\end{align*}
By taking the infimum on vector fields $w$ on $M$ satisfying $\nabla \cdot w = \mu-\nu$ and using Proposition \ref{prop:congestionned}, we obtain the conclusion. The second inequality in \eqref{eq:wass_neg} follows from Proposition \ref{prop:congestionned}.
\end{proof}

\section{Proofs of Section \ref{sec:study_estimator}}\label{sec:proof_bias}

\begin{proof}[Proof of Lemma  \ref{lem:kernel_of_order_k}]
We first prove \eqref{eq:kernel_1}. Note that if $\vert x-y\vert \geq h$ for $x,y\in M$, then $K_h(x-y)=0$. Hence, by a change of variable, using that $\BB_M(x,h)\subset \Psi_x(\BB_{T_xM}(0,h))$ according to Lemma \ref{lem:prop_proj}\ref{it:param_contains_ball},
\begin{align*}
\int_M K_h(x-y) B[(x-y)^{\otimes j}] \dd y &= \int_{\BB_{T_xM}(0,h)} K_h(x-\Psi_x(v)) B[(x-\Psi_x(v))^{\otimes j}] J\Psi_x(v)\dd v\\*
&= \int_{\BB_{T_xM}(0,1)} K\p{\frac{x-\Psi_x( h v)}{h}} B[(x-\Psi_x( hv))^{\otimes j}]J\Psi_x( hv)\dd v.
\end{align*}
As the functions $\Psi_x$ and $K$ are $\CC^{k}$, according to Lemma \ref{lem:prop_proj}\ref{it:existence_N} and Lemma \ref{lem:prop_proj}\ref{it:taylor_jacob}, we can write by a Taylor expansion, for $v,u \in \BB_{T_xM}(0,r_0)$,
\begin{equation}
\begin{cases}
\Psi_x(v) = x + v + \sum_{i=2}^{k-1} \frac{d^i\Psi_x(0)}{i!}[v^{\otimes i}] + R_1(x,v)\\
J\Psi_x(v) = 1 +  \sum_{i= 2}^{k-1} B_x^i [v^{\otimes i}] +  R_2(x,v)\\
K(v+u) = K(v) + \sum_{i=1}^{k-1} \frac{d^i K(v)}{i!}[u^{\otimes i}] + R_3(v,u)\\
B[(v+u)^{\otimes j}] = B[v^{\otimes j}] + \sum_{\emptyset \neq \sigma \subset \{1,\dots,j\}} B[v^{\sigma},u^{\sigma^c}],
\end{cases}
\end{equation}
where  $\vert R_j(x,v)\vert \leq C_j\vert v\vert ^{k}$ for $j=1,2$, $\vert R_3(v,u)\vert \leq C_3\vert u\vert ^{k}$ and $(v^{\sigma},u^{\sigma^c})$ is the $j$-tuple whose $l$th entry is equal to $v$ if $l\in \sigma$, $u$ otherwise.  We obtain that 
\[\frac{x-\Psi_x( hv)}{h}= -v - \sum_{i=2}^{k-1} \frac{d^i\Psi_x(0)}{i!}[(h v)^{\otimes i}] h^{-1} -R_1(x,hv)h^{-1},\] and that the expression $K\p{\frac{x-\Psi_x( h v)}{h}} B[(x-\Psi_x( hv))^{\otimes j}]J\Psi_x( hv)$ is written as a sum of terms of the form 
\begin{equation}\label{eq:big_dl}
\begin{split}
&C_{i_0,i_1,i_2}h^{-i_0}d^{i_0}K( v)[ (d^{i_1}\Psi_x(0)[(hv)^{\otimes i_1}])^{\otimes i_0}] F_{i_2}[ (h v)^{\otimes i_2}]
  \end{split}
\end{equation}
for $0\leq i_0\leq k-1$, $2\leq i_1\leq k-1$ and $j\leq i_2\leq k'$, where $F_{i_2}$ is some tensor of order $i_2$ and $k'$ is some integer depending on $k$ and $j$, plus a remainder term smaller than $\op{B}\vert  h v\vert ^{k-1+j}$ up to a constant depending on $k$, $j$, $L_k$ and $K$. The terms for which $i_0i_1+i_2-i_0\geq k$ are smaller than $\op{B}h^{k}$ up to a constant, whereas the integrals of the other the terms are null as the kernel is of order $k$. The first inequality in \eqref{eq:kernel_0} is proven in a similar manner. Let us now bound $\|\rho_h\|_{\CC^j(M)}$. Given $x\in M$, we have
\begin{align*}
d^j(\rho_h\circ\Psi_x)(0) = h^{-j} \int_{\BB_{T_xM}(0,h)}(d^j K)_h(x-\Psi_x(v)) J\Psi_x(v)\dd v.
\end{align*}
Therefore, using the same argument as before, we obtain that $\op{d^j(\rho_h\circ\Psi_x)(0)}\lesssim h^{k-1-j}$.
\end{proof}

\begin{proof}[Proof of Lemma \ref{lem:bias2}]
Let  $0\leq l \leq k-1$ be even, $\phi\in \CC^\infty(M)$  be supported in $\BB_M(x_0,h_0)$ for some $h_0$ small enough and $g\in L_{p^* }(M)$ with $\|g\|_{L_{p^* }(M)}\leq  1$.  Let $x=\Psi_{x_0}(u)\in \BB_M(x_0,h_0)$ and let $\t\phi_{x_0}=\t\phi\circ\Psi_{x_0}$. Recall that $\t\phi_l=d^l\t\phi_{x_0} \circ\t\pi_{x_0}$. We have $K_h(x-\Psi_{x_0}(v))\neq 0$ only if $\vert x-\Psi_{x_0}(v)\vert \leq h$. Hence, as $\vert x-\Psi_{x_0}(v)\vert \geq \vert u-v\vert $ (recall that $\Psi_{x_0}$ is the inverse of the projection $\t\pi_{x_0}$), the function $K_h(x-\Psi_{x_0}(\cdot))$ is supported on $\BB_{T_{x_0}M}(u,h)\subset \BB_{T_{x_0} M}(0,r_0)\eqdef B_0$ for $h,h_0$ small enough. 
Thus, 
\begin{align*}
 A_h \phi(x) &= \int_{\BB_M(x,h)} K_h(x-y)( \t\phi(y)- \t\phi(x))\dd y\\*
& = \int_{B_0} K_h(x-\Psi_{x_0}(v))( \t\phi_{x_0}(v)- \t\phi_{x_0}(u))J\Psi_{x_0}(v)\dd v.
\end{align*}
We may write  
\begin{align*}
&\t\phi_{x_0}(v)-\t\phi_{x_0}(u)= \sum_{i=1}^{l-1}\frac{d^i \t \phi_{x_0}(u)}{i!}[(v-u)^{\otimes i}] +\int_0^1 d^l \t \phi_{x_0}(u+\lambda (v-u))[(v-u)^{\otimes l}]\frac{(1-\lambda)^{l-1}}{(l-1)!}\dd \lambda.
\end{align*}
 Each term $\int_{B_0} K_h(x-\Psi_{x_0}(v))\frac{d^i \t \phi_{x_0}(u)}{i!}[(v-u)^{\otimes i}] J\Psi_{x_0}(v)\dd v$ is equal to
\[ \int_M K_h(x-y) \frac{d^i \t \phi_{x_0}(\t\pi_{x_0}(x))}{i!}[(\pi_{x_0}(y-x))^{\otimes i}]\dd y, \]
and is therefore of order smaller than $h^k\max_{1\leq i \leq  l}\op{\t\phi_i(x)}$ by Lemma \ref{lem:kernel_of_order_k}. Hence, $A_h \phi(x)$ is equal to the sum of a remainder term of order $h^{k}\max_{1\leq i \leq  l}\op{\t\phi_i(x)}$ and of
\begin{align*}
&\int_0^1 \int_{B_0} K_h(x-\Psi_{x_0}(v))d^l \t \phi_{x_0}(u+\lambda (v-u))[(v-u)^{\otimes l}]\frac{(1-\lambda)^{l-1}}{(l-1)!}J\Psi_{x_0}(v)\dd v\dd \lambda \\
&= \int_0^1 \hspace{-.1cm}\int_{B_0}K_h(x-\Psi_{x_0}(v))\p{ d^l \t \phi_{x_0}(u+\lambda (v-u))-d^l \t \phi_{x_0}(u)}[(v-u)^{\otimes l}] \times \frac{(1-\lambda)^{l-1}}{(l-1)!}J\Psi_{x_0}(v)\dd v\dd \lambda \\
&\qquad +R_1(x),
\end{align*} 
where $\vert R_1(x)\vert \lesssim h^{k}\max_{1\leq i \leq  l}\op{\t\phi_i(x)}$ by Lemma \ref{lem:kernel_of_order_k}. We now fix $\lambda\in (0,1)$ and write by a change of variables
\begin{align*}
U(x)&\defeq \int_{B_0}K_h(x-\Psi_{x_0}(v))( d^l \t \phi_{x_0}(u+\lambda (v-u))-d^l \t \phi_{x_0}(u))[(v-u)^{\otimes l}]J\Psi_{x_0}(v)\dd v\\
&= \int_{B_0}K_h\p{x-\Psi_{x_0}\p{u+\frac{w-u}{\lambda}}}( d^l \t \phi_{x_0}(w)-d^l \t \phi_{x_0}(u))\left[ \frac{(w-u)^{\otimes l}}{\lambda^l}\right]\\
&\qquad\qquad \times J\Psi_{x_0}\p{u+\frac{w-u}{\lambda}}\frac{\dd w}{\lambda^d}
\end{align*}
Note that $\vert K_h(u)-K_h(v)\vert  \lesssim h^{-d-1}\vert u-v\vert \ones\{\vert u\vert \leq h \text{ or } \vert v\vert \leq h\}$, and that, as $\Psi_{x_0}$ is $\CC^2$,
\begin{align*}
\left\vert x-\Psi_{x_0}\p{u+\frac{w-u}{\lambda}}-\frac{x-\Psi_{x_0}(w)}{\lambda}\right \vert &\leq \left\vert \frac{d\Psi_{x_0}(u)[w-u] -( x -\Psi_{x_0}(w))}{\lambda}\right \vert  + \frac{L_k\vert w-u\vert ^2}{2\lambda^2} \\
&\leq  \frac{L_k\vert w-u\vert ^2}{\lambda} \lesssim \frac{\vert w-u\vert ^2}{\lambda},
\end{align*} 
whereas, as $J\Psi_{x_0}$ is Lipschitz continuous,
\begin{align*}
&\left\vert  J\Psi_{x_0}\p{u+\frac{w-u}{\lambda}}-J\Psi_{x_0}(w)\right \vert \lesssim \left\vert  u+\frac{w-u}{\lambda}-w\right \vert  \lesssim \frac{\vert w-u\vert }{\lambda}.
\end{align*}
Hence, $U(x)$ is equal to the sum of
\begin{align*}
&\lambda^{-l}\int_{B_0}K_{h\lambda}\p{x-\Psi_{x_0}(w)}\p{ d^l \t \phi_{x_0}(w)-d^l \t \phi_{x_0}(u)}[(w-u)^{\otimes l}]J\Psi_{x_0}\p{w}\dd w\\*
&= \lambda^{-l}\int_{M}K_{h\lambda}\p{x-y}\p{ \t\phi_l(y)-\t\phi_l(x)}[(\pi_{x_0}(y-x))^{\otimes l}]\dd y,
\end{align*}
and of a remainder term  smaller than
\begin{align*}
&\lambda^{-l}\int_{B_0}\Big\vert  \lambda^{-d}K_h\p{x-\Psi_{x_0}\p{u+\frac{w-u}{\lambda}}}J\Psi_{x_0}\p{u+\frac{w-u}{\lambda}} -K_{h\lambda}\p{x-\Psi_{x_0}(w)}J\Psi_{x_0}\p{w} \Big\vert \\
&\qquad\times \op{ d^l \t \phi_{x_0}(w)-d^l \t \phi_{x_0}(u)}\vert w-u\vert ^l \dd w \\
&\lesssim \lambda^{-l}\int_{\vert w-u\vert \lesssim \lambda h}\Bigg(\frac{\vert w-u\vert ^2}{(\lambda h)^{d+1}}J\Psi_{x_0}\p{u+\frac{w-u}{\lambda}}+\vert K_{h\lambda}\p{x-\Psi_{x_0}(w)}\vert \frac{\vert w-u\vert }{\lambda}\Bigg)\\
&\qquad\qquad \times \op{ d^l \t \phi_{x_0}(w)-d^l \t \phi_{x_0}(u)}\vert w-u\vert ^l \dd w \\
&\lesssim h^{l+1}(\lambda h)^{-d}\int_{\vert w-u\vert \lesssim \lambda h} \op{ d^l \t \phi_{x_0}(w)-d^l \t \phi_{x_0}(u)} \dd w.
\end{align*}
 Putting all the estimates together, we may now write $\int_M A_h\phi(x)g(x)\dd x$ as $S+R_2$, where, by the symmetrization trick (using that $l$ is even)
\begin{align*}
S &= \iint_{M\times M}K^{(l)}_{h}\p{x-y}\p{ \t\phi_l(y)-\t\phi_l(x)}[(\pi_{x_0}(y-x))^{\otimes l}] g(x)\dd y\dd x \\
&= \iint_{M\times M}K^{(l)}_{h}\p{x-y}\p{ \t\phi_l(x)-\t\phi_l(y)}[(\pi_{x_0}(x-y))^{\otimes l}] g(y)\dd y\dd x \\
&=\frac{1}{2}\iint_{M\times M}K^{(l)}_{h}\p{x-y}\p{ \t\phi_l(y)-\t\phi_l(x)}[(\pi_{x_0}(x-y))^{\otimes l}]( g(x)- g(y))\dd y\dd x,
\end{align*}
and, as $A_h\phi$ is supported on $\BB_M(x_0,h_0+h)\subset\BB_M(x,2h_0)$ if $h$ is small enough, $R_2$ is smaller up to a constant than, 
\begin{align}
&h^{l+1}(\lambda h)^{-d}\hspace{-.1cm}\int_{x\in \BB_M(x_0,2h_0)} \hspace{-.05cm} \int_{\vert w-\t\pi_{x_0}(x)\vert \lesssim \lambda h}\hspace{-.05cm} \op{ d^l \t \phi_{x_0}(w)-d^l \t \phi_{x_0}(\t\pi_{x_0}(x))}\vert  g(x)\vert \dd w \dd x  \label{eq:longproof_bias}\\
&\qquad +\int_M h^{k}\max_{1\leq i \leq  l}\op{\t \phi_i(x)}\vert g(x)\vert \dd x \nonumber\\
&\lesssim h^{l+1}(\lambda h)^{-d}\int_{w\in \BB_M(x,3h_0)}\op{ d^l \t \phi_{x_0}(w)}\int_{\vert w-\t\pi_{x_0}(x)\vert \lesssim \lambda h} \vert  g(x)\vert  \dd  x\dd w \label{eq:longproof_bias2}\\
&\qquad + h^{l+1}\int_{x\in \BB_M(x,2h_0)} \op{\t \phi_l(x)}\vert g(x)\vert \dd x +\int_M h^{k}\max_{1\leq i \leq  l}\op{\t \phi_i(x)}\vert  g(x)\vert \dd x,\nonumber
\end{align}
where we also used Lemma \ref{lem:prop_proj}\ref{it:bound_volume}. By the chain rule, 
\begin{equation}\label{eq:bound_of_tildephi}
\begin{split}
\max_{1\leq i \leq  l}\op{\t \phi_i(x)} &= \max_{1\leq i \leq  l}\op{d^i(\tilde \phi\circ \Psi_{x_0})\circ \tilde\pi_{x_0}(x)} \\
&=\max_{1\leq i \leq  l}\op{d^i(\tilde \phi\circ\Psi_x \circ \tilde \pi_x \circ \Psi_{x_0})\circ \tilde\pi_{x_0}(x)} \\
& \lesssim \max_{1\leq i \leq l} \op{d^i (\tilde \phi\circ\Psi_x)(\tilde \pi_x \circ\Psi_{x_0}\circ\tilde\pi_{x_0}(x))}\\
 &\lesssim \max_{1\leq i \leq l} \op{d^i (\tilde \phi\circ\Psi_x)(0)} = \op{d^l \tilde\phi(x)}.
\end{split}
\end{equation}
 Hence, applying H\"older's inequality and using that $\|g\|_{L_{p^* }(M)} \leq  1$ show that the two last terms in \eqref{eq:longproof_bias2} are of order $h^{l+1}\|\t\phi\|_{H^l_p(M)}$. To bound the first term in \eqref{eq:longproof_bias2}, remark that by Young's inequality for integral operators \cite[Theorem 0.3.1]{sogge_2017}, if $\TT_{\lambda h}(g)(y) = (\lambda h)^{-d}\int_{\vert x-y\vert \lesssim \lambda h} \vert g(x)\vert \dd x$, then $\|\TT_{\lambda h}g\|_{L_{p^* }(M)} \lesssim \|g\|_{L_{p^* }(M)}$. This yields, by H\"older's inequality,
\begin{align*}
h^{l+1}\int_{w\in \BB_M(x,3h_0)}\op{ d^l \t \phi_{x_0}(w)}\TT_{h\lambda}(g)(\Psi_{x_0}(w))\dd w  \lesssim h^{l+1}\|\t\phi\|_{H^l_p(M)},
\end{align*}
which concludes the proof of the first statement of Lemma \ref{lem:bias2}. To bound the remainder term in terms of $\|\t\phi\|_{H^{l+1}_p(M)}$, we bound the second term in \eqref{eq:longproof_bias}  in the same fashion, while, to bound the first term, we write, by a change of variables,
\begin{align*}
&\int_{\BB_M(x_0,2h_0)}\int_{\vert w-\t\pi_{x_0}(x)\vert \lesssim \lambda h} \op{ d^l \t \phi_{x_0}(w)-d^l \t \phi_{x_0}(\t\pi_{x_0}(x))}\vert g(x)\vert \dd x \dd w\\
&\leq \int_0^1 \hspace{-.1cm} \int_{\BB_M(x_0,2h_0)}\int_{\vert w-\t\pi_{x_0}(x)\vert \lesssim \lambda h} \op{ d^{l+1} \t \phi_{x_0}(\t\pi_{x_0}(x)+\lambda'(w-\t\pi_{x_0}(x)))} \\
&\qquad\qquad \times \vert \t\pi_{x_0}(x)-w\vert \vert g(x)\vert \dd x \dd w \dd \lambda'\\
&\lesssim h \int_0^1\int_{\BB_M(x_0,2h_0)}\int_{\vert u-\t\pi_{x_0}(x)\vert \lesssim \lambda'\lambda h} \op{ d^{l+1} \t \phi_{x_0}(u)}\vert g(x)\vert \dd x \frac{\dd u}{\lambda'^d} \dd \lambda',
\end{align*}
and this term is bounded as the first term in \eqref{eq:longproof_bias2} by $h(h\lambda)^d \|\t\phi\|_{H^{l+1}_p(M)}$, concluding the proof of Lemma \ref{lem:bias2}.
\end{proof}

\begin{proof}[Proof of Lemma \ref{lem:equiv_norm_sobolev}]
The second inequality of Lemma \ref{lem:equiv_norm_sobolev} follows from the definition of the Sobolev norm and of the bound \eqref{eq:bound_of_tildephi} applied to $\tilde\phi=\eta$.
 To prove the first inequality, write
\begin{align*}
&h^{-d} \iint_{\BB_M(x_0,h_0)^2} \ones\{\vert x-y\vert \leq  h\} \frac{\op{\eta_l(x)-\eta_l(y)}^p}{\vert x-y\vert ^p}\dd x\dd y \\
&\lesssim h^{-d} \iint_{\BB_{T_{x_0}M}(0,h_0)^2} \ones\{\vert \Psi_{x_0}(u)-\Psi_{x_0}(v)\vert \leq  h\}\frac{\op{d^l(\eta\circ\Psi_{x_0})(u)-d^l(\eta\circ\Psi_{x_0})(v)}^p}{\vert \Psi_{x_0}(u)-\Psi_{x_0}(v)\vert ^p}\dd u\dd v \\
&\hspace{10cm}\text{ as $J\Psi_{x_0}(u)\lesssim 1$ for $\vert u\vert \lesssim 1$}\\
&\lesssim h^{-d} \int_0^1\iint_{\BB_{T_{x_0}M}(0,h_0)^2} \ones\{\vert u-v\vert \leq   h\} \op{d^{l+1}(\eta\circ\Psi_{x_0})(u+\lambda(v-u))}^p\dd u\dd v\dd \lambda \\
&\hspace{10cm}\text{ as $\vert v-u \vert \leq \vert \Psi_{x_0}(u)-\Psi_{x_0}(v)\vert$}\\
&\lesssim h^{-d} \int_0^1\iint_{\BB_{T_{x_0}M}(0,2h_0)^2} \ones\{\vert w-u\vert \leq  \lambda h\} \op{d^{l+1}(\eta\circ\Psi_{x_0})(w)}^p\dd u\dd w \lambda^{-d} \dd \lambda\\
&\lesssim  \int_0^1\int_{\BB_{T_{x_0}M}(0,2h_0)} \op{d^{l+1}(\eta\circ\Psi_{x_0})(w)}^p\dd w \lesssim  \int_{\BB_M(x_0,h_0)}\op{\eta_{l+1}(x)}^p\dd x,
\end{align*}
where at the second to last line, we used that $w=u+\lambda(v-u)$ is of norm smaller than $2h_0$ if $\vert u\vert \leq h_0$ and $\vert v-u\vert \leq h\leq h_0$, and, at the last line, we used that $J\Psi_{x_0}(w)\geq 1/2$ for $\vert w\vert $ small enough.
\end{proof}

\section{Proof of Lemma \ref{lem:fluc2}}\label{sec:proof_fluctuation}
Lemma \ref{lem:fluc2} is heavily based on the following classical control on the gradient of the Green function.
\begin{lemma}\label{lem:green_function}
Let $x,y\in M$, then
\begin{equation}
\vert \nabla_x G(x,y)\vert  \lesssim \frac{1}{d_g(x,y)^{d-1}}\leq \frac{1}{\vert x-y\vert ^{d-1}}.
\end{equation}
\end{lemma}

\begin{proof}
For $d\geq  2$, a proof of Lemma \ref{lem:green_function} is found in \cite[Theorem 4.13]{aubin1982nonlinear}. See also \cite[Theorem 5.2]{hiroshima1996construction} for a proof with more explicit constants in the case $d\geq 3$. Constants in their proofs depend on $d$, bounds on the curvature of $M$, $\vert \vol_M\vert $ and the geodesic diameter of $M$. As, those three last quantities can be further bounded by constants depending on $\taumin$, $\fmin$ and $d$, see Lemma \ref{lem:geodesic_diameter} and \cite[Proposition 6.1]{niyogi2008finding}, this concludes the proof. For $d=1$, $M$ is isometric to a circle, for which a closed formula for $G$ exists \cite{burkhardt1894fonctions}, and satisfies $\vert \nabla_x G(x,y)\vert \leq  1$.
\end{proof}

Recall that, by Lemma \ref{lem:kernel_of_order_k}, $\vert \rho_h(x)\vert \geq 1/2$ for all $x\in M$. Therefore, Lemma \ref{lem:green_function} yields
\begin{align*}
\left\vert \nabla G\p{K_h*\p{\frac{\delta_{x}}{\rho_h}}}(z)\right \vert  &= \left\vert \int_{ M} \nabla_z G(z,y) \frac{K_h(x-y)}{\rho_h(x)}\dd y \right \vert \lesssim\int_{ \BB_M(x,h)} \frac{\|K\|_\infty h^{-d}}{\vert z-y\vert ^{d-1}}\dd y.
\end{align*}
If $d=1$, this quantity is smaller than a constant as $\vol_M(\BB_M(x,h))\lesssim h^d$ by Lemma \ref{lem:prop_proj}\ref{it:bound_volume}. We then obtain directly the result in this case by integrating this inequality against $f(x)\dd x$. If $d\geq 2$, we use the following argument. 
\begin{itemize}
\item If $\vert x-z\vert \geq 2h$ and $y\in \BB_M(x,h)$, then $\vert z-y\vert \geq \vert x-z\vert -h\geq \vert x-z\vert /2$. Therefore, by Lemma \ref{lem:prop_proj}\ref{it:bound_volume},
\begin{align*}
&\int_{\BB_M(x,h)}\frac{\|K\|_\infty h^{-d}}{\vert z-y\vert ^{d-1}}\dd y \leq \frac{2^{1-d}\|K\|_\infty h^{-d}}{\vert x-z\vert ^{d-1}}\vol_M(\BB_M(x,h)) \lesssim \frac{1}{\vert x-z\vert ^{d-1}}.
\end{align*}
\item If $\vert x-z\vert \leq 2h$, then 
\begin{align*}
\int_{\BB_M(x,h)}\frac{\|K\|_\infty h^{-d}}{\vert z-y\vert ^{d-1}}\dd y &\leq \int_{ \BB_M(z,3h)}\frac{\|K\|_\infty h^{-d}}{\vert z-y\vert ^{d-1}} \dd y  \\
&\leq  \int_{ \BB_{T_z M}(0,3h)}\frac{\|K\|_\infty h^{-d} J\Psi_z(u)}{\vert z-\Psi_z(u)\vert ^{d-1}} \dd u\\
&\lesssim h^{-d}\int_{\BB_{T_z M}(0,3h)}\frac{\dd u}{\vert u\vert ^{d-1}} \lesssim h^{1-d},
\end{align*}
where at the last line we used that $\vert z-\Psi_z(u)\vert \geq \vert u\vert $ and that $J\Psi_z(u)\lesssim 1$ by Lemma \ref{lem:prop_proj}. 
\end{itemize}
Hence, 
\begin{align*}
&\E\left[\vert \nabla (G(K_h*\delta_{X}))(z)\vert ^p\right] =  \int_M f(x)\vert \nabla (G(K_h*\delta_{x}))(z)\vert ^p \dd x \\*
&\leq \fmax \int_M \vert \nabla (G(K_h*\delta_{x}))(z)\vert ^p \dd x \\
&\lesssim \int_{\BB_M(z,2h)} \vert \nabla (G(K_h*\delta_{x}))(z)\vert ^p \dd x+ \int_{M\backslash \BB_M(z,2h)} \vert \nabla (G(K_h*\delta_{x}))(z)\vert ^p \dd x \\
&\lesssim  \int_{\BB_M(z,2h)} h^{(1-d)p} \dd x+ \int_{M\backslash \BB_M(z,2h)} \vert z-x\vert ^{(1-d)p} \dd x \\
&\lesssim h^{(1-d)p+d}+ \int_{M\backslash \BB_M(z,2h)} \vert z-x\vert ^{(1-d)p} \dd x.
\end{align*}
The latter integral is bounded by
\begin{align*}
& \int_{2h\leq \vert x-z\vert \leq r_0}\vert z-x\vert ^{(1-d)p}\dd x +\int_{\vert x-z\vert \geq r_0}\vert z-x\vert ^{(1-d)p}\dd x \\
&\qquad \leq  \int_{2h\leq \vert \Psi_z(u)-z\vert \leq r_0} \vert z-\Psi_z(u)\vert ^{(1-d)p}J\Psi_z(u)\dd u+\vert \vol_M\vert r_0^{(1-d)p} \\
&\qquad\lesssim \int_{14h/8\leq  \vert u\vert \leq r_0}\vert u\vert ^{(1-d)p}\dd u+1 \lesssim h^{(1-d)p+d} \text{ if $(1-d)p+d< 0$},
\end{align*}
where at the last line we use that $\vert u\vert \leq\vert z-\Psi_z(u)\vert \leq 8\vert u\vert /7$ by Lemma \ref{lem:prop_proj}. If $d>2$ or if $d=2$ and $p>2$, the condition $(1-d)p+d< 0$ is always satisfied. If $d=2$ and $p=2$, then $\int_{14h/8\leq  \vert u\vert \leq h_0}\vert u\vert ^{(1-d)p}\dd u$ is of order $-\log h$, concluding the proof.

\section{Proof of Theorem \ref{thm:estimator_M_known}\ref{it:pointwise_control}} \label{sec:pointwise_proof}
Let $f$ be the density of $\mu$ and $\t f=f/\rho_h$. By Lemma \ref{lem:kernel_of_order_k}, $\fmin(1-c_0h^{k-1})\leq \t f\leq \fmax(1+c_0h^{k-1})$ for $h$ small enough. We have
\begin{equation}\label{eq:point1}
\begin{split}
&K_h *\tilde f(x) =\int_M K_h(x-y)\t f(y)\dd y\\
& = \int_{\BB_{T_xM}(0,h)}K_h(x-\Psi_x(v)) \t f\circ \Psi_x(v)  J\Psi_x(v)\dd v \nonumber\\
&\geq \int_{\BB_{T_xM}(0,h)} K_h(v)\t f\circ \Psi_x(v)  J\Psi_x(v)\dd v - \hspace{-.1cm}\int_{\BB_{T_xM}(0,h)}\vert K_h(x-\Psi_x(v))-K_h(v)\vert  \t f\circ \Psi_x(v)  J\Psi_x(v)\dd v. 
\end{split}
\end{equation}
By Lemma \ref{lem:prop_proj}\ref{it:existence_N}, the quantity $\vert K_h(x-\Psi_x(v))-K_h(v)\vert $ is bounded by 
\[\frac{\|K\|_{\CC^1(\R^d)}}{h^{d+1}}\vert x-v-\Psi_x(v)\vert  \lesssim \frac{\vert v\vert ^2}{h^{d+1}},\]
 so that the second term in \eqref{eq:point1} is bounded by $C\fmax \int_{\BB_{T_x M}(0,h)}\frac{\vert v\vert ^2}{h^{d+1}}\dd v \lesssim  h$. Also, using that $\vert J\Psi_x(v)-1\vert \leq c_1 \vert v\vert $ by Lemma \ref{lem:prop_proj}, the first term is larger than
\begin{align*}
&\fmin(1-c_0h^{k-1})(1-c_1h)\int_{\R^d} K_+ - \fmax(1+c_1h)(1+c_0h^{k-1})\int_{\R^d} K_-\\
&=\fmin(1-c_2h)\p{1+\int_{\R^d} K_-} - \fmax(1+c_2h)\int_{\R^d} K_-\\
&=\fmin(1-c_2h) - (\fmax(1+c_2h)-\fmin(1-c_2h))\int_{\R^d} K_-(v) \dd v \\
&\geq \fmin(1-c_2h) -(\fmax(1+c_2h)-\fmin(1-c_2h))\beta \\
&\geq 3\fmin/4,
\end{align*}
if $\beta < \fmin/(4(\fmax-\fmin))$ and $h$ is small enough. Likewise, we show that $K_h*\t f(x)\leq 3\fmax/2$. It remains to show that $\vert K_h*\t f(x)-K_h*(\mu_n/\rho_h)(x)\vert $ is small enough for all $x\in M$ with high probability. Note that $K_h*\t f-K_h*(\mu_n/\rho_h)$ is $L$-Lipschitz with $L\lesssim h^{-d-1}$. Let $t=\fmin/4$ and consider a covering of $M$ by $N$ balls $\BB_M(x_j,t/(2L))$. By standard packing arguments, such a covering exists with $N\lesssim (L/t)^d$. If $\vert K_h*\t f(x_j)-K_h*\mu_n(x_j)\vert \leq t/2$ for all $j=1,\dots,N$, then $\|K_h*\t f-K_h*\mu_n\|_{L_\infty(M)} \leq t/2 + Lt/(2L) \leq t$. Hence,  using Bernstein inequality \cite[Theorem 3.1.7]{gine_nickl_2015}, as $\vert K_h(x_j-Y_i)\vert \leq \|K\|_{\CC^0(\R^D)}h^{-d}$ and $\Var(K_h(x_j-Y_i))\leq \|K^2\|_{\CC^0(\R^D)}h^{-d}$, we obtain
\begin{align*}
&\P(\|K_h*\t f-K_h*\mu_n\|_{L_\infty(M)} \geq t) \leq \P(\exists j,\ \vert K_h*\t f(x_j)-K_h*\mu_n(x_j)\vert \geq t/2) \\
&\qquad \lesssim (L/t)^d\P(\vert K_h*\t f(x_j)-K_h*\mu_n(x_j)\vert \geq t/2) \lesssim h^{-d(d+1)} \exp(-Cnh^d).
\end{align*}
Choosing $nh^d= C'\log n$ for $C'$ large enough yields the conclusion.

\section{Proofs of Section \ref{sec:estim_volume}}\label{sec:last_thm}

We first prove Lemma \ref{lem:pointwise_comparison}.
\medskip

\paragraph{}\textit{Proof of \ref{it:comp_diffeo}.} The application $\Psi_{Y_{j}}\circ\pi_{Y_{j}}: \BB_{\hat T_{j}}(0,3\eps)\to M$ is a diffeomorphism on $\BB_{\hat T_{j}}(0,3\eps)$, as the composition of the diffeomorphisms $\Psi_{Y_{j}}$ and $(\pi_{Y_{j}})_{\vert \hat T_{j}}$ (recall that $\angle(\hat T_{j},T_{Y_{j}}M) \lesssim \eps^{m-1}+ \gamma \eps^{-1} \lesssim 1$ by Proposition \ref{prop:estim_manifold}). Furthermore, by Lemma \ref{lem:prop_proj}\ref{it:param_contains_ball} and the bound on the angle,
\[ \BB_M(Y_j,2\eps)\subset \Psi_{Y_j}(\BB_{T_{Y_j}M}(0,2\eps)) \subset (\Psi_{Y_j}\circ\pi_{Y_j})(\BB_{\hat T_j}(0,3\eps)).\]
This proves the first part of Lemma \ref{lem:pointwise_comparison}\ref{it:comp_diffeo}. Let $S_j:\BB_M(Y_j,2\eps)\to \BB_{\hat T_j}(0,3\eps)$ be the inverse of $\Psi_{Y_j}\circ\pi_{Y_j}$. By Lemma \ref{lem:prop_poly_estim}\ref{it:prop_poly_estim_3}, $\hat\Psi_j$ is injective on $\hat T_j$, while, for $v\in \hat T_j$ with $\vert v\vert  \leq 3\eps$,
\begin{equation}\label{eq:pointwise2}
\op{\id-d\hat \Psi_j(v)} \leq \left\|\sum_{a= 2}^{m-1}a \hat V_{a,j}[\cdot,v^{\otimes (a-1)}] \right\| \lesssim \ell \eps \leq 1/2
\end{equation}
if $\ell\lesssim \eps^{-1}$ is small enough. Hence, $\hat\Psi_j:\BB_{\hat T_j}(0,3\eps)\to \hat\Psi_j(\hat T_j)$ is a diffeomorphism on its image, and $\hat\Psi_j\circ S_j$ is a diffeomorphism as a composition of diffeomorphisms. Note that the inverse of $\hat\Psi_j$ is given by $\hat\pi_j(\cdot-X_j)$, so that $\BB_{\hat\Psi_j(\hat T_j)}(X_j,\eps)\subset \hat\Psi_j(\BB_{\hat T_j}(0,\eps))$. Furthermore, by Lemma \ref{lem:prop_proj},
\[ (\Psi_{Y_j}\circ\pi_{Y_j})(\BB_{\hat T_j}(0,\eps)) \subset \Psi_{Y_j}(\BB_{T_{Y_j}}(0,\eps)) \subset \BB_M(Y_j,8\eps/7),\]
so that $(\hat\Psi_j\circ S_j)(\BB_M(Y_j,2\eps))$ contains $\BB_{\hat\Psi_j(\hat T_j)}(X_j,\eps)$. Furthermore, these inclusions of balls also hold for any $\eps'\leq \eps$, proving that $\vert \hat\Psi_j\circ S_j(z)-X_j\vert \geq (7/8)\vert z-Y_j\vert $ for any $z\in \BB_M(Y_j,2\eps)$.
\medskip

\paragraph{}\textit{Proof of \ref{it:comp_diffeo2}.} The formula for the density $\t\chi_j$ follows from a change of variables.
\medskip

\paragraph{}\textit{Proof of \ref{it:comp_diffeo3}.} The inequality \eqref{eq:first_eq} follows from Proposition \ref{prop:estim_manifold}. We now prove that, for $z\in \BB_M(Y_j,2\eps)$, 
\begin{equation}\label{eq:second_eq}
\vert \pi_{Y_i}(z-\hat \Psi_j\circ S_j(z))\vert \lesssim (\eps+\gamma\eps^{-1})(\eps^{m} + \gamma).
\end{equation}
Let $u\in\hat T_j$ be such that $z=\Psi_{Y_j}\circ\pi_{Y_j}(u)$ and $y=\hat \Psi_j(u)$.  Recall that $X_j\in T_{Y_j}M^\bot$ by assumption, so that $\pi_{Y_j}(X_j-Y_j)=0$. Also, by Lemma \ref{lem:prop_proj}\ref{it:existence_N}, we have $\Psi_{Y_j}(\pi_{Y_j}(u))=Y_j+ \pi_{Y_j}(u) + N_{Y_j}(\pi_{Y_j}(u))$ with $N_{Y_j}(\pi_{Y_j}(u))\in T_{Y_j}M^\bot$, while by Lemma \ref{lem:prop_poly_estim}\ref{it:prop_poly_estim_3}, we have $\hat \Psi_j(u)=X_j+u+\hat N_j(u)$ with $\hat N_j(u)\in \hat T_j^\bot$. Hence,
\begin{align*}
\vert \pi_{Y_j}(&z-y)\vert =\vert \pi_{Y_j}(Y_j+ \pi_{Y_j}(u) + N_{Y_j}(\pi_{Y_j}(u))-(X_j+u+\hat N_j(u)))\vert \\
&= \vert \pi_{Y_j}(N_{Y_j}(\pi_{Y_j}(u))-\hat N_j(u))\vert \\
&\leq \angle(T_{Y_j}M,\hat T_j)\vert N_{Y_j}(\pi_{Y_j}(u))-\hat N_j(u)\vert +\vert \hat\pi_j(N_{Y_j}(\pi_{Y_j}(u))-\hat N_j(u))\vert \\
&\lesssim (\eps^{m-1}+\gamma\eps^{-1})(\eps^{m}+\gamma) + \vert \hat\pi_j(\pi_{Y_j}^\bot(N_{Y_j}(\pi_{Y_j}(u))))\vert \\
&\lesssim(\eps^{m-1}+\gamma\eps^{-1})(\eps^{m}+\gamma) + \angle(T_{Y_j}M,\hat T_j)\vert N_{Y_j}(\pi_{Y_j}(u))\vert \\
&\lesssim(\eps^{m-1}+\gamma\eps^{-1})(\eps^{m}+\gamma + \eps^2) \lesssim (\eps^{m-1}+\gamma\eps^{-1})(\eps^2+\gamma),
\end{align*}
where we used Proposition \ref{prop:estim_manifold} to bound $\angle(T_{Y_j}M,\hat T_j)$, Lemma \ref{lem:prop_poly_estim} to bound $\vert N_{Y_j}(\pi_{Y_j}(u))-\hat N_j(u)\vert $ and Lemma \ref{lem:prop_proj} to bound $\vert N_{Y_j}(\pi_{Y_j}(u))\vert $. We obtain \eqref{eq:second_eq}.

To prove inequality \eqref{eq:third_eq}, we first bound $\vert \chi_j(\hat \Psi_j\circ S_j(z))-\chi_j(z)\vert $ and then bound $\vert J(\hat \Psi_j\circ S_j)(z)-1\vert $. The first bound is based on the following elementary lemma.
\begin{lemma}\label{lem:radial}
Let $\theta:\R^D\to \R$ be a smooth radial function. Then, $\vert \theta(x)-\theta(y)\vert \leq \frac{\|\theta\|_{\CC_2(\R^D)}}{2} \vert \vert x\vert ^2-\vert y\vert ^2\vert $.
\end{lemma}
\begin{proof}
As $d\theta(0)=0$, one can write $\theta(x) =\t\theta(\vert x\vert ^2)$ for some function $\t \theta$ which is Lipschitz continuous with Lipschitz constant $\frac{\|d^2 \theta\|_{\CC^0(\R^D)}}{2}$. This implies the conclusion.
\end{proof}
Recall from the proof of Lemma \ref{lem:existence_pou} that we have $\chi_j(z) =\zeta_j(z)/\sum_{i=1}^J\zeta_{i}(z)$ where $\zeta_i=\theta\p{\frac{z-X_i}{\eps}}$ for some smooth radial function $\theta$, and that furthermore, there is at most $c_d$ non-zero terms in the sum in the denominator, which is always larger than $1$. Hence, if we control for every $i=1,\dots,J$ the difference $\vert \vert z-X_i\vert ^2-\vert \hat\Psi_j\circ S_j(z)-X_i\vert ^2\vert $, then we obtain a control on $\vert \chi_j(z)-\chi_j(\hat\Psi_j\circ S_j(z))\vert $. Let  $z\in M$ be such that $\vert z-X_j \vert\leq 2\eps$ (for otherwise both $\chi_j(z)$ and $\chi_j(\hat\Psi_j\circ S_j(z))$ are equal to zero). We have by \eqref{eq:first_eq} and \eqref{eq:second_eq},
\begin{equation}\label{eq:comp_to_do_again}
\begin{split}
\vert \vert \hat\Psi_j\circ S_j(z)-X_i\vert ^2-\vert &z-X_i\vert ^2\vert = \vert \vert \hat\Psi_j\circ S_j(z)-z\vert ^2 + 2\dotp{\hat\Psi_j\circ S_j(z)-z, z-X_i}\vert \\
&\lesssim  (\eps^m+\gamma)^2 + \vert \dotp{\hat\Psi_j\circ S_j(z)-z, z-Y_i}\vert  +  \vert \dotp{\hat\Psi_j\circ S_j(z)-z, X_i-Y_i}\vert \\
&\lesssim  (\eps^m+\gamma)^2 + \vert \dotp{\pi_{Y_j}(\hat\Psi_j\circ S_j(z)-z), \pi_{Y_j}(z-Y_i)}\vert \\
&\qquad\qquad + \vert \dotp{\pi_{Y_j}^\bot(\hat\Psi_j\circ S_j(z)-z), \pi_{Y_j}^\bot(z-Y_i)}\vert + (\eps^m+\gamma)\gamma\\
&\lesssim  (\eps^m+\gamma)^2 + (\eps+\gamma\eps^{-1})(\eps^m +\gamma)\vert z-Y_i\vert  \\
&\qquad\qquad+ (\eps^m+\gamma)\vert \pi_{Y_j}^\bot(z-Y_i)\vert + (\eps^m+\gamma)\gamma.
\end{split}
\end{equation}
By Lemma \ref{lem:prop_proj}\ref{it:angle} and the fact that $\vert z-Y_j\vert \leq \vert z-X_j\vert +\gamma\lesssim \eps$, we have $\vert \pi_{Y_j}^\bot(z-Y_i)\vert \leq\vert \t\pi_{Y_j}^\bot(z)\vert +\vert \t\pi_{Y_j}^\bot(Y_i)\vert  \lesssim \eps^2+\vert Y_i-Y_j\vert ^2$. Hence, we obtain  that 
\begin{equation}\label{eq:square_diff}
\vert \vert \hat\Psi_j\circ S_j(z)-X_i\vert ^2-\vert z-X_i\vert ^2\vert  \lesssim (\eps^m+\gamma)(\eps^2 +\gamma+ \vert Y_i-Y_j\vert ^2).
\end{equation}
Therefore,
\begin{equation}\label{eq:thetaclose}
\begin{split}
\left\vert \theta\p{\frac{z-X_i}{\eps}}-\theta\p{\frac{\hat\Psi_j\circ S_j(z)-X_i}{\eps}} \right \vert  &\lesssim \frac{(\eps^m+\gamma)(\eps^2 + \gamma+\vert Y_i-Y_j\vert ^2)}{\eps^2} \\
&\lesssim (\eps^m + \gamma)\p{1+ \gamma\eps^{-2}+\frac{\vert Y_i-Y_j\vert ^2}{\eps^2}}.
\end{split}
\end{equation}
Note also that if $\vert Y_i-Y_j\vert \geq 3\eps$, then $\vert z-X_i\vert \geq \vert X_i-X_j\vert -\vert z-X_j\vert  \geq 3\eps  -\eps-3\gamma\geq \eps$, while by the same argument $\vert \hat\Psi_j\circ S_j(z)-X_i\vert \geq \eps$. Hence, both terms in the left-hand side of \eqref{eq:thetaclose} are null in that case. Thus, we may assume that $\vert Y_i-Y_j\vert \leq 3\eps$, so that $\left\vert \theta\p{\frac{z-X_i}{\eps}}-\theta\p{\frac{\hat\Psi_j\circ S_j(z)-X_i}{\eps}} \right \vert  \lesssim (\eps^m + \gamma)(1+\gamma \eps^{-2})$. From the definition of $\chi_j(z)$, and as the function $t\mapsto 1/t$ is Lipschitz on $[1,\infty[$, we obtain that 
\[\vert \chi_j(z)-\chi_j(\hat\Psi_j\circ S_j(z))\vert \lesssim (\eps^m+\gamma)(1+\gamma\eps^{-2}).\]
We now provide a bound on $\vert J(\hat \Psi_j\circ S_j)(z)-1\vert $. One has, for $u= S_j(z)\in \hat T_j$,
\begin{align*}
\vert J(\hat \Psi_j\circ S_j)(z)-1\vert &= \frac{\vert J\hat\Psi_j(u)-J(\Psi_{Y_j}\circ\pi_{Y_j})(u)\vert }{J(\Psi_{Y_j}\circ\pi_{Y_j})(u)}.
\end{align*}
By Lemma \ref{lem:prop_proj}\ref{it:existence_N} and Lemma \ref{lem:prop_poly_estim}\ref{it:prop_poly_estim_3},  $\op{\id_{\hat T_j}-d(\Psi_{Y_j}\circ\pi_{Y_j})(u)}\lesssim 1 $ and $\op{\id_{\hat T_j}-d\hat\Psi_j(u)}\lesssim 1$ for $u$ small enough. As a consequence, both Jacobians are larger than, say $1/2$ for $u$ small enough, and, as the function $A\in \R^{d\times d} \mapsto \sqrt{\det(A)}$ is Lipschitz continuous on the set of matrices with $\det(A)\geq 1/2$ and $\op{A}\leq 2$, we have 
\begin{align}\label{eq:control_jacob}
\vert J(\hat \Psi_j\circ S_j)(z)-1\vert &\lesssim \hspace{-.1cm} \op{d\hat\Psi_j(u)^\top  d\hat\Psi_j(u)-d(\Psi_{Y_j}\circ\pi_{Y_j})(u)^\top  d(\Psi_{Y_j}\circ\pi_{Y_j})(u)}\hspace{-.1cm}.
\end{align}
Recall that $\hat\Psi_j(u)=X_j+u+\hat N_j(u)$ and $\Psi_{Y_j}\circ\pi_{Y_j}(u)=Y_j+\pi_{Y_j}(u)+N_{Y_j}\circ\pi_{Y_j}(u)$. We may write
\begin{align*}
&d\hat\Psi_j(u)^\top  d\hat\Psi_j(u) = \mathrm{id}_{\hat T_j} + (d\hat N_j(u))^\top  d \hat N_j(u) \quad \text{ and}\\
&d(\Psi_{Y_j}\circ\pi_{Y_j})(u)^\top  d(\Psi_{Y_j}\circ\pi_{Y_j})(u) = \hat\pi_j\pi_{Y_j}\hat\pi_j + (d(N_{Y_j}\circ\pi_{Y_j})(u))^\top  d(N_{Y_j}\circ\pi_{Y_j})(u).
\end{align*}
One has $\op{\mathrm{id}_{\hat T_j}-\hat\pi_j\pi_{Y_j}\hat\pi_j }=\op{\hat \pi_j \pi_{Y_j}^\bot \pi_{Y_j}^\bot \hat\pi_j}\leq \angle(T_{Y_j}M,\hat T_j)^2 \lesssim (\eps^{m-1}+\gamma\eps^{-1})^2\leq (\eps^m+\gamma)(1+\gamma\eps^{-2})$. Furthermore, by Lemma \ref{lem:prop_poly_estim}\ref{it:prop_proly_estim_4},
\begin{align*}
&\op{(d\hat N_j(u))^\top  d \hat N_j(u)-(d(N_{Y_j}\circ\pi_{Y_j})(u))^\top  d(N_{Y_j}\circ\pi_{Y_j})(u)} \\
&\qquad\leq \p{\op{d \hat N_j(u)}+\op{d(N_{Y_j}\circ\pi_{Y_j})(u)}}\op{d \hat N_j(u)-d(N_{Y_j}\circ\pi_{Y_j})(u)} \\
&\qquad \lesssim \eps(\eps^{m-1}+\gamma\eps^{-1})\lesssim \eps^m+\gamma.
\end{align*}
Putting together \eqref{eq:control_jacob} with those two inequalities, we obtain that $\vert J(\hat \Psi_j\circ S_j)(z)-1\vert \lesssim (\eps^m+\gamma)(1+\gamma\eps^{-2})$, concluding the proof of Lemma \ref{lem:pointwise_comparison}.

To conclude the proof of Theorem \ref{thm:M_unknown}, it remains to control the quantity $T$ appearing in Lemma \ref{lem:the_smart_lemma} for $\phi=K_h*(\nu_n/\hat\rho_h)$ and $\phi'=K_h*(\mu_n/\rho_h)$.
\begin{lemma}\label{lem:bound_on_T}
The quantity $T=\max_{j=1\dots J}\sup_{z\in \BB(Y_j,\eps)}\vert \phi(\hat\Psi_j\circ S_j(z))-\phi'(z)\vert $ satisfies $T\lesssim (\eps^m + \gamma)(1+\gamma\eps^{-2})$ with probability larger than $1-cn^{-k/d}$.
\end{lemma}
\begin{proof}
For $z\in \BB(Y_j,\eps)$, we have
\[ \vert \phi(\hat\Psi_j\circ S_j(z))-\phi'(z)\vert \leq \frac{1}{n}\sum_{i=1}^n \left\vert  \frac{K_h*\delta_{X_i}(\hat\Psi_j\circ S_j(z))}{\hat\rho_h(X_i)}- \frac{K_h*\delta_{Y_i}(z)}{\rho_h(Y_i)}\right \vert .\]
The same computation than in \eqref{eq:comp_to_do_again} shows that
\[ \vert \vert \hat\Psi_j\circ S_j(z)-Y_i\vert ^2-\vert z-Y_i\vert ^2\vert  \lesssim (\eps^m+\gamma)(\eps^2 +\gamma+ \vert Y_i-Y_j\vert ^2).\]
This inequality together with Lemma \ref{lem:radial} yield
\begin{align*}
&\vert K_h(X_i-\hat\Psi_j\circ S_j(z))-K_h(Y_i-z)\vert \\
&\qquad\lesssim h^{-d-2}(\eps^m+\gamma)(\eps^2 +\gamma+ \vert Y_i-Y_j\vert ^2).
\end{align*}
We may assume that $\vert Y_i-Y_j\vert \leq 3h$ and $\vert z-Y_i\vert \leq 2h$, for otherwise both quantities in the left-hand site of the above equation are equal to zero. Hence, as $\eps\lesssim h$ by assumption, we have 
\begin{equation}\label{eq:control_kernel}
\vert K_h(X_i-\hat\Psi_j\circ S_j(z))-K_h(Y_i-z)\vert \lesssim h^{-d}(\eps^m+\gamma)(1+\gamma\eps^{-2})\ones\{Y_i\in \BB_M(z,2h)\}.
\end{equation}
Let us now bound $\vert \hat\rho_h(X_i)-\rho_h(Y_i)\vert $. By the triangle inequality, and using \eqref{eq:third_eq} and \eqref{eq:control_kernel}, we obtain that this quantity is smaller than
\begin{align*}
&\vert \sum_{j=1}^J \int_{\hat\Psi_j(\hat T_j)}\chi_j(w)K_h(X_i-w)\dd w - \sum_{j=1}^J \int_M \chi_j(z) K_h(Y_i-z)\dd z\vert\\
&\leq \sum_{j=1}^J \int_M\vert \t\chi_j(z)K_h(X_i-\hat\Psi_j\circ S_j(z))-\chi_j(z)K_h(Y_i-z)\vert \dd z \\
&\lesssim \sum_{j=1}^J \int_M( \ones\{z\in\BB_M(Y_j,2\eps)\}(\eps^m+\gamma)(1+\gamma\eps^{-2})\vert K_h(Y_i-z)\vert \\*
&\qquad\qquad\qquad+ \t\chi_j(z)h^{-d}(\eps^m+\gamma)(1+\gamma\eps^{-2})\ones\{z\in \BB_M(Y_i,2h)\})\dd z\\
&\lesssim  h^{-d}(\eps^m+\gamma)(1+\gamma\eps^{-2})\sum_{j=1}^J \int_M \ones\{z\in\BB_M(Y_j,2\eps)\}\ones\{z\in \BB_M(Y_i,2h)\}\dd z \\
&\lesssim \eps^dh^{-d}(\eps^m+\gamma)(1+\gamma\eps^{-2})\sum_{j=1}^J \ones\{\vert Y_j-Y_i\vert \leq 4h\} \\
&\lesssim h^{-d}(\eps^m +\gamma)(1+\gamma\eps^{-2})\sum_{j=1}^J \ones\{\vert Y_j-Y_i\vert \leq 4h\}\vol_M(\BB_M(Y_j,\eps/8))\\
&\lesssim  h^{-d}(\eps^m +\gamma)(1+\gamma\eps^{-2})\vol_M(\BB_M(Y_i,5h))\lesssim (\eps^m +\gamma)(1+\gamma\eps^{-2}),
\end{align*}
where  we use that $\{X_1,\dots,X_J\}$ is $7\eps/24$-sparse, so that $\{Y_1,\dots,Y_J\}$ is $\eps/4$-sparse. Therefore, the balls $\BB_M(Y_j,\eps/8)$ for $\vert Y_j-Y_i\vert \leq 4h$ are pairwise distincts, and are all included in $\BB_M(Y_i,4h+\eps/8)\subset \BB_M(Y_i,5h)$. We conclude by Lemma \ref{lem:prop_proj}\ref{it:bound_volume}.
Letting $N(z,2h)$ be the number of points $Y_i$ belonging to $\BB_M(z,2h)$, we obtain
\begin{align*}
&\vert \phi(\hat\Psi_j\circ S_j(z))-\phi'(z)\vert \\
&\lesssim \frac{1}{n}\sum_{i=1}^n (\vert K_h(Y_i-z)\vert (\eps^m +\gamma)(1+\gamma\eps^{-2})+h^{-d}(\eps^m+\gamma)(1+\gamma\eps^{-2})\ones\{Y_i\in \BB_M(z,2h)\})\\
&\lesssim \frac{N(z,2h)}{nh^{d}}(\eps^m+\gamma)(1+\gamma\eps^{-2}).
\end{align*}
If, for every $z\in M$ and some $\lambda >0$, $N(z,2h)\leq \lambda nh^d$, then we have the conclusion. Let us bound
\[ P_0= \P(\exists z\in M,\ N(z,2h)> \lambda nh^d). \]
If $N(z,2h)>\lambda nh^d$, then there exists a point $Y_i$ with $N(Y_i,4h)\geq N(z,2h)>\lambda nh^d$. Hence, $P_0 \leq n\P(N(Y_1,4h)>\lambda nh^d)$. Conditionally on $Y_1$, $N(Y_1,4h)=1+U$ with $U$ a binomial random variable of parameters $n-1$ and  $\mu(\BB_M(Y_1,4h))\leq \fmax \vol_M(\BB_M(Y_1,4h))\lesssim h^d$ (see Lemma \ref{lem:prop_proj}\ref{it:bound_volume}). In particular, for $\lambda$ large enough, the probability $P_0$ is smaller than $n^{-k/d}$ by Bernstein's inequality, as long as $nh^d\gtrsim 1$.
\end{proof}

We conclude this section by giving a proof of Proposition \ref{prop:hat_u_m_approx}.
\begin{proof}
Recall that $W_1,\dots,W_N$ is a $N$-sample of law $\hat U_M$, and that we are in the noiseless regime $\gamma =0$ with $m=k$. Define $j_a$ the index with $W_a\in \hat\Psi_{j_a}(\hat T_{j_a})$, and let $H_a = (\hat \Psi_{j_a} \circ S_{j_a})^{-1}(W_a)$. Then, Lemma \ref{lem:pointwise_comparison} implies that $\vert W_a-H_a\vert \lesssim \eps^k$. Furthermore, the sample $H_1,\dots,H_N$ has a law $\mu_H$ with density $\sum_{j=1}^J \tilde \chi_j$ on $M$.
We decompose the distance into
\begin{align*}
W_\infty((\hat U_M)_N,U_M) &\leq W_\infty((\hat U_M)_N,N^{-1}\sum_{a=1}^N \delta_{H_a}) + W_\infty(N^{-1}\sum_{a=1}^N \delta_{H_a},\mu_H) + W_\infty(\mu_H,U_M).
\end{align*} 
The first term is of order $\eps^k$, while the second term scales as the second term of \eqref{eq:hat_u_m_approx} according to \cite{trillos2020error}. The third term was already shown to be bounded by $\eps^k$ in the proof of Lemma \ref{lem:the_smart_lemma} (with $\phi=\tilde \phi=1$). As $\eps^k \simeq (\log n/n)^{k/d}$, this concludes the proof.
\end{proof}

\section{Lower bounds on minimax risks}\label{sec:lowerbounds}
In this section, we prove the different lower bounds on minimax risks stated in the article. The main tool used will be Assouad's lemma. Fix a statistical model $(\QQ,\vartheta,\LL)$, where we observe a sample of law $\mu\in\QQ$, while $\vartheta(\mu)$ is a quantity of interest to be estimated, with risk measured by the loss function $\LL$.

\begin{lemma}[Assouad's lemma \cite{yu1997assouad}]
Let $m\geq 1$ be an integer and let $\QQ_m = \{\mu_\sigma,\ \sigma\in \{-1,1\}^m\}\subset \QQ$ be a set of probability measures.  
Assume that for all $\sigma,\sigma' \in \{-1,1\}^m$, 
\begin{equation}
\LL(\vartheta(\mu_\sigma),\vartheta(\mu_{\sigma'})) \geq \vert \sigma-\sigma'\vert \delta ,
\end{equation}
where $\vert \sigma-\sigma'\vert =\sum_{i=1}^m \ones\{\sigma(i)\neq\sigma'(i)\}$ is the Hamming distance between $\sigma$ and $\sigma'$. Then,
\begin{equation}
\RR_n(\vartheta,\QQ,\LL) \geq m\frac{\delta}{16} \p{1-\max\left\{\TV(\mu_\sigma,\mu_{\sigma'}),\ \vert \sigma - \sigma'\vert =1\right\}}^{2n}.
\end{equation}
\end{lemma}

The lower bound on the minimax rates we prove are actually going to hold on the smaller model of uniform distributions on manifolds.

\begin{definition}
Let $k\geq 2$ and $\gamma \geq 0$.  The set $\QQ^{k}_{d}$ is the set of uniform distributions on some manifold $M\in\MM^k_d$ with $\fmax^{-1}\leq \vert \vol_M\vert \leq \fmin^{-1}$.
\end{definition}
One can check that $\QQ^k_d\subset \QQ^{k,s}_d$, with parameter $L_s= \fmin^{-1/p}\vee \fmax^{1-1/p}$. Therefore, a lowerbound on the minimax risk on the model $\QQ^k_d$ yields a lowerbound on the minimax risk on the model $\QQ^{k,s}_d$ should the parameter $L_s$ be large enough.

We build a subfamily of manifolds indexed by $\sigma \in \{-1,1\}^m$ following \cite{aamari2019nonasymptotic}. By \cite[Section C.2]{aamari2019nonasymptotic}, there exists a $d$-dimensional manifold $M \subset \R^{d+1}$ of reach $2\taumin$, of volume $C_d\taumin^d$ which contains $\BB_{\R^d}(0,\taumin)\times \{0\}$ (that we identify with $\BB_{\R^d}(0,\taumin)$). Let $\delta>0$ and consider a family of $m$ points $x_1,\dots,x_m \in \BB_{\R^d}(0,\taumin/2)$, with $\vert x_i-x_{i'}\vert \geq 4\delta$ for $i\neq i'$ and $c_d(\taumin/\delta)^d\leq m \leq C_d(\taumin/\delta)^d$. Let $0<\Lambda <\delta$ and let $\phi:\R^{d+1}\to [0,1]$ be a smooth radial function supported on $\BB(0,1)$, with $\phi\equiv 1$ on $\BB(0,1/2)$. Let $e$ be the unit vector in the $(d+1)$th direction. We then let, for $\sigma\in \{-1,1\}^m$,
\begin{equation}
\Phi_\sigma^{\Lambda}(x) =x + \sum_{i=1}^m \frac{\sigma(i) +1}{2}\Lambda\phi\p{\frac{x-x_i}{\delta}} e.
\end{equation}
Let $M_\sigma^\Lambda = \Phi_\sigma^\Lambda(M)$ and $\mu_\sigma^\Lambda$ be the the uniform measure on $M_\sigma^\Lambda$. Informally, the manifold $M_\sigma^\Lambda$ is obtained by adding bumps of height $\Lambda$ to the base manifold $M$ at locations $x_i$ such that $\sigma(i)=+1$.
If $\Lambda \leq c_{k,d,\taumin}\delta^k$, then $\mu_\sigma^\Lambda \in \QQ^{k}_{d}$, provided that  $L_k$ is large enough \cite[Lemma C.13]{aamari2019nonasymptotic}. 
If $\sigma(i)=1$, the volume of $\Phi_\sigma^\Lambda(\BB_{\R^{d}}(x_i,\delta))$ satisfies (with $\omega_d$ denoting the volume of the $d$-dimensional unit ball)
\begin{align*}
&\left\vert \vol_{M_\sigma^\Lambda}(\Phi_\sigma^\Lambda(\BB_{\R^{d}}(x_i,\delta))-\omega_d\delta^d\right \vert  \leq \int_{\BB_{\R^{d}}(x_i,\delta)} \vert J\Phi_\sigma^\Lambda(x)-1\vert \dd x \\
&\qquad\leq  \int_{\BB_{\R^{d}}(x_i,\delta)} \left\vert \sqrt{1+\Lambda^2\delta^{-2}\left\vert \nabla\phi\p{\frac{x-x_i}{\delta}}\right \vert ^2}-1\right \vert \dd x \leq C_d \delta^d\Lambda^2 \delta^{-2}.
\end{align*} 
Hence, for $\delta$ small enough, we have $\vert \vert \vol_{M_\sigma^\Lambda}\vert -C_d\taumin^d\vert \leq mC_d\delta^d\Lambda^2\delta^{-2} \leq C_d\taumin^d/3$, as $m\leq C_d(\taumin/\delta)^d$ and $\Lambda\leq c_{k,d,\taumin}\delta^k$. As a consequence, if $\vert \sigma- \sigma'\vert =1$, with for instance $\sigma(i)=1$ and $\sigma'(i)=-1$, then
\begin{align}
\TV(\mu_\sigma^\Lambda,\mu_\sigma'^\Lambda) &\leq \max(\mu_{\sigma}^\Lambda(\Phi_\sigma^\Lambda(\BB_{\R^{d}}(x_i,\delta))),\mu_{\sigma'}^\Lambda(\BB_{\R^{d}}(x_i,\delta))\leq  C_{d,\taumin}\delta^d.
\end{align}
We may now prove the different minimax lower bounds using Assouad's Lemma on the family $\{\mu_\sigma^\Lambda,\ \sigma\in \{-1,1\}^m\}$.

\begin{proof}[Proof of Theorem \ref{thm:choice_of_loss}]
As $g$ is nondecreasing and convex, by Jensen's inequality, we may assume without loss of generality that $\LL=\TV$. Let $\Gamma = \vert (\mu_\sigma^\Lambda-\mu_{\sigma'}^\Lambda)(B_i)\vert $, where $B_i= \BB_{\R^{d}}(x_i,\delta)$ and $\sigma(i)\neq \sigma'(i)$. Then, $\TV(\mu_\sigma^\Lambda,\mu_{\sigma'}^\Lambda)\geq \vert \sigma-\sigma'\vert \Gamma$. Furthermore, if for instance $\sigma'(i)=1$, $\Gamma\geq \mu_{\sigma'}^\Lambda(B_i) = (\omega_d \delta^d)/\vert \vol_{M_{\sigma'}^\Lambda}\vert  \geq c_d \delta^d/\taumin^d.$
 By Assouad's Lemma,
\begin{align*}
\RR_n(\mu,\QQ^{s,k}_d,\TV) &\geq\RR_n(\mu,\QQ^{k}_d,\TV) \geq \frac{m}{16} c_d \frac{\delta^d}{\taumin^d} \p{1-C_{d,\taumin}\delta^d}^{2n}\\
&\geq C_d\p{1-C_{d,\taumin}\delta^d}^{2n}.
\end{align*} 
We obtain the conclusion by letting $\delta$ go to $0$.
\end{proof}

\begin{proof}[Proof of Theorem \ref{thm:estimator_volume}\ref{it:minimax_volume}]
As, $W_r\geq W_1$, we may assume that $r=1$. Let $\sigma,\sigma'\in \{-1,1\}^m$ with $\sigma(i)\neq \sigma'(i)$. Let $p_{\sigma(i)}=\vol_{M_\sigma^\Lambda}(\BB(x_i,\delta))$ and $U_{\sigma,i}^\Lambda = p_{\sigma(i)}^{-1}(\vol_{M_\sigma^\Lambda})_{\vert \BB(x_i,\delta)}$. By the Kantorovitch-Rubinstein duality formula, $W_1(\mu,\nu) = \max \int f\dd(\mu-\nu)$, where the maximum is taken over all $1$-Lipschitz continuous functions $f:\R^D\to \R$. Recall that $e$ is the unit vector in the $(d+1)$th direction and let $f:x\mapsto x\cdot e$. Assume for instance that $\sigma(i)=-1$ and $\sigma'(i)=1$. We have $f(x) = 0$ for $x\in \BB_{M_\sigma^\Lambda}(x_i,\delta)$ and $f(x)=\Lambda$ for $x\in \BB_{M_{\sigma'}^\Lambda}(x_i,\delta/2)$. Therefore, we have, as $p_{\sigma'(i)}\leq c\delta^{-d}$,
\begin{align*}
W_1(U_{\sigma,i}^\Lambda,U_{\sigma',i}^\Lambda ) \geq p_{\sigma'(i)}^{-1}\Lambda \omega_d(\delta/2)^d \geq c_1\Lambda.
\end{align*}
Note also that $\vert p_{\sigma(i)}-p_{\sigma'(i)}\vert \leq \left\vert \vol_{M_\sigma^\Lambda}(\Phi_\sigma^\Lambda(\BB_{\R^{d}}(x_i,\delta))-\omega_d\delta^d\right \vert  \leq C_d\delta^{d}\Lambda^2 \delta^{-2}$. Furthermore, $\vert \vert \vol_{M_\sigma^\Lambda}\vert -\vert \vol_{M_{\sigma'}^\Lambda}\vert \vert  \leq \sum_{i=1}^m\vert p_{\sigma(i)}-p_{\sigma'(i)}\vert \leq \vert \sigma-\sigma'\vert C_d\delta^{d}\Lambda^2 \delta^{-2}$. Let $f_i$ be a $1$-Lipschitz continuous function such that $W_1(U_{\sigma,i}^\Lambda,U_{\sigma',i}^\Lambda)=\int f_i d(U_{\sigma,i}^\Lambda-U_{\sigma',i}^\Lambda)$. One can choose $f_i$ such that $f_i(x_i)=0$, so that the maximum of $\vert f_i\vert $ on $\BB(x_i,\delta)$ is at most $\delta$. One can then change the value of $f_i$ outside the ball without changing the value of the integral, so that $f_i$ is supported on $\BB(x_i,2\delta)$ and is $1$-Lipschitz continuous. Consider the function $f$ obtained by gluing together the different functions $f_i$. The function $f$ is $1$-Lipschitz continuous, so that
\begin{align*}
&W_1\p{\mu_\sigma^\Lambda,\mu_{\sigma'}^\Lambda} \geq \sum_{i=1}^m  \p{ \frac{p_{\sigma(i)}}{\vert \vol_{M_\sigma^\Lambda}\vert }U_{\sigma,i}^\Lambda-\frac{p_{\sigma'(i)}}{\vert \vol_{M_{\sigma'}^\Lambda}\vert }U_{\sigma',i}^\Lambda }(f)\\
& \geq  \sum_{i=1}^m \frac{p_{\sigma(i)}}{\vert \vol_{M_\sigma^\Lambda}\vert }(U_{\sigma,i}^\Lambda-U_{\sigma',i}^\Lambda)(f) - \vert p_{\sigma(i)}-p_{\sigma'(i)}\vert \frac{\vert U_{\sigma',i}^\Lambda(f)\vert }{\vert \vol_{M_\sigma^\Lambda}\vert }- p_{\sigma'(i)}\vert U_{\sigma',i}^\Lambda(f)\vert \left\vert \frac{1}{\vert \vol_{M_\sigma^\Lambda}\vert }-\frac{1}{\vert \vol_{M_{\sigma'}^\Lambda}\vert }\right \vert \\
&\geq \sum_{i=1}^m  \frac{p_{\sigma(i)}}{\vert \vol_{M_\sigma^\Lambda}\vert }W_1(U_{\sigma,i}^\Lambda,U_{\sigma',i}^\Lambda) -\sum_{i=1}^m c_4\vert p_{\sigma(i)}-p_{\sigma'(i)}\vert \delta\ones\{\sigma(i)\neq \sigma'(i)\}- c_5\delta \vert \sigma-\sigma'\vert \delta^{d}\Lambda^2 \delta^{-2}\\
&\geq \sum_{i=1}^m  \ones\{\sigma(i)\neq \sigma'(i)\} (c_6\delta^d\Lambda-c_4\delta^{d}\Lambda^2 \delta^{-1}) - c_5\delta \vert \sigma-\sigma'\vert \delta^{d}\Lambda^2 \delta^{-2} \geq c_7\delta^d \Lambda \vert \sigma-\sigma'\vert,
\end{align*}
where we used at the last line that we choose $\Lambda\leq c\delta^2$ for some constant $c$ small enough. More precisely, we let $\Lambda = c_{k,d,\taumin,L_k}\delta^k$ and $\delta=n^{-1}$, and obtain, by Assouad's Lemma,
\[
\RR_n\p{\frac{\vol_M}{\vert \vol_M\vert },\QQ^k_d(\gamma),W_r}\gtrsim n^{-k/d}.
\]
\end{proof}

\begin{proof}[Proof of Theorem \ref{thm:estimator_M_known}\ref{it:noiseless_minimax}]
Let $a_n=n^{-\frac{s+1}{2s+d}}$ if $d\geq 3$ and $a_n = n^{-1/2}$ if $d\leq 2$. As $W_p\geq W_1$, we may assume without loss of generality that $p=1$, and up to rescaling, we assume that $\taumin=\sqrt{d}$. Consider the manifold $M \subset \R^{d+1}$ containing $\BB_{\R^d}(0,\sqrt{d})$ of the previous proof. In particular, $M$ contains the cube $[-1,1]^d$. We adapt the proof of Theorem 3 in \cite{weed2019estimation}, where authors consider a family of functions $f_\sigma: [-1,1]^d \to M$ indexed by $\sigma\in \{-1,1\}^m$, with $f_\sigma = 1+n^{-1/2}\sum_{j=1}^m \sigma_j\psi_j$, where $(\psi_j)_{j=1,\dots,m}$ are elements of a wavelet basis of $[-1,1]^d$ that satisfy $\int \psi_j=0$ (see \cite[Appendix E]{weed2019estimation}  for details on the construction of the wavelet basis). If $m\lesssim n^{d/(2s+d)}$, then $t_0\leq f_\sigma\leq t_1$ for some positive constants $t_0<1<t_1$, and $\|f_\sigma\|_{B^s_{p,q}([-1,1]^d)} \lesssim 1$. Note that each $\psi_j$ is supported on a small rectangle inside $[-1,1]^d$, and can be extended to a smooth function on $M$ (by simply defining $\psi_j=0$ outside $[-1,1]^d$). Therefore, we can also consider $f_\sigma$ as being defined on $M$. This extension (that we still denote by $f_\sigma$) also satisfies $t_0\leq f_\sigma\leq t_1$ and $\|f_\sigma\|_{B^s_{p,q}(M)} \lesssim 1$ (this last inequality is clear for the  $\|\cdot\|_{H^l_p(M)}$ norm for $l$ an integer, while the result follows from interpolation for Besov spaces \cite[Corollary 1.1.7]{lunardi2018interpolation}).

As $\int\psi_j=0$, we have $\int_M f_\sigma = \vert \vol_M \vert$. Let $\tilde f_\sigma = f_\sigma/\vert \vol_M \vert$, that is larger than $\fmin = t_0/ \vert \vol_M\vert $ and smaller than $\fmax=t_1/ \vert \vol_M\vert $. Hence, identifying measures with their densities, the set 
\[ \QQ_m = \{\t f_{\sigma},\ \sigma\in \{-1,1\}^m\}\]
is a subset of $\QQ^{s,k}_d$ for $\fmin$ small enough and $L_k$, $L_s$, $\fmax$ large enough. Furthermore, for $\sigma, \sigma'\in \{-1,1\}^m$, we have $\TV(\t f_\sigma,\t f_{\sigma'}) = \TV(f_\sigma,f_\sigma')/\vert\vol_M\vert$. Also, for any function $\phi:\R^{d+1}\to \R$ that is $1$-Lipschitz, we have
\begin{align*}
\int_M \phi(\t f_\sigma - \t f_{\sigma'}) &= \int_{[-1,1]^d} \frac{\phi(f_{\sigma} - f_{\sigma'})}{\vert\vol_M\vert},
\end{align*}
so that $W_1(\t f_\sigma,\t f_{\sigma'}) = W_1(f_\sigma,f_\sigma')/\vert\vol_M\vert$. Hence, we have reduced our problem to the case of the cube, and applying Assouad's inequality in the same fashion than in \cite[Theorem 3]{weed2019estimation} yields that $\RR_n(\mu,\QQ^{s,k}_d,W_1)\gtrsim a_n$.
\end{proof}

\section{Existence of kernels satisfying conditions \texorpdfstring{$A$}{A}, \texorpdfstring{$B(m)$}{B(m)} and \texorpdfstring{$C(\beta)$}{C(beta)}}\label{sec:kernel}
The goal of the section is to prove the existence of a kernel $K$ satisfying the conditions $A$, $B(m)$ and $C(\beta)$ stated at the beginning of Section \ref{sec:def_estim}.

If $K$ is a radial kernel, we have by integration by parts, as $K$ is smooth with compact support,
\begin{align*}
 \int_{\R^d}\partial^{\alpha_0}K(v)v^{\alpha_1}\dd v &= C_{\alpha_0,\alpha_1}\int_{\R^d}K(v)v^{\alpha_1+\alpha_0}\dd v\\
 &= C'_{\alpha_0,\alpha_1}\int_{\R} K(r)r^{d+\vert  \alpha_0 \vert +\vert  \alpha_1 \vert -1}\dd r.
\end{align*}
Hence, to show the existence of such a kernel, it suffices to find, for every $m\geq 0$ and every positive constant $\kappa$, a smooth even function $K:\R\to \R$ supported on $[-1,1]$ satisfying 
\begin{itemize}
\item \textbf{Condition $A'$:} $\int_{\R} K(r)r^{d-1}\dd r= \kappa$,
\item \textbf{Condition $B'(m)$:} $\int_{\R} K(r)r^{d+i-1}\dd r=0$ for $i= 1,\dots,m$,
\item \textbf{Condition $C'(\beta)$:} $\int_{\R} K(r)^- r^{d-1}\dd r \leq \beta$.
\end{itemize}
We show by recursion on $m$ that for any $\beta>0$, there exists a such a kernel. For $m=0$, let $K_0$ be any smooth even nonnegative function supported on $[-1,1]$. Then, letting $K=\kappa K_0/\int_{\R}K_0$, we obtain a kernel $K$ satisfying the desired conditions for any $\beta>0$. Consider now the case $m>0$. Let $\beta >0$. 
\begin{itemize}
\item If $m+d$ is even, then any $K$ satisfying conditions $A'$, $B'(m-1)$ and $C'(\beta)$ will also satisfy $B'(m)$. Indeed, as $K$ is even, we have $\int_{\R} K(r)r^{m+d-1}\dd r=0$, so that the induction step is proven.
\item If $m+d$ is odd, let $K$ be a kernel satisfying conditions $A'$, $B'(m-1)$ and $C'(\beta/2)$. We use the following lemma.
\end{itemize}

\begin{lemma}\label{lem:appendix_kernel}
For $i\geq 0$, let $e_i:x\in \R \mapsto x^{i+d-1}$ and fix an integer $m>0$. For any $a\in \R$, let $F_a$ be the set of smooth functions $f:(1,\infty)\to \R$ with compact support satisfying $\int fe_i=0 \text{ for } 0\leq i<m\text{ and } \int fe_{m}=a$. Then,
\begin{equation}
\inf\left\{\int \vert f(r)\vert r^{d-1}\dd r,\ f\in F_a \right\} = 0.
\end{equation}
\end{lemma}

Assume first that the lemma is true. Let $a=-\frac{1}{2}\int_{\R} K(r)r^{m+d-1}$ and $f\in F_a$. Then,
\begin{align*}
\begin{cases}
\int (K(r)+f(\vert r\vert ))r^{d-1}\dd r= \kappa+\int f(\vert r\vert )r^{d-1}\dd r=\kappa \\
\int (K(r)+f(\vert r\vert ))r^{i+d-1} \dd r=\int f(\vert r\vert )r^{i+d-1}\dd r = 0\text{ for } 0<i<m \\
\int (K(r)+f(\vert r\vert ))r^{m+d-1} \dd r=\int K(r)r^{m+d-1}\dd r + 2\int_1^{\infty} f(r)r^{m+d-1}\dd r = 0.
\end{cases}
\end{align*} 
Hence, the kernel $K+f(\vert \cdot\vert )$ satisfies the conditions $A$ and $B'(m)$. Also, we have, as $K(r)=0$ if $\vert r\vert \geq 1$,
\begin{align*}
&\int_{\R} (K(r)+f(\vert r\vert ))_-r^{d-1} \dd r= \int_{\R} K(r)_- \dd r + 2\int_1^{\infty}f(r)_-r^{d-1}\dd r \\
&\qquad \leq \beta/2 + \int_1^\infty \vert f(r)\vert r^{d-1}\dd r,
\end{align*}
where we used at the last line that $\int_1^{\infty}f(r)_-r^{d-1}\dd r=\int_1^{\infty}f(r)_+r^{d-1}\dd r=\frac{1}{2}\int_1^{\infty}\vert f(r)\vert r^{d-1}\dd r$.
Lemma \ref{lem:appendix_kernel} asserts the existence of $f\in F_a$ with $\int \vert f(r)\vert r^{d-1}\dd r\leq  \beta /2$. For such a choice of $f$, the kernel $\t K=K+f(\vert \cdot\vert )$ satisfies also $C'(\beta)$. Finally, $f$ has a compact support, included in $[0,R]$ for some $R>0$. The kernel $\t K_{1/R}$  is supported on $\BB(0,1)$, and satisfies conditions $A'$, $B'(m)$ and $C'(\beta)$. This concludes the induction step, and the proof of the existence of kernels satisfying conditions $A$, $B(m)$ and $C(\beta)$.

\begin{proof}[Proof of Lemma \ref{lem:appendix_kernel}]
Consider functions $f$ supported on $[r_0,r_1]$ for some constants $1< r_0\leq r_1$ to fix. Let $G_{r_0,r_1}$ be the subspace of $L_2([r_0,r_1])$ spanned by the functions $e_i$ for $0\leq i \leq m-1$ and let $g_{m}$ be the projection of $e_{m}$ on $G_{r_0,r_1}^\bot$, the orthogonal space of $G_{r_0,r_1}$. Let $\ell=\|g_m\|_{L_2[r_0,r_1]}$. The function $f = \frac{ag_{m}}{\ell^2} $ is a polynomial of degree $m$ restricted to $[r_0,r_1]$ and satisfies $\int fe_i=0$ for $0\leq i\leq m-1$ by construction, with $\int fe_{m} = \frac{a}{\ell^2} \int e_{m} g_{m} = a$. Also, we have for any polynomial $P\in G_{r_0,r_1}$,
\begin{align*}
\|e_{m}-P\|^2_{L_2([r_0,r_1])} &= \int_{r_0}^{r_1} \vert r^{m+d-1}-P(r)\vert ^2\dd r=\int_{1}^{\frac{r_1}{r_0}} r_0\vert (r_0r)^{d+m-1}-P(rr_0)\vert ^2\dd r\\
&= r_0^{2(d+m)-1}\int_{1}^{\frac{r_1}{r_0}} \vert r^{d+m-1}-r_0^{-(d+m-1)}P(rr_0)\vert ^2\dd r.
\end{align*}
As $r\mapsto r_0^{-(d+m-1)}P(rr_0)$ is an element of $G_{1,r_1/r_0}$, letting $r_1=2r_0$, we obtain
\begin{align*}
\ell^2 &= \|g_{m}\|_{L_2([r_0,r_1])}^2 = \min_{P\in G_{r_0,r_1}}\|e_{m}-P\|^2_{L_2([r_0,r_1])} \\
&= r_0^{2(d+m)-1} \min_{P\in G_{1,2}}\|e_{m}-P\|^2_{L_2([1,2])}=Cr_0^{2(d+m)-1},
\end{align*} 
where $C=C_m>0$ is the distance between $e_m$ restricted to $[1,2]$ and $G_{1,2}$.  The function $f$ is not smooth so that it does not belong to $F_a$. To overcome this issue, we consider a smooth kernel $\rho$ on $\R$ satisfying  $\int \rho=1$ and $\int\rho(r)r^i \dd r=0$ for $i=1,\dots,m+d-1$, with support included in $\BB_{\R}(0,r_0/2)$. See e.g.~\cite[Section 3.2]{berenfeld2019density} for the construction of such a kernel $\rho$. The map $\rho*f$ is supported on $(1,\infty)$ and it is straighforward to check that $\rho*f \in F_a$ for $r_0>2$. By Young's inequality, $\|\rho*f\|_{L_2(\R)} \leq  \|\rho\|_{L_\infty(\R)} \|f\|_{L_2(\R)}$, so that
\begin{align*}
\int \vert \rho*f(r)\vert r^{d-1}\dd r &\leq \p{\int_{r_0/2}^{5r_0/2}r^{2d-2}\dd r}^{1/2}\|\rho*f\|_{L_2(\R)}\\
&\leq \p{c_dr_0^{2d-1}}^{1/2}  \|\rho\|_{L_\infty(\R)}\|f\|_{L_2(\R)}  \leq C_{d,m} a r_0^{-m}.
\end{align*}
 By letting $r_0$ goes to $\infty$, we see that $\inf\left\{\int \vert f(r)\vert r^{d-1}\dd r,\ f\in F_a \right\} = 0$.
\end{proof}

\section{Details on Section \ref{sec:conv_opt}}\label{sec:det_num}

\subsection{Optimization of convex functions on Riemannian manifolds}

Let $\MM$ be a complete Riemannian manifold endowed with a metric $g_{\MM}$. We write $\mathrm{Exp}_x$ for the exponential map at $x\in \MM$. The geodesic distance is written as $d_{\MM}$. We say that a set $\Omega\subset \MM$ is geodesically convex if every geodesic joining two points of $\Omega$ is included in $\Omega$. We will assume that $\Omega$ is small enough so that the logarithmic map $\mathrm{Exp}_x^{-1}$ is defined on $\Omega$ for every $x\in \Omega$. We say that a $\CC^2$ function $G:\Omega \to \R$ is $\lambda$-strongly geodesically convex and $\beta$-smooth if $G\circ \gamma$ is $\lambda$-strongly convex and $\beta$-smooth for every unit-speed geodesic $\gamma$ in $\Omega$. In particular, this implies 
\begin{equation}\label{eq:strong_cvx}
\begin{split}
&G(y)\geq G(x) + \dotp{\nabla G(x),\mathrm{Exp}_x^{-1}(y)} + \frac{\lambda}{2}d_{\MM}(x,y)^2 \\
&G(y)\leq G(x) + \dotp{\nabla G(x),\mathrm{Exp}_x^{-1}(y)} + \frac{\beta}{2}d_{\MM}(x,y)^2 .
\end{split}
\end{equation}

A fundamental result of convex optimization \cite{udriste2013convex} states that a $\beta$-smooth and $\lambda$-strongly geodesically convex function can be optimized efficiently through a gradient descent.

\begin{proposition}\label{prop:grad_descent}
Assume that $\MM$ has nonnegative curvature. Let $G:\Omega\to \R$ be  $\beta$-smooth and $\lambda$-strongly geodesically convex, with minimizer $x^*$. Assume that $\Omega$ contains a geodesic ball centered at $x^*$ of radius $r_0$. Fix $x^0$ a point in this geodesic ball and let $0\leq \alpha \leq \beta^{-1}$.
\begin{enumerate}
\item The sequence of iterates $x^{a+1} = \mathrm{Exp}_{x^a}(-\alpha \nabla G(x^a))$ is well-defined for any $a\in \N$.
\item The sequence of iterates satisfies
\begin{equation}\label{eq:sequence_converge}
d_{\MM}(x^a,x^*)^2 \leq (1-\lambda \alpha)^t d_{\MM}(x^0,x^*)^2.
\end{equation}
\end{enumerate} 
\end{proposition}
Such a result is standard, although we could not find it in this form in the literature. We provide a short proof here.
\begin{proof}
The fact that the sequence of iterates is well-defined follows from $\MM$ being complete, inequality \eqref{eq:sequence_converge},  and the fact that $\Omega$ contains a geodesic ball centered at $x^*$. It therefore suffices to show \eqref{eq:sequence_converge}. As the manifold $\MM$ has nonnegative curvature, we have
\begin{align*}
d_{\MM}(x^{a+1},x^*)^2 &\leq d_{\MM}(x^a,x^*)^2 + d_{\MM}(x^a,x^{a+1})^2 -2 \dotp{\mathrm{Exp}^{-1}_{x^a}(x^{a+1}),\mathrm{Exp}^{-1}_{x^a}(x^{*})} \\
&\leq d_{\MM}(x^a,x^*)^2  + \alpha^2 \vert \nabla G(x^a)\vert ^2 + 2\alpha \dotp{\nabla G(x^a), \mathrm{Exp}^{-1}_{x^a}(x^{*})} \\
&\leq  (1-\lambda \alpha)d_{\MM}(x^a,x^*)^2  + \alpha^2 \vert \nabla G(x^a)\vert ^2 + 2\alpha (G(x^*)-G(x^a)),
\end{align*}
where we used \eqref{eq:strong_cvx} at the last line. Also, we have by \eqref{eq:strong_cvx}
\begin{align*}
G(x^*)-G(x^a) &\leq G(\mathrm{Exp}_{x^a}(-\beta^{-1}\nabla G(x^a))) - G(x^a) \\
&\leq \dotp{\nabla G(x^a),-\beta^{-1}\nabla G(x^a) } + \frac{\beta}{2} \vert \nabla G(x^a)\vert ^2 \leq - \frac{\beta}{2} \vert \nabla G(x^a)\vert ^2,
\end{align*}
concluding the proof.
\end{proof}

Proposition  \ref{prop:strong_convexity} that is proven just below asserts that $G_m$ is with high probability $\beta$-smooth and $\lambda$-strongly geodesically convex with both $\beta$ and $\lambda$ of order $\eps^2$ on $\Omega=\{(Q,V)\in \MM:\ d_{\OO_*(d,D)}(Q,Q^*)\leq \delta\eps,\ \op{V}\leq \ell\}$. Our initialization point is given by $(Q^0,0)$, where $Q^0$ is the output of a PCA, that satisfies with high probability $d_{\OO_*(d,D)}(Q^0,Q^*)\leq c\eps$ for some constant $c$. The geodesic distance between $(Q^0,0)$ and $(Q^*,V^*)$ is smaller than $C\eps$ for some larger constant $C$ (using the definition of the metric \eqref{eq:def_metric}). Hence, for $\delta$ large enough, $\Omega$ contains the geodesic ball centered at $(Q^*,V^*)$ of radius $C\eps$, and we can apply Proposition \ref{prop:grad_descent}.

Letting $\alpha=\beta^{-1}$, the iterates of a gradient descent converge at rate
\begin{equation}
d_{\MM}((Q^a,V^a),(Q^*,V^*))^2 \leq c^t d_{\MM}((Q^0,V^0),(Q^*,V^*))^2,
\end{equation}
where $c\in (0,1)$ depends on the parameter of the model.

\subsection{Convexity of \texorpdfstring{$G_m$}{Gm}}
We prove in this section Proposition \ref{prop:strong_convexity}. We assume without loss of generality that $\delta\geq 1/(2\taumin)$ and that $\delta\leq \ell$.
Fix $(Q,V)\in \Omega$ and let $(B,W)\in T_{(Q,V)}\MM$ be a tangent vector with unit norm. Write $U$ for the vector space spanned by the first $d$ columns of $Q$.  
 The exponential map $\mathrm{Exp}_{(Q,V)}$ on $\MM$ is given by 
\[\mathrm{Exp}_{(Q,V)}(B,W) = \p{ Q\exp \begin{pmatrix}
0 & -B^\top  \\
B & 0
\end{pmatrix}, V+W}.\]
Introduce the function $F_x:t\mapsto G_{m,x}(\mathrm{Exp}_{(Q,V)}(tB,tW))$. We denote by $\E_N$ the expectation with respect to the empirical distribution associated with $X_1,\dots,X_N$, so that $\E_N F_X = \frac{1}{N}\sum_{i=1}^N F_{X_i}$. To show that $G_m$ is geodesically $\lambda$-strongly convex and $\beta$-smooth on $\Omega$, we need to show that 
\[\lambda \leq \frac{d^2}{dt^2} \E_N F_X(t)_{\vert t=0} \leq \beta. \]
To simplify the notation, write $(Q^t,V^t)= \mathrm{Exp}_{(Q,V)}(tB,tW)$, and let $Q^t= (e_1^t \cdots e_D^t)$. We will also write $\dot a$ (resp. $\ddot a$) for the first (resp. second) time derivative of a function $a$ evaluated at $0$. Let $\V_j^t(x) =\iota_j(Q^t,V^t)[x^{\otimes j}]$ and let $\V = \sum_{j=2}^{m-1}\V_j$. Remark that
\begin{equation}
F_x = \frac{1}{2}\p{\vert x\vert ^2 -\sum_{k=1}^d \dotp{e_k,x}^2 + \vert \V(x)\vert ^2 - 2\dotp{x,\V(x)}}.
\end{equation}
One can directly compute
\begin{equation}\label{eq:def_Fx}
\begin{split}
&\ddot F_x = -\sum_{k=1}^d \p{\dotp{\ddot e_k,x}\dotp{e_k,x} + \dotp{\dot e_k,x}^2}  + \dotp{\ddot \V(x) ,\V(x)-x} +\vert \dot \V(x)\vert ^2 \\
& \E_N \ddot F_X  = \frac{1}{N}\sum_{i=1}^N \ddot F_{X_i}.
\end{split}
\end{equation}
Also, we have 
\begin{equation}\label{eq:expression_derivative}
\begin{split}
&\dot Q= (\dot e_1 \cdots \dot e_D) = Q\begin{pmatrix}
0 & -B^\top  \\
B & 0
\end{pmatrix} = (Q_{[d,D]}B\ \vert -Q_{[d]}B^\top) \\
&\ddot Q= (\ddot e_1 \cdots \ddot e_D)  =Q\begin{pmatrix}
0 & -B^\top  \\
B & 0
\end{pmatrix}^2 = -(Q_{[d]}B^\top B \ \vert\ Q_{[d,D]}B B^\top).
\end{split}
\end{equation}
Note that \eqref{eq:expression_derivative} yields the following identities: for $1\leq k,l\leq D$,
\begin{equation}\label{eq:identity_ek}
\begin{cases}
\dotp{\dot e_k, e_l} = -\dotp{e_k,\dot e_l} \text{ and } \dotp{\ddot e_k,e_l} = \dotp{e_k,\ddot e_l},\\
\text{for $k\leq d$, } \dot e_k \in U^\bot \text{  and }\ddot e_k \in U,\\
\text{for $k>d$, } \dot e_k \in U\text{  and }\ddot e_k \in U^\bot.
\end{cases}
\end{equation}
We let $\tilde x = Q_{[d]}^\top x\in \R^d$. Also, we insist on the distinction between the tensor $\V_j$ (that is a tensor from $\R^D$ to $\R^D$) and the tensor $V_j$ (that is a tensor from $\R^d$ to $\R^{D-d}$). The two are related by the identity $\V_j(x) = Q_{[d,D]}V_j[\tilde x^{\otimes j}]$. We will also write $\W_j$ for the tensor given by $\W_j=\iota_j(Q,W_j)$ and let $\W= \sum_{j=2}^{m-1}\W_j$. We write $B=u\Sigma v^\top$ for the SVD of $B$, with $u$ (resp.~$v$) a $(D-d)\times d$ (resp.~$d\times d$) matrix with orthogonal columns $u_k$ (resp.~$v_k$) and $\Sigma$ a $d\times d$ diagonal matrix with nonnegative entries $\sigma_1,\dots,\sigma_d$. In particular, we have $\vert B \vert^2 = \sum_{k=1}^d \sigma_k^2$. We will use the following fact.

\begin{lemma}\label{lem:one_for_all}
Let $a=(a_{d+1},\dots,a_{D})\in\R^{D-d}$. Then,
\begin{equation}
\vert \sum_{k=d+1}^D \dot e_k a_k \vert \leq \vert a \vert \vert B \vert.
\end{equation}
\end{lemma}
\begin{proof}
We have $\sum_{k=d+1}^D \dot e_k a_k =\dot Q_{[d,D]} a = Q_{[d]}v\Sigma u^\top a$. As $Q_{[d]}$ and $v$ have orthogonal columns, the squared norm of this vector is equal to the squared norm of $\Sigma u^\top a$, that is equal to 
\[\sum_{k=d+1}^D \sigma_k^2 \dotp{u_k,a}^2 \leq  \vert a \vert^2 \vert B \vert^2,\] as each $u_k$ is of norm $1$.
\end{proof}

\subsubsection*{Step 1} We first give bounds on $-\sum_{k=1}^d \p{\dotp{\ddot e_k,x}\dotp{e_k,x} + \dotp{\dot e_k,x}^2}$. First, the dot product $\vert \dotp{\dot e_k,x}\vert $ is negligible.

\begin{lemma}\label{lem:dot_ek_is_perpendicular}
For $1\leq k \leq d$, we have $\vert \dotp{\dot e_k,x}\vert  \leq 2\vert \dot e_k\vert \eps^2\delta$.
\end{lemma}

\begin{proof}
Let $x\in M$ with $\vert x\vert \leq \eps$. Recall that $U$ is the vector space spanned by $Q_{[d]}$. It holds that
\begin{align*}
\vert \dotp{\dot e_k,x}\vert  &= \vert \dotp{\dot e_k,\pi_U^\bot(x)}\vert  \leq \vert \dot e_k\vert  (\vert (\pi_U^\bot-\pi_{T_0 M}^\bot)(x)\vert  + \vert \pi_{T_0 M}^\bot(x)\vert )\\
&\leq \vert \dot e_k\vert (r\eps + \eps^2/(2\taumin)).
\end{align*}
The fact that $\delta\geq 1/(2\taumin)$ and that $r =\delta \eps$ gives the conclusion.
\end{proof}
Lemma \ref{lem:dot_ek_is_perpendicular} implies that $\sum_{k=1}^d \dotp{\dot e_k,x}^2 \leq 4\delta^2\vert B\vert ^2 \eps^4\leq 4\delta^2 \eps^4$. Also, we have
\begin{align*}
-\sum_{k=1}^d \dotp{\ddot e_k,x}\dotp{e_k,x} &= -\sum_{k=1}^d x^\top  \ddot e_k e_k^\top  x = -x^\top  \ddot Q_{[d]}Q_{[d]}^\top  x \\
&= x^\top  Q_{[d]}B^\top B Q_{[d]}^\top x =\vert B  \tilde x\vert ^2.
\end{align*}
Therefore, we may lower bound the first term in the expression of $\ddot F$:
\begin{equation}\label{eq:low_step1}
\begin{split}
&-\E_N \sum_{k=1}^d \p{\dotp{\ddot e_k,X}\dotp{e_k,X} + \dotp{\dot e_k,X}^2}\geq \E_N \vert B  \tilde X\vert ^2 -4\delta^2\eps^4 
\end{split}
\end{equation}
 Also, as $\dotp{e_k,X_i}\leq \eps$ and as $(B,W)$ is of norm $1$, we have the upper bound
\begin{equation}\label{eq:up_step1}
-\E_N\sum_{k=1}^d \p{\dotp{\ddot e_k,X}\dotp{e_k,X} + \dotp{\dot e_k,X}^2} \leq c_d\eps^2\vert B\vert ^2 \leq c_d\eps^2.
\end{equation}

\subsubsection*{Step 2} One can compute
\begin{align*}
\dot \V &= \sum_{j=2}^{m-1}\sum_{k=d+1}^D \Bigg( \dot e_k \sum_{1\leq i_1\leq \dots \leq i_j\leq d} V_{j,k}^{i_1,\dots ,i_j} e_{i_1}\otimes \cdots \otimes e_{i_j} \\*
&\qquad + e_k \sum_{1\leq i_1\leq \dots \leq i_j\leq d} W_{j,k}^{i_1,\dots ,i_j} e_{i_1}\otimes \cdots \otimes e_{i_j}\\*
&\qquad + e_k \sum_{1\leq i_1\leq \dots \leq i_j\leq d} V_{j,k}^{i_1,\dots ,i_j} \sum_{a=1}^j e_{i_1}\otimes \cdots \dot e_{i_a}\cdots \otimes e_{i_j} \Bigg)
\end{align*}
Let us lower bound $\vert \dot \V(x)\vert ^2$. As $\dot e_k \in U$ and $e_k\in U^\bot$ for $d+1\leq k \leq D$,
\begin{align*}
\vert \dot \V(x)\vert ^2 &\geq \sum_{k=d+1}^D \Bigg(\sum_{j=2}^{m-1} \sum_{1\leq i_1\leq \dots \leq i_j\leq d} W_{j,k}^{i_1,\dots ,i_j} \prod_{c=1}^j \dotp{e_{i_c},x} \\
&\qquad +\sum_{1\leq i_1\leq \dots \leq i_j\leq d} V_{j,k}^{i_1,\dots ,i_j} \sum_{a=1}^j  \dotp{\dot e_{i_a},x}\prod_{c\neq a} \dotp{e_{i_c},x} \Bigg)^2 \\
&= \sum_{k=d+1}^D (A_{1,k}+A_{2,k})^2.
\end{align*} 
We lower bound $(A_{1,k}+A_{2,k})^2$ by $A_{1,k}^2 - 2\vert A_{1,k}\vert \vert A_{2,k}\vert $. Notice first that 
\begin{equation}\label{eq:e_dot_x}
\dotp{\dot e_{i_a},x} = \sum_{f=d+1}^D \dotp{\dot e_{i_a},e_f}\dotp{e_f,x} =  -\sum_{f=d+1}^D \dotp{e_{i_a},\dot e_f}\dotp{e_f,x} = \dotp{e_{i_a}, z},
\end{equation}  
where $z = -\sum_{f=d+1}^D \dot e_f \dotp{e_f,x} = -\dot Q_{[d,D]} Q_{[d,D]}^\top x = Q_{[d]}B^\top Q_{[d,D]}^\top x$. We have $\vert z\vert^2 = \sum_{k=1}^d \sigma_k^2 \vert \tilde y_k\vert^2$, where $\tilde y_k$ is the $k$th entry of the vector $\t y = u^\top Q_{[d,D]}^\top x\in \R^d$, that is equal to $\dotp{x,Q_{[d,D]}u_k}$. As $Q_{[d,D]}u_k\in U^\bot$ and $u_k$ is of unit norm, we have $\vert \t y_k \vert \leq c \delta \eps^2$ by the same argument than in Lemma \ref{lem:dot_ek_is_perpendicular}. Therefore, $\vert z \vert \leq c \delta \vert B \vert\eps^2$. Write $\t z=Q_{[d]}^\top z\in \R^d$. This implies
\begin{equation}\label{eq:bound_A2k}
\begin{split}
\p{\sum_{k=d+1}^D \vert A_{2,k}\vert ^2}^{1/2} &= \p{\sum_{k=d+1}^D \p{\sum_{j=2}^{m-1}\sum_{1\leq i_1\leq \dots \leq i_j\leq d} V_{j,k}^{i_1,\dots ,i_j} \sum_{a=1}^j  \dotp{e_{i_a},z}\prod_{c\neq a} \dotp{e_{i_c},x}^2}}^{1/2} \\
&= \p{\sum_{k=d+1}^D \sum_{j=2}^{m-1}V_{j,k}[\t z, \t x^{\otimes (j-1)}]^2}^{1/2} \leq \sum_{j=2}^{m-1} \vert V_{j}[\t z,\t x^{\otimes (j-1)}]\vert  \\
&\leq \ell\sum_{j=2}^{m-1} \eps^{j-1}  \vert z\vert  \leq C_{d,m}\delta(\ell \eps)  \eps^{2} \vert B\vert.
\end{split}
\end{equation}

Also, we have
\[\vert A_{1,k}\vert \leq \sum_{j=2}^{m-1}\eps^j \sum_{1\leq i_1\leq \dots \leq i_j\leq d} \vert W_{j,k}^{i_1,\dots ,i_j}\vert  \leq C_{d,m}\sum_{j=2}^{m-1}\eps^j \vert W_{j,k}\vert ,\]
so that
\begin{align*}
&\sum_{k=d+1}^D \vert A_{1,k}\vert \vert A_{2,k}\vert  \leq C_{d,m}\p{\sum_{k=d+1}^D \p{\sum_{j=2}^{m-1}\eps^j \vert W_{j,k}\vert}^2}^{1/2} \p{\sum_{k=d+1}^D \vert A_{2,k}\vert ^2}^{1/2} \\
&\leq C'_{d,m} \sum_{j=2}^{m-1}\eps^j \vert W_{j}\vert \delta(\ell \eps)\eps^2  \vert B\vert \leq C''_{d,m}\delta(\ell \eps)(\eps^2 \vert B\vert^2 +  \sum_{j=2}^{m-1}\eps^{2j+2} \vert W_{j}\vert^2) \\
&\leq C''_{d,m}\delta(\ell \eps) \eps^2,
\end{align*}
where we used at the last line that $\vert B\vert ^2 + \sum_{j=2}^{m-1}\eps^{2(j-1)}\vert W_j\vert ^2$ is the norm of the vector $(B,W)$, that we assume is equal to $1$. As $\sum_{k=d+1}^D A_{1,k}^2 = \vert \W(x)\vert^2$, we obtain that
\begin{equation}\label{eq:low_step2}
\E_N \vert \dot \V(X)\vert ^2 \geq \E_N \vert \W(X)\vert ^2 -  C_{d,m}\delta(\ell \eps)\eps^2.
\end{equation}
Let us now upper bound $\vert \dot \V(x)\vert$. We have $\vert\sum_{k=d+1}^D  \dot e_k V_{j,k}[\t x^{\otimes j}]\vert^2  = \vert\dot Q_{[d,D]} V_j[\t x^{\otimes j}]\vert^2$ (where $V_j[\t x^{\otimes j}]$ is the vector in $\R^{D-d}$ with entries $V_{j,k}[\t x^{\otimes j}]$). Therefore,
\begin{align*}
&\vert\sum_{k=d+1}^D  \dot e_k V_{j,k}[\t x^{\otimes j}]\vert^2 = \vert Q_{[d]}B^\top V_j[\t x^{\otimes j}]\vert^2 = \vert Q_{[d]}v\Sigma u^\top V_j[\t x^{\otimes j}]\vert^2 \\
&= \sum_{l=1}^d \sigma_l^2 \dotp{u_l,V_{j}[\t x^{\otimes j}]}^2 \leq \vert B \vert^2 \ell^2 \eps^{2j} \leq \vert B \vert^2 \ell^2 \eps^{4},
\end{align*} 
where we used that each $u_l$ is of norm $1$. We therefore obtain the upper bound (recalling that $\ell \leq c \eps^{-1}$ for a certain constant $c$)
\begin{equation}\label{eq:up_step2}
\begin{split}
&\E_N\vert \dot \V(X)\vert ^2 \leq C_{d,m} \sum_{j=2}^{m-1}\E_N \vert \sum_{k=d+1}^D  \dot e_k V_{j,k}[\t X^{\otimes j}]\vert ^2  +  2\sum_{k=d+1}^D( A_{1,k}^2 + A_{2,k}^2)\\
&\leq C'_{d,m} ( \vert B \vert^2 \ell^2 \eps^{4} + \sum_{j=2}^{m-1}\eps^{2j}\vert W_j\vert^2 + \delta^2 (\ell \eps)^2 \eps^4 \vert B \vert^2) \leq C''_{d,m}\eps^2,
\end{split}
\end{equation}
where we used that $(B,W)$ is of norm $1$.

\subsubsection*{Step 3} Eventually, we upper bound $\vert \ddot \V(x)\vert $. We first compute
\begin{align*}
\ddot \V &= \sum_{j=2}^{m-1}\sum_{k=d+1}^D \Bigg( \ddot e_k \sum_{1\leq i_1\leq \dots \leq i_j\leq d} V_{j,k}^{i_1,\dots ,i_j} e_{i_1}\otimes \cdots \otimes e_{i_j} \\
& + 2\dot e_k \sum_{1\leq i_1\leq \dots \leq i_j\leq d} W_{j,k}^{i_1,\dots ,i_j} e_{i_1}\otimes \cdots \otimes e_{i_j} \\
& + 2 \dot e_k \sum_{1\leq i_1\leq \dots \leq i_j\leq d} V_{j,k}^{i_1,\dots ,i_j} \sum_{a=1}^j e_{i_1}\otimes \cdots \dot e_{i_a}\cdots \otimes e_{i_j} \\
& + 2e_k \sum_{1\leq i_1\leq \dots \leq i_j\leq d} W_{j,k}^{i_1,\dots ,i_j} \sum_{a=1}^j e_{i_1}\otimes \cdots \dot e_{i_a}\cdots \otimes e_{i_j} \\
& + 2e_k \sum_{1\leq i_1\leq \dots \leq i_j\leq d} V_{j,k}^{i_1,\dots ,i_j} \sum_{a=1}^j \sum_{b> a} e_{i_1}\otimes \cdots \dot e_{i_a} \otimes \dot e_{i_b} \cdots \otimes e_{i_j}  \\
&+ e_k \sum_{1\leq i_1\leq \dots \leq i_j\leq d} V_{j,k}^{i_1,\dots ,i_j} \sum_{a=1}^j e_{i_1}\otimes \cdots \ddot e_{i_a}\cdots \otimes e_{i_j} \Bigg)\\
&= A_3 + A_4+A_5+A_6 + A_7+A_8.
\end{align*}
\begin{itemize}
\item Bound on $A_3$. We have \[\vert A_3(x)\vert  \leq \sum_{j=2}^{m-1}\p{\sum_{k=d+1}^D \vert \ddot e_k\vert ^2}^{1/2}\vert \V_j(x)\vert  \leq C_{d,m}\ell \vert B\vert ^2 \eps^2\leq  C_{d,m}\ell \eps^2.\]
\item Bound on $A_4$. By Lemma \ref{lem:one_for_all} applied to $a= W_j[\t x^{\otimes j}]\in \R^{D-d}$, we have 
\begin{align*}
\vert A_4(x) \vert &\leq 2 \vert B \vert \sum_{j=2}^{m-1}\p{\sum_{k=d+1}^D  W_{j,k}[\t x^{\otimes j}]^2}^{1/2}  \leq C_{d,m}\sum_{j=2}^{m-1}\eps^j \vert B \vert \vert W_j \vert\\
&\leq C'_{d,m} \eps(\vert B \vert^2 + \sum_{j=2}^{m-1} \eps^{2(j-1)}\vert W_j\vert^2 ) \leq C'_{d,m} \eps.
\end{align*}
Note also that \[A_4(x) =  \sum_{j=2}^{m-1}2\dot Q_{[d,D]}W_j[\t x^{\otimes j}]= -2 Q_{[d]}B^\top \sum_{j=2}^{m-1} W_j[\t x^{\otimes j}].\]
\item Bound on $A_5$. By Lemma \ref{lem:one_for_all} applied to $a= (A_{2,1},\dots,A_{2,D-d})\in \R^{D-d}$ (where the $A_{2,k}$s were introduced in Step 2) and by \eqref{eq:bound_A2k}, we have
\[ \vert A_5(x) \vert \leq 2 \vert B \vert \p{\sum_{k=d+1}^D  A_{2,k}^2}^{1/2}\leq  C_{d,m}\delta(\ell \eps)\eps^2 \vert B\vert^2\leq C_{d,m}\delta(\ell \eps)\eps^2.\]
\item Bound on $A_6$. The quantity $\vert A_6(x)\vert$ is smaller than
\begin{align*}
&2\p{\sum_{j=2}^{m-1} \sum_{k=d+1}^D \p{ \sum_{1\leq i_1\leq \dots \leq i_j\leq d} W_{j,k}^{i_1,\dots ,i_j} \sum_{a=1}^j \dotp{\dot e_{i_a},x}\prod_{c\neq a} \dotp{e_{i_c},x}}^2 }^{1/2}\\
&\leq 2\p{ \sum_{j=2}^{m-1}\sum_{k=d+1}^D \vert W_{j,k}\vert ^2\p{ \sum_{1\leq i_1\leq \dots \leq i_j\leq d} \p{\sum_{a=1}^j \dotp{\dot e_{i_a},x}\prod_{c\neq a} \dotp{e_{i_c},x}}^2 }}^{1/2} \\
&\leq C_{d,m}\sum_{j=2}^{m-1}\eps^{j-1}\p{ \sum_{k=d+1}^D \vert W_{j,k}\vert ^2 \sum_{l=1}^d\dotp{\dot e_{l},x}^2}^{1/2} \\
&\leq C'_{d,m}\sum_{j=2}^{m-1}\eps^{j-1}\p{ \sum_{k=d+1}^D \vert W_{j,k}\vert ^2 \delta^2\eps^4\sum_{l=1}^d\vert \dot e_{l}\vert ^2}^{1/2} \text{ using Lemma \ref{lem:dot_ek_is_perpendicular}}\\
&\leq C'_{d,m}\delta\sum_{j=2}^{m-1}\eps^{j+1}\vert B\vert \vert W_j\vert \leq C'_{d,m}\delta \eps^2(\vert B\vert^2 +\sum_{j=2}^{m-1} \eps^{2(j-1)}\vert W_j\vert^2) \\
&\leq C'_{d,m}\delta\eps^2.
\end{align*}
\item Bound on $A_7$. Using \eqref{eq:e_dot_x}, we obtain that $\vert A_7(x)\vert$ is smaller than
\begin{align*}
&2\left\vert \sum_{j=2}^{m-1}  \sum_{k=d+1}^D e_k \sum_{1\leq i_1\leq \dots \leq i_j\leq d} V_{j,k}^{i_1,\dots ,i_j} \sum_{a=1}^j \sum_{b> a} \dotp{ e_{i_a},z}\dotp{e_{i_b},z}\prod_{c\neq a,b} \dotp{e_{i_c},x} \right \vert  \\
&\leq C_{d,m} \vert \sum_{j=2}^{m-1}  V_j[\t z,\t z,\t x^{\otimes (j-2)}] \leq C_{d,m}\ell \vert z\vert ^2 \leq C'_{d,m}\delta^2 \ell \eps^{4}\vert B\vert ^2 \leq C'_{d,m}\delta^2 \ell \eps^{4}.
\end{align*}
\item Bound on $A_8$. We have 
\[ \dotp{\ddot e_{i_a},x} = \sum_{f=1}^d \dotp{\ddot e_{i_a},e_f}\dotp{e_f,x} =  \sum_{f=1}^d \dotp{e_{i_a},\ddot e_f}\dotp{e_f,x} = \dotp{e_{i_a}, y},\]
where $y= \sum_{f=1}^d \dotp{e_f,x}\ddot e_f$. In particular, $\vert y\vert \leq \eps \sum_{f=1}^d \vert \ddot e_f\vert \leq c_d \eps \vert B\vert^2$. Therefore, letting $\t y =Q_{[d]}^\top y$,
\begin{align*}
&\vert A_8(x)\vert = \left\vert \sum_{j=2}^{m-1}  \sum_{k=d+1}^D e_k \sum_{1\leq i_1\leq \dots \leq i_j\leq d} V_{j,k}^{i_1,\dots ,i_j} \sum_{a=1}^j \dotp{y, e_{i_a}}\prod_{c\neq a} \dotp{e_{i_c},x}\right \vert   \\
&\leq C_{d,m} \vert \sum_{j=2}^{m-1} V_j[\t y,\t x^{\otimes (j-1)}]\vert \leq C_{d,m} \ell \eps^2 \vert B\vert ^2\leq C_{d,m} \ell \eps^2.
\end{align*}
\end{itemize}
Putting the different terms together, and recalling that $\ell \geq \delta$, we obtain that 
\[\ddot \V(x) =-2 Q_{[d]}B^\top \sum_{j=2}^{m-1}W_j[\t x^{\otimes j}] + R,\]
 where $R$ is a remainder term of norm smaller than $C_{d,m}\ell \eps^2$. Also, we have $\vert \dotp{A_4(x),\V(x)}\vert\leq  C_{d,m}(\ell\eps)\eps^2$. We may therefore write 
\begin{equation}\label{eq:up_step4}
\begin{split}
\dotp{\ddot \V(x), \V(x)-x} &= 2 x^\top Q_{[d]}B^\top \sum_{j=2}^{m-1} W_j[\t x^{\otimes j}] + R'\\
&=2 (B\tilde x)^\top \sum_{j=2}^{m-1} W_j[\t x^{\otimes j}] + R',
\end{split}
\end{equation}
where $R'$ has norm smaller than $C'_{d,m}(\ell \eps)\eps^2$.

\subsubsection*{Step 4} Putting the lower bounds \eqref{eq:low_step1} and \eqref{eq:low_step2} together with identity \eqref{eq:up_step4},  we obtain the lowerbound
\begin{align}
\E_N\ddot F_X &\geq \E_N \vert B \tilde X \vert^2 + \E_N \vert \sum_{j=2}^{m-1}W_j[\tilde X^{\otimes j}]\vert ^2 + 2\sum_{j=2}^{m-1}\E_N (B\tilde X)^\top W_j[\tilde X^{\otimes j}]   -C_{d,m}(\ell \eps)\eps^2 \nonumber\\
&\geq \E_N\vert B\tilde X +  \sum_{j=2}^{m-1}W_j[\tilde X^{\otimes j}]\vert^2 -C_{d,m}(\ell \eps)\eps^2.\label{eq:the_last_eq}
\end{align}
Let us now lower bound the quantity $\E\vert B\tilde X +  \sum_{j=2}^{m-1}W_j[\tilde X^{\otimes j}]\vert^2$, where we take the expectation with respect to the density $f$ of the sample $X_1,\dots,X_N$. Letting $Y= \tilde X/\eps$ and $Z_j = \eps^{j-1}W_j$, we have
\begin{align*}
\E\vert B\tilde X +  \sum_{j=2}^{m-1}W_j[\tilde X^{\otimes j}]\vert^2 &= \eps^2 \E \vert BY + \sum_{j=2}^{m-1} Z_j[Y^{\otimes j}]\vert^2,
\end{align*}
where $\vert B\vert^2 + \sum_{j=2}^{m-1}\vert Z_j\vert^2=1$. We may decompose this expectation as
\begin{align*}
\eps^2\sum_{k=d+1}^D \E ( B_k^\top Y + \sum_{j=2}^{m-1} Z_{j,k}[Y^{\otimes j}])^2.
\end{align*}
We  show in the next lemma that each term in the sum is larger than $c_{d,m}\fmin (\vert B_k\vert^2 + \sum_{j=2}^{m-1}\vert Z_{j,k}\vert^2)$. By summing over $k$, we obtain that the expectation is larger than $c_{d,m}\fmin \eps^2$.

\begin{lemma}
Let $S_j$ be a $j$-tensor from $\R^d$ to $\R$ for each $j=1,\dots,m-1$. Then,
\begin{equation}
\E\left[ \p{\sum_{j=1}^{m-1} S_j[Y^{\otimes j}]}^2\right] \geq \fmin c_{d,m} \sum_{j=1}^{m-1}\vert S_j\vert^2.
\end{equation}
\end{lemma}
\begin{proof}
The random variable $\tilde X$ has entries $\dotp{e_k,X}$ for $1\leq k \leq d$. As $(e_1,\dots,e_d)$ is an orthonormal basis of $U$ that is $\delta\eps$-close from $T_0 M$, the random variable $\tilde X$ has a density lower bounded by $\fmin/2$ on its support, and this support contains $\BB(0,\eps/2)$. Therefore, the expectation with respect to $Y$ is larger than
\[ \fmin\int_{\BB(0,1/2)} \p{\sum_{j=1}^{m-1} S_j[y^{\otimes j}]}^2 \dd y.\]
We may also write $\sum_{j=1}^{m-1} S_j[y^{\otimes j}]$ as the dot product $\dotp{\mathbf{S},\mathbf{y}}$, where $\mathbf{S}$ and $\mathbf{y}$ are vectors indexed by $\sigma\in \bigcup_{j=1}^{m-1} \{1,\dots,d\}^j$, with the entries corresponding to $\sigma=(i_1,\dots,i_j)$ given by $\mathbf{S}_\sigma= S_j^{i_1,\dots,i_j}$ and $\mathbf{y}_\sigma= \prod_{a=1}^j y_{i_a}$. Therefore, this integral is exactly equal to $\mathbf{S}^\top \mathbf{C} \mathbf{S}$, where $\mathbf{C}$ is the matrix with entries $\mathbf{C}_{\sigma,\sigma'} = \int_{\BB(0,1/2)} \mathbf{y}_\sigma \mathbf{y}_{\sigma'}\dd y$. To conclude, we need to show that this matrix is positive definite. This follows from $\mathbf{C}$ being a Gram matrix for the $L_2$ dot product on $\BB(0,1/2)$ associated with the $L_2$ functions $y\mapsto \mathbf{y}_\sigma$ that are linearly independent. Therefore, we have $\mathbf{S}^\top \mathbf{C} \mathbf{S} \geq c_{d,m} \vert \mathbf{S}\vert^2 = c_{d,m}\sum_{j=1}^{m-1}\vert S_j\vert^2$.
\end{proof}

Eventually, by Hoeffding's inequality, with probability at least $1-\exp(-cN)$, the empirical expectation $\E_N\vert B\tilde X +  \sum_{j=2}^{m-1}W_j[\tilde X^{\otimes j}]\vert^2$ is larger than $c_{d,m}\fmin \eps^2/2$. From \eqref{eq:the_last_eq}, we obtain a lower bound of order $\eps^2$ by choosing $\ell \leq c\eps$ for $c$ small enough.

The upper bound, also of order $\eps^2$, is obtained by gathering the different upper bounds \eqref{eq:up_step1}, \eqref{eq:up_step2} obtained in Steps 1 and 2 as well as the identity \eqref{eq:up_step4}.

\bibliographystyle{sn-mathphys}
\bibliography{biblio.bib}

\end{document}